\documentclass[10pt,a4paper]{article}
\usepackage[utf8]{inputenc}
\usepackage[T1]{fontenc}
\usepackage[english]{babel}
\usepackage{authblk}
\usepackage[shortlabels]{enumitem}
\usepackage[top=2.5cm, bottom=2.5cm, left=2.5cm, right=2.5cm]{geometry}

\usepackage{amssymb,amsthm}
\usepackage{latexsym,amsmath,subfigure,tikz,siunitx}

\usepackage{latexsym}
\usepackage{subfigure,latexsym,amsthm}
\usepackage{siunitx}
\usepackage[babel]{csquotes}
\usepackage{amsmath}
\usepackage{forest}
\usepackage{cancel}
\usepackage{url}
\usepackage{algorithm,algpseudocode}
\usepackage[extra]{tipa}
\usepackage{pstricks}
\usepackage{pst-all}
\usepackage{pstricks-add}
\usepackage{pst-grad} 
\usepackage{pst-plot} 
\usepackage[space]{grffile} 
\usepackage{etoolbox} 

\usepackage{siunitx}
\usepackage[shortlabels]{enumitem}

\newcommand{\pa}{\partial}
\newcommand{\ep}{\epsilon}
\newcommand{\eps}{\varepsilon}

\newcommand{\lam}{\lambda}

\newcommand{\omg}{\omega}
\newcommand{\Omg}{\Omega}

\newcommand{\alp}{\alpha}
\newcommand{\lap}{\Delta}

\newcommand{\bke}[1]{\left( #1 \right)}
\newcommand{\bkc}[1]{\left< #1 \right>}

\newcommand{\abs}[1]{\left| #1 \right|}
\newcommand{\sinc}{\operatorname{sinc}}

\newcommand{\disp}{\displaystyle}

\newcommand{\xh}{\hat{x}}

\newcommand{\rmi}{\mathrm{i}}
\newcommand{\rme}{\mathrm{e}}
\newcommand{\rmd}{\mathrm{d}}

\newcommand{\inc}{\textrm{inc}}
\newcommand{\scat}{\textrm{scat}}

\newcommand{\calI}{\mathcal{I}}

\newcommand{\ds}{\displaystyle}

\newcommand{\eqnref}{\eqref}
\newcommand{\RR}{\mathbb{R}}

\DeclareMathOperator{\BesselJ}{J}

\DeclareMathOperator{\Struve}{S}

\makeatletter
\newcommand{\doublewidetilde}[1]{{%
		\mathpalette\double@widetilde{#1}%
}}
\newcommand{\double@widetilde}[2]{%
	\sbox\z@{$\m@th#1\widetilde{#2}$}%
	\ht\z@=.9\ht\z@
	\widetilde{\box\z@}%
}
\makeatother

\theoremstyle{plain}
\newtheorem{theorem}{Theorem}[section]

\newtheorem*{problem*}{Problem}

\theoremstyle{remark}
\newtheorem{remark}{Remark}[section]

\title{Monostatic sampling methods in limited-aperture configuration}
\author{Sangwoo Kang} 
\author{Mikyoung Lim}  

\affil{Department of Mathematical Sciences, Korea Advanced Institute of Science and Technology, Daejeon 34141, Korea}

\date{}

\providecommand{\keywords}[1]{\textbf{Keywords:} #1 }

\begin{document}
\maketitle
\begin{abstract}
		We present monostatic sampling methods for limited-aperture scattering problems in two dimensions. 
	The direct sampling method (DSM) is well known to provide a robust, stable, and fast numerical scheme for imaging inhomogeneities from multistatic measurements even with only one or two incident fields. 
	However, in practical applications, monostatic measurements in limited-aperture configuration are frequently encountered. 
	A monostatic sampling method (MSM) was studied in full-aperture configuration in recent literature.
	In this  paper, we develop MSM in limited-aperture configuration and derive an asymptotic formula of the corresponding indicator function. Based on the asymptotic formula, we then analyze the imaging performance of the proposed method depending on the range of measurement directions and the geometric, material properties of inhomogeneities. 
	Furthermore, we propose a modified numerical scheme with multi-frequency measurements that improve imaging performance, especially for small anomalies. Numerical simulations are presented to validate the analytical results.
\end{abstract}

\keywords{Helmholtz equation, Direct sampling method, Monostatic imaging, Limited aperture, Multiple frequencies, Bessel function}

	\section{Introduction}

Determining geometric characteristics of unknown inhomogeneities from the measurement of a scattered field is of great interest due to its potential applications, such as in biomedical imaging and radar imaging. A variety of inverse scattering algorithms have been developed, and they can be categorized into three classes following the classification in the survey paper \cite{2Dsurvey}: iterative, decomposition, and sampling methods. Among them, sampling methods allow us to non-iteratively retrieve the support of (possibly multiconnected) targets assuming no {\it a priori} information about the targets with low computational cost.   
A sampling method tests a region of interest with its associated indicator function; the indicator function blows up if a test location is in the support of inhomogeneities.
Various sampling schemes were proposed, including the MUltiple SIgnal Classification (MUSIC), the linear sampling method (LSM), the topological derivative, the Kirchhoff migration, the orthogonality sampling method (OSM), and the direct sampling method (DSM) \cite{2Dsurvey, survey_sampling}. 
The sampling methods show promising results in multistatic, full-aperture measurement configuration, but they may show inaccurate results in the case of limited measurement data \cite{Kang_DSM_mono-static,Park,Joh_MUSIC_CL}. In practical applications such as synthetic aperture radar \cite{zhang2015ofdm} and ground penetrable radar \cite{ConstructionAndBuildingMaterials-2016}, monostatic measurements in limited-aperture configuration are frequently encountered. In this paper, we propose monostatic sampling methods in a limited-aperture configuration in two dimensions by modifying the DSM. 

The main advantage of the DSM is that, unlike other sampling methods, one can localize the inhomogeneities even with only one or two incident fields. 
Furthermore, it does not require any additional operations, such as singular value decomposition, and it is highly tolerant of the noise of measured data \cite{dsm2d_farfield,Kang_3DDSM,PARK2018648}. 
Consequently, the DSM has been applied to various imaging modalities, such as impedance tomography \cite{DSM2D_tomography}, diffusive optical tomography \cite{DSM2D_diffusive_tomography}, radar imaging \cite{bektas2016direct}, and recovering moving potentials in heat equations \cite{DSM2D_time_heat}.
We refer the reader to \cite{Kang_mfDSM_limited} for the DSM in limited-aperture configuration. 
	We also refer the reader to \cite{ESM-Bayesian-LimitedAperture} for the sampling-type method with Bayesian approach in limited-aperture configuration.
In monostatic configuration, an intuitive indicator function of the DSM was proposed in \cite{bektas2016direct} without a theoretical explanation. Later in \cite{Kang_DSM_mono-static}, this indicator function was analyzed to verify the reason for the failure of the DSM  with monostatic measurements. Then, the so-called monostatic sampling method (MSM), which successfully employs the DSM in monostatic format, was developed.

In the present paper, we consider the MSM for limited-aperture problems. First, under a small volume and well-separated assumptions of inhomogeneities, the asymptotic structure of the MSM's indicator function is identified in terms of the Bessel functions of the first kind, where the asymptotic formula reveals the dependence of the MSM on the range of measurement direction and material, geometric chracteriestics of the inhomogeneities. We discuss the proper measurement angle condition for a successful imaging performance.
Second, we suggest a multi-frequency MSM, named MMSM, where using multiple frequencies is one of the traditional approaches to improve imaging performance. Following a path of derivation similar to that of the single frequency case, the asymptotic property of the MMSM is verified by using the Struve functions as well as the Bessel functions of the first kind. Based on the asymptotic analysis of the MSM and MMSM, we compare their imaging performance. It turns out that the MMSM is an improved version of the MSM for imaging small inhomogeneities in limited-aperture configuration. Furthermore, by considering the fact that multistatic measurements of one fixed incident field and monostatic measurements have similar amounts of information, we compare the proposed methods with the classical DSM (in limited-view configuration) developed in \cite{Kang_mfDSM_limited}.

This paper is organized as follows. In Section \ref{sec:preliminary}, we briefly explain the scattering problem in two dimensions and introduce the concept of the direct sampling method. In Section \ref{sec:3}, monostatic sampling methods in limited-aperture configuration with single- and multi-frequency measurements are proposed, and the asymptotic structures of the methods are identified. We also compare the MSM and the DSM in Section \ref{sec:4}. We exhibit numerical simulations to support our theoretical results in Section \ref{sec:5}. Conclusions and perspectives are summarized in Section \ref{sec:6}.


\section{Preliminary}\label{sec:preliminary}
\subsection{Scattering Problem}\label{sec:2}
We briefly review the direct scattering problem of dielectric inhomogeneities in two dimensions. 
Let a finite number of dielectric inhomogeneities, namely $\tau_m$ ($m=1,\dots,M$), be embedded in a homogeneous background medium. 
We assume that $\tau_m$ is given by $\tau_{m}=c_{m}+\alp_m D_m$, where $c_m$, $\alpha_m$ and $D_m$ indicate the center, size, and reference shape of $\tau_m$; Figure \ref{Location} shows an example of such a circular shape.
We denote by $\eps_0$ and $\mu_0$ the electric permittivity and the magnetic permeability of the background, respectively, and set $\eps_m$ and $\mu_m$ to be the parameters of $\tau_m$. 
For the purpose of simplicity, we only consider the non-magnetic inhomogeneities, that is, $\mu_{m}=\mu_{0}$ for all $m$ (other cases can be dealt with in a similar way).
We then define the piecewise constant dielectric permittivity $0<\eps(x)<\infty$ for $x\in\mathbb{R}^{2}$ as
\begin{equation*}
	\eps(x)=\left\{
	\begin{array}{cl}
		\ds	\eps_{0}\quad&\mbox{in}~\mathbb{R}^{2}\backslash \left(\cup_{m=1}^M \overline{\tau_m}\right),\\
		\ds	\eps_{m}\quad&\mbox{in}~\tau_{m}\ (m=1,\dots,M).
	\end{array}
	\right.
\end{equation*}

Fix an angular frequency $\omg>0$. The corresponding wavenumber and wavelength in the background medium  are $k=\omg\sqrt{\eps_{0}\mu_{0}}$ and $\lambda=2\pi/k$, respectively. 
We let the incident field $u_{\inc}$ be given by the plane wave $u_{\inc}(x,\vartheta;k) = \rme^{\rmi k \vartheta\cdot x}$ with a direction vector $\vartheta\in\mathbb{S}^{1}$. Then the direct scattering problem is to find the solution $u= u_{\inc} + u_{\scat}$ of the Helmholtz equation 
\begin{equation}\label{HelmholtzEquation}
	\lap u(x,\vartheta;k)+ \omega^2 \mu_0\eps(x)\, u(x,\vartheta;k) = 0
\end{equation}
such that the scattered field $u_{\scat}$ satisfies the Sommerfeld radiation condition
\begin{equation}\label{SmmerfeldRadiationCondition}
	\lim_{|x|\to\infty} \sqrt{|x|}\left(\frac{\pa u_{\scat}(x,\vartheta;k)}{\pa|x|} - \rmi k u_{\scat}(x,\vartheta;k)\right) =0
\end{equation}
uniformly in $\hat{x}=\frac{x}{|x|}\in\mathbb{S}^1$. 
Here, $\lap$ is the Laplacian in $x$. 
It is well known that the solution to \eqref{HelmholtzEquation}--\eqref{SmmerfeldRadiationCondition} exists, and that $u_{\scat}$ admits the asymptotic behavior
\begin{equation}\label{FarFieldPattern}
	u_{\scat}(x,\vartheta;k) = \frac{\rme^{\rmi k \vartheta\cdot x}}{\sqrt{|x|}}\left( u_{\infty}(\hat{x},\vartheta;k) + O\bke{\frac{1}{|x|}}\right)
\end{equation}
with the so-called {\it far-field pattern} $u_{\infty}(\hat{x},\vartheta;k)$ defined on $\mathbb{S}^1\times\mathbb{S}^1$ (for a fixed $k$). 

We assume that the inhomogeneities are small, i.e., 
$\alpha=\max\left\{\alp_m\,:\,m=1,\dots,M\right\}\ll\frac{\lam}{2}$,
and that they are well separated, meaning that for some positive constant $l$, 
$\mbox{dist}(\tau_m ,\tau_{m'})>l$ for $m\neq m'$.
Here, dist indicates the distance between the two inhomogeneities. Then, it holds that (see \cite{Ammari1})
\begin{equation}\label{FarFieldAsymptoticFormula}
	u_{\infty}(\hat{x},\vartheta;k) = \sum_{m=1}^{M}\alp_{m}^{2}|D_m|\,\frac{\eps_{m}-\eps_{0}}{\sqrt{\eps_{0}\mu_{0}}}\,\rme^{-\rmi k\hat{x}\cdot c_{m}}\rme^{\rmi k\vartheta\cdot c_m}
	+O(\alp^{3}).
\end{equation}

We mainly treat the direct and inverse scattering problem in monostatic configuration; that is, the data is measured via a single moving transducer (see Figure \ref{Location}). Hence, we have the far-field pattern only for $\hat{x}=-\vartheta$. We simplify the far-field pattern as a one-variable function $u_{\infty}(\hat{x};k)$. It then follows from \eqref{FarFieldAsymptoticFormula} that
\begin{equation}\label{MonoFarFieldAsymptoticFormula}
	u_{\infty}(\hat{x};k) = \sum_{m=1}^{M}\alp_{m}^{2}|D_m|\,\frac{\eps_{m}-\eps_{0}}{\sqrt{\eps_{0}\mu_{0}}}\,\rme^{-2\rmi k\hat{x}\cdot c_m}+O(\alp^{3}).
\end{equation}

\begin{figure}[t!]
	\centering
	\begin{minipage}[c][4.5cm][c]{.35\linewidth}\centering
		\includegraphics[width=0.5\textwidth]{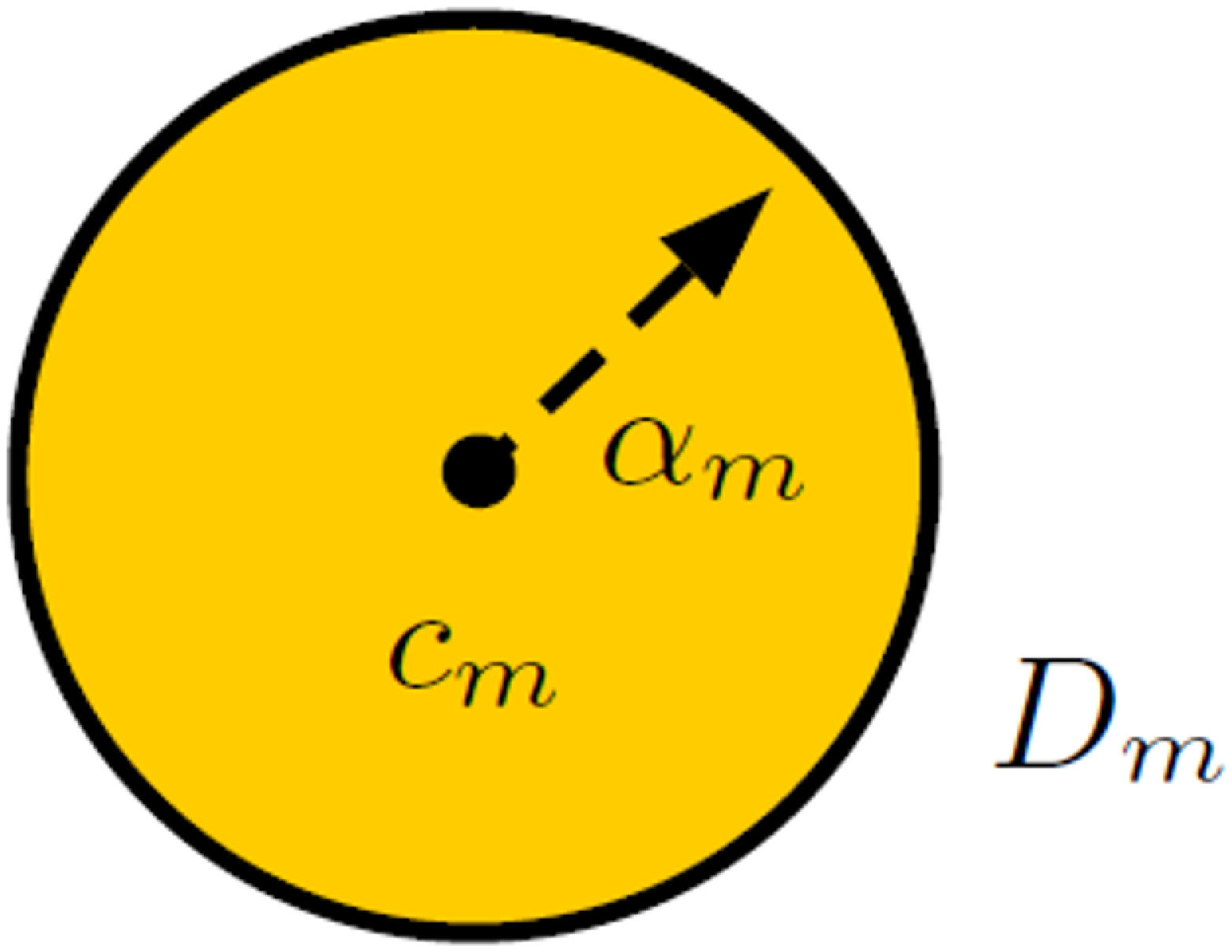}	
	\end{minipage}
	\begin{minipage}[c][4.5cm][c]{.45\linewidth}\centering
		\includegraphics[width=0.9\textwidth]{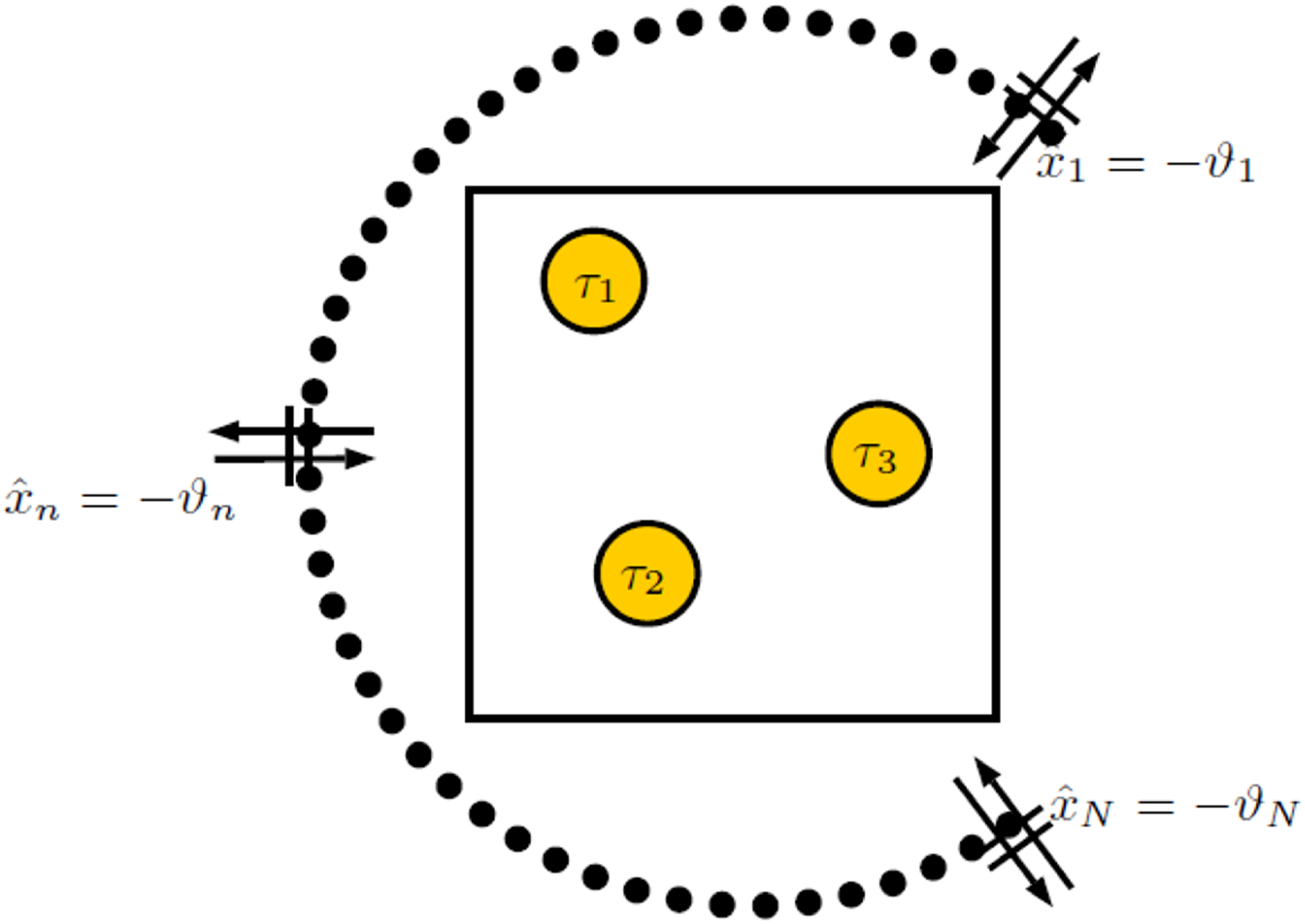}
	\end{minipage}
	\caption{Sketch of the inhomogeneity $\tau_{m}$ (left) and the scattering problem in monostatic configuration (right).}
	\label{Location}
\end{figure}

It is worth remarking that in monostatic configuration the multi-static response (MSR) matrix is diagonal (see Figure \ref{3matrices}). As a result, sampling methods based on the singular value decomposition, such as the MUSIC algorithm, linear sampling method, factorization method, and subspace migration, perform poorly in monostatic configuration. 

\begin{figure}[h!]
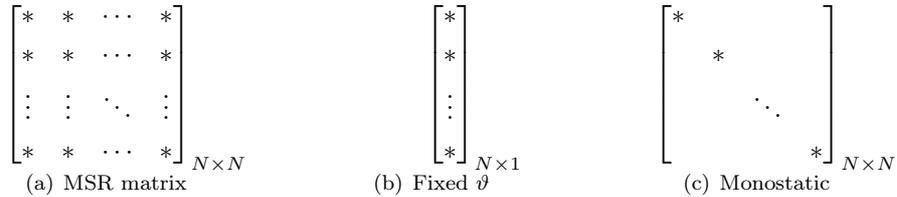

	\centering
	\subfigure[\label{MSR}MSR matrix]{
		$\ds \hskip 5mm
		\begin{bmatrix}
			\ds  * & * &\cdots &*\\[1mm]
			\ds *&* & \cdots &*\\[1mm]
			\vdots&\vdots &\ddots &\vdots \\[1mm]
			*& *& \cdots&\ds *  \end{bmatrix}_{N\times N}$
	}  \hskip 1cm
	\subfigure[\label{MSRSingle}Fixed $\vartheta$]{
		$\hskip 1cm
		\begin{bmatrix}
			*   \\[1mm]
			*  \\[1mm]
			\vdots  \\[1mm]
			* 
		\end{bmatrix}_{N\times1}
		$}
	\hskip 1cm
	\subfigure[\label{Mono}Monostatic]{
		$\ds \hskip 5mm
		\begin{bmatrix}
			\ds  * & & &\\[1mm]
			\ds &* &  &\\[1mm]
			& &\ddots & \\[1mm]
			& & &\ds *  \end{bmatrix}_{N\times N}$
	}  
	\caption{(a) shows the MSR matrix with $N$ incident waves and $N$ measurement directions; (b) is the measurement data with one incident wave; (c) indicates possible nonzero values of MSR matrix in monostatic configuration, which are diagonal entries.}
	\label{3matrices}
\end{figure}


\subsection{Direct Sampling Method}

For notational simplicity, we define a discrete $l^2(\mathbb{S}^1)$ type inner product as
\begin{equation}\label{innerproduct}
	\bkc{a(\xh_{n}),\, b(\xh_{n})}=\frac{1}{N}\sum_{n=1}^{N}a(\xh_{n})\,\overline{b(\xh_{n})}
\end{equation}
for given functions $a,b$ on $\mathbb{S}^1$ and $\hat{x}_n\in \mathbb{S}^1$. 

We first consider the classical DSM with one incident field. 
Let $\vartheta$ denote the direction of the incident field. Let the measurement data $u_\infty(\hat{x}_n,\vartheta;k),\  n=1,\dots,N,$ be given in full-aperture configuration (see Figure \ref{MSRSingle}). To highlight that the data is obtained for $\hat{x}_n$ in the full range $\mathbb{S}^1$, we denote by $\langle\cdot,\cdot\rangle_{l^2(\mathbb{S}^1)}$ for the inner product \eqref{innerproduct} with this measurement data. 
The indicator function of the classical DSM is given by (see \cite{dsm2d_farfield})
\begin{equation}\label{DSM}
	\calI_{\textrm{DSM}}(z,\vartheta;k): = \frac{\left|\bkc{u_{\infty}(\xh_{n},\vartheta;k),\, \rme^{-\rmi k\xh_n\cdot z}}_{l^2(\mathbb{S}^1)}\right|}{\disp\max_{z\in\Omega}\left|\bkc{u_{\infty}(\xh_{n},\vartheta;k),\, \rme^{-\rmi k\xh_n\cdot z}}_{l^2(\mathbb{S}^1)}\right|},
\end{equation}
where $z\in\RR^2$ is a point in a compact region $\Omega$, and $k,\vartheta$ are fixed. 
It was shown in \cite{Kang_2DDSM} by using \eqref{FarFieldAsymptoticFormula} and the property of the Bessel function that
\begin{equation}\label{DSM:asym}
	\calI_{\textrm{DSM}}(z,\vartheta;k)\approx\frac{\left|\mathfrak{F}_{\textrm{DSM}}(z,\vartheta;k)\right|}{\max_{z\in\Omega}\left|\mathfrak{F}_{\textrm{DSM}}(z,\vartheta;k)\right|}
\end{equation}
with 
\begin{equation}\label{DSMresultfull}
	\mathfrak{F}_{\textrm{DSM}}(z,\vartheta;k)=\sum_{m=1}^M \alpha_m^2 \left|D_m\right|(\eps_m-\eps_0) e^{\rmi k \vartheta\cdot c_{m}} \BesselJ_{0}\left(k|z-c_{m}|\right),
\end{equation}
where $\BesselJ_0$ denotes the Bessel function of the first kind of order zero. 
If multiple impinging waves with various propagation directions $\vartheta$ are used, then one can modify the indicator function as the maximum of $\calI_{\textrm{DSM}}(z,\vartheta;k)$ with respect to $\vartheta$ values.

\subsection{DSM in Monostatic Configuration}

In monostatic configuration, as explained previously, we assume that multiple impinging waves with various $\vartheta$ are used and that the resulting far-field pattern is obtained only for the measurement angle $-\vartheta$ (see Figure \ref{Mono}).
We define the far-field pattern in monostatic configuration as 
\begin{equation}\label{data:mono:x_n}
	u_\infty(\hat{x}_n;k):=u_\infty(\hat{x}_n,-\hat{x}_n;k), \quad n=1,\dots,N.
\end{equation}
If we use this data on the right-hand side of \eqref{DSM}, that is, $\vartheta$ is replaced by $-\hat{x}_n$, then the resulting indicator function admits the asymptotic relation similar to \eqref{DSM:asym} with (see \cite[Theorem 1]{Kang_DSM_mono-static})
\begin{equation}\label{DSMresultMono}
	{\mathfrak{F}}_{\textrm{DSM}}(z;k)=\sum_{m=1}^M \alpha_m^2 \left|D_m\right|(\eps_m-\eps_0) \BesselJ_{0}\left(k|z-2c_{m}|\right).
\end{equation}

Note that, while the basis functions of the main term of the DSM in full-aperture configuration in \eqref{DSMresultfull} are $\BesselJ_0(k|z - c_{m}|)$, those with monostatic measurement in \eqref{DSMresultMono} are $\BesselJ_0(k|z - 2c_{m}|)$. This causes the miss-localization phenomenon for the traditional DSM in monostatic configuration. 
This problem was fixed in \cite{Kang_DSM_mono-static} by developing the so-called monostatic sampling method (MSM)
\begin{equation}\label{MSM}
	\calI_{\textrm{MSM}}(z;k): = \frac{\left|\bkc{u_{\infty}(\xh_{n};k),\, \rme^{-2\rmi k\xh_n\cdot z}}_{l^2(\mathbb{S}^1)}\right|}{\disp\max_{z\in \Omg}\left|\bkc{u_{\infty}(\xh_{n};k),\, \rme^{-2\rmi k\xh_n\cdot z}}_{l^2(\mathbb{S}^1)}\right|},
\end{equation}
where this indicator function admits the asymptotic relation similar to \eqref{DSM:asym} with 
$$\mathfrak{F}_{\textrm{MSM}}(z;k)=\sum_{m=1}^M \alpha_m^2 \left|D_m\right|(\eps_m-\eps_0) \BesselJ_{0}\left(2k|z - c_{m}|\right).$$

\section{Single- and Multi-Frequency Monostatic Sampling Methods in Limited-Aperture Configuration}\label{sec:3}
In this section, we study single- and multi-frequency MSMs in limited-aperture configuration. This is one of the main subjects of this paper.
As in the previous section, we investigate sampling methods to recover the inhomogeneities $\tau_m$ in the asymptotic framework in which the inhomogeneities are small and well-separated and in which the far-field data are obtained from the monostatic measurements in limited-aperture configuration. On the other hand, the far-field patterns of scattered waves, which are in the form of \eqref{data:mono:x_n}, are now assumed to be known only for $\hat{x}_n= [\cos(\theta_n),\sin(\theta_n)]^T$ with $n=1,\dots,N$ restricted in a subset $\mathbb{S}_*^1$ of $\mathbb{S}^1$ given by ($2\pi$ is identified with $0$)
\begin{equation}\label{MeasureRange}
	\mathbb{S}^1_*=\left\{[\cos(\theta),\sin(\theta)]^T\,:\, \theta\in I\right\}\quad\mbox{with an interval }I \subsetneq [0,2\pi].
\end{equation}
We choose $\theta_n$ to form a set of equidistant points such that $I=[\theta_1,\theta_N]$. 
To highlight that the data is obtained for $\hat{x}_n$ in the restricted range $\mathbb{S}_*^1$, we denote by $\langle\cdot,\cdot\rangle_{l^2(\mathbb{S}_*^1)}$ for the inner product \eqref{innerproduct} with this measurement data and generalize the indicator function in \eqref{MSM} to the limited-aperture case as
	\begin{equation}\label{MSM2}
		\calI_{\textrm{MSM}}(z;k) = \frac{\left|\bkc{u_{\infty}(\xh_{n};k),\, \rme^{-2\rmi k\xh_n\cdot z}}_{l^2(\mathbb{S}^1_{*})}\right|}{\disp\max_{z\in \Omg}\left|\bkc{u_{\infty}(\xh_{n};k),\, \rme^{-2\rmi k\xh_n\cdot z}}_{l^2(\mathbb{S}^1_{*})}\right|}.
	\end{equation}

We denote by $\BesselJ_s(r)$ the Bessel function of the first kind of order $s$. To express the asymptotic structure of single- and multi-frequency MSMs in limited-aperture configuration, we define the following functions:
\begin{gather}\label{def:tilde:J}
	\widetilde{\BesselJ}_{s}(z;k_1,k_P):=\frac{1}{k_{P}-k_{1}}\int_{k_{1}}^{k_{P}}\BesselJ_{s}(k|z|)\, \rmd k,\quad s=0,1,\dots,\\ 
	\label{R:expan}
	R(z,k;\theta_1,\theta_N): = 2\sum_{s=1}^{\infty}{\rmi^{s}}\BesselJ_{s}(k| z|)\cos\bke{\frac{s(\theta_{N}+\theta_{1}-2\varphi(z))}{2}}\sinc\bke{\frac{s\left(\theta_{N}-\theta_{1}\right)}{2}},\\
	\widetilde{R}(z;\theta_1,\theta_N,k_1,k_P): = 2\sum_{s=1}^{\infty}{\rmi^{s}}\,\widetilde{\BesselJ}_{s}(z;k_1,k_P)\cos\bke{\frac{s(\theta_{N}+\theta_{1}-2\varphi(z))}{2}}\sinc\bke{\frac{s\left(\theta_{N}-\theta_{1}\right)}{2}}\label{tildeR:expan},
\end{gather}
where we denote by $\varphi(z)$ the angle of $z$ in polar coordinates for $z\in\mathbb{R}^2$, and $\sinc(t)=\frac{\sin{t}}{t}$.
It holds that 
\begin{gather*}\widetilde{\BesselJ}_{s}(2z;k_1,k_P)=\widetilde{\BesselJ}_{s}(z;2k_1,2k_P),\\
	\widetilde{R}(2z;\theta_1,\theta_N,k_1,k_P)=\widetilde{R}(z;\theta_1,\theta_N,2k_1,2k_P),\\
	\widetilde{R}(z;\theta_1,\theta_N,k_1,k_P) = \frac{1}{k_{P}-k_{1}}\int_{k_{1}}^{k_{P}} R(z,k;\theta_1,\theta_N)\,\rmd k.
\end{gather*}

	For a fixed $z\in\RR^2$ and $k>0$, it holds as shown in \cite[Theorem 4.1]{Park} that
	\begin{equation}\label{R:int:1}
		\frac{1}{\theta_N-\theta_1}\int_{\mathbb{S}^{1}_{*}} \rme^{\rmi k\xh\cdot z} \rmd S(\xh) =J_0(k|z|) +R(z,k;\theta_1,\theta_N),
	\end{equation}
	where $\mathbb{S}^{1}_{*}$ is given by \eqnref{MeasureRange} and $I=[\theta_1,\theta_N]$.
	For a better understanding of the reader, we briefly present the proof of \eqnref{R:int:1} provided in \cite{Park}. 
	From the Jacobi--Anger expansion
	\begin{equation}\notag
		e^{\rmi r\cos\theta}=J_0(r)+2\sum_{s=1}^\infty \rmi^s J_s(r)\cos(s\theta)\quad\mbox{for }r\geq 0,\end{equation}
	we have
	\begin{align}\notag
		\int_{\mathbb{S}^{1}_{*}} \rme^{\rmi k\xh\cdot z} \rmd S(\xh)
		&= \int_{\theta_1}^{\theta_N} e^{\rmi k|z|\cos(\theta-\varphi(z))}\,d\theta\\ \label{eqn:Jacobi_Anger}
		&=(\theta_N-\theta_1)J_0(k|z|) + 2\sum_{s=1}^\infty \rmi^s J_s(k|z|)\int_{\theta_1}^{\theta_N}\cos\left(s(\theta-\varphi(z)\right)\,d\theta.
	\end{align}
	It then follows \eqnref{R:int:1} by estimating the integration term in \eqnref{eqn:Jacobi_Anger}.

\subsection{Single Frequency Measurement}\label{subsec:single}
\begin{theorem}[Asymptotic structure of the MSM in limited-aperture configuration]\label{Thm1}
	For a fixed $k$, let the measurement data $\left\{u_\infty(\hat{x}_n;k)\,:\, n=1,\dots,N\right\}$ be given for $\hat{x}_n = [\cos(\theta_n),\sin(\theta_n)]^T\in\mathbb{S}^1_*$,  where $\mathbb{S}^1_*$ is given by \eqnref{MeasureRange} for a given interval $I$. For $z\in\RR^2$ in a test domain $\Omega$, we define $\calI_{\rm{MSM}}(z;k)$ as in \eqnref{MSM2}. Set the weights $w_{m}=\alp_{m}^{2}(\eps_{m}-\eps_{0})|D_{m}|$. Then, for a sufficiently large $N$, it holds that 
	\begin{equation}\label{MSMresult}
		\calI_{\rm{MSM}}( z;k)\approx 
		\frac{\left|\Phi_{\rm{MSM}}(z;k) + \Lambda_{\rm{MSM}}(z;k)\right|}{\disp\max_{z\in\Omega}\left|\Phi_{\rm{MSM}}(z;k) + \Lambda_{\rm{MSM}}(z;k)\right|}
	\end{equation}
	with
	\begin{gather*}
		\Phi_{\rm{MSM}}(z;k)=\sum_{m=1}^{M}w_{m}\BesselJ_{0}(2k| z- c_{m}|),\\
		\Lambda_{\rm{MSM}}(z;k)=\sum_{m=1}^{M}w_{m}R(z-c_m,2k\,;\,\theta_1,\theta_N).\end{gather*}
	
\end{theorem}

\begin{proof}
	From the asymptotic formula of the far-field pattern \eqref{MonoFarFieldAsymptoticFormula}, we obtain for a sufficiently large $N$ that
	\begin{equation}\label{eqn:thm3.1}
		\begin{split}
			\bkc{u_{\infty}(\xh_{n};k),\, \rme^{-2\rmi k\xh_n\cdot z}}_{l^2(\mathbb{S}_*^1)}
			&=\sum_{m=1}^{M}\frac{1}{N}\sum_{n=1}^{N}\alp_{m}^{2}\left| D_{m}\right|\,\frac{\eps_{m}-\eps_{0}}{\sqrt{\eps_{0}\mu_{0}}}\,\rme^{2\rmi k\xh_{n}\cdot( z- c_{m})}+O(\alp^{3}),\\
			&\approx\sum_{m=1}^{M}\alp_{m}^{2}\left| D_{m}\right|\,\frac{\eps_{m}-\eps_{0}}{\sqrt{\eps_{0}\mu_{0}}}\,\frac{1}{\theta_{N}-\theta_{1}}\int_{\mathbb{S}^{1}_{*}}\rme^{2\rmi k\xh\cdot( z- c_{m})}\rmd S(\xh).
		\end{split}
	\end{equation}
	From \eqref{MSM2} and \eqref{R:int:1}, we complete the proof.	\end{proof}

\smallskip

\paragraph{Properties of the indicator function with single frequency measurement} From the asymptotic structure derived in Theorem \ref{Thm1}, we can observe the asymptotic properties of the indicator function of the MSM in limited-aperture configuration, assuming that the inhomogeneities are small and well separated and that the far-field data are obtained from the measurements, as follows:

\begin{enumerate}[(i)]
	%
	\item	The indicator function of the MSM has large values at the centers $z=c_m$ of $\tau_{m}$ since $\BesselJ_{0}(0)=1$ and $\BesselJ_{s}(0)=0$ for $s\geq1$. 
	Hence, $\Phi_{\textrm{MSM}}(z;k)$ is the concentrating part of the indicator function at the inhomogeneities $\tau_m$. On the other hand, $\Lambda_{\textrm{MSM}}(z;k)$ is the disturbing part. Note that the measurement angle $[\theta_{1},\theta_{N}]$ affects only the disturbing part and
	$$\Phi_{\textrm{MSM}}(z;k)=\mathfrak{F}_{\textrm{MSM}}(z;k),$$
	where $\mathfrak{F}_{\textrm{MSM}}(z;k)$ corresponds to the indicator function in full-aperture configuration given by \eqnref{MSM}. 
	%
	%
	
	\item It is well known that the Bessel function of the first kind of order zero satisfies
	\begin{equation}\label{Bessel:decay}
		\BesselJ_{0}(\xi)\, \approx\sqrt{\frac{2}{\pi \xi}}\,\cos\bke{ \xi-\frac{\pi}{4}}\quad\mbox{when}~\abs{\arg \xi}<\pi,~\abs{\xi}\to\infty.
	\end{equation}
	As the concentrating part $\Phi_{\textrm{MSM}}(z;k)$ in Theorem \ref{Thm1} is a linear combination of $J_0(2k|z-c_{m}|)$, it follows from \eqnref{Bessel:decay} that a high frequency ($k\to\infty$) leads to 
	the fast decay of $\Phi_{\textrm{MSM}}(z;k)$ as $z$ tends to be located away from the inhomogeneities $\tau_m$. In other words, a high frequency of the impinging wave improves the imaging accuracy of $\tau_m$ by $\calI_{\textrm{MSM}}( z;k)$. However, this is an ideal case and, in real applications, there is a limitation on the frequency.
	
	%
	%
	%
	%
	%
	\item
	Note that $\Lambda_{\textrm{MSM}}(z;k)$ is given in terms multiplied by 
	$$\cos\bke{\frac{s\left(\theta_{N}+\theta_{1}-2\varphi(z-c_m)\right)}{2}}\sinc\bke{\frac{s(\theta_{N}-\theta_{1})}{2}},$$
	which is zero for all $s$ and $m$ when $\theta_N -\theta_1 =2\pi$ (full-aperture case). 		This explains why the MSM in full-aperture configuration is more accurate than in the limited case. From the numerical study, the artifacts in the MSM are small for $\theta_N-\theta_1\geq \pi$.
	However, as $|\theta_N-\theta_1|$ decreases, the resulting $\sinc\bke{\frac{s(\theta_{N}-\theta_{1})}{2}}$ increases. 
	As a result, the imaging performance of the MSM becomes worse as the measurement interval becomes narrower. 

	\item Let us now consider more details for the case $\theta_N-\theta_1<\pi$.
	From \eqref{eqn:Jacobi_Anger} and \eqref{eqn:thm3.1}, we can express $\Lambda_{\textrm{MSM}}(z;k)$ as an integral, that is, 
	\begin{equation}\label{Lambda1:integral}
		\Lambda_{\textrm{MSM}}(z;k)=\sum_{m=1}^M w_m \sum_{s=1}^\infty 2\,\rmi^s \BesselJ_{s}(2k| z- c_{m}|)\,\frac{1}{\theta_N -\theta_1}\int_{\theta_1}^{\theta_N}\cos\big(s(\theta-\varphi(z-c_m))\big)\,d\theta.
	\end{equation}
	Hence, the disturbing term $\Lambda_{\textrm{MSM}}(z;k)$ is relatively small for $z$ such that the integral in \eqref{Lambda1:integral} is small. 
	We have
	\begin{align*}
		\int_{\theta_1}^{\theta_N}\cos\big(s(\theta-\varphi(z-c_m))\big)\,d\theta
		=\frac{1}{s} \int_{s(\theta_1-\varphi(z-c_m))}^{s(\theta_N -\varphi(z-c_m))} \cos\theta\,d\theta.
	\end{align*}
	This integral with $s=1$ vanishes if $z$ satisfies 
	\begin{equation}\label{cond:varphi}
		\varphi(z-c_m)\approx \frac{\theta_1+\theta_N}{2}\pm \frac{\pi}{2}.
	\end{equation}
	For the single inhomogeneity case, the artifacts in the indicator function $\calI_{\textrm{MSM}}( z;k)$ are small for $z$ satisfying \eqref{cond:varphi} as shown in Figure \ref{Ex1-MSMResult}\,(a).
	For the multiple inhomogeneities case, the artifacts are small for $z$ located near to one of the inhomogeneities and satisfying \eqnref{cond:varphi} as shown in Figure \ref{MSM2-1}.

	%
	%
	
\end{enumerate}

\smallskip

\subsection{Multiple Frequencies Measurement}
We now propose a multi-frequency MSM (MMSM).
Different from the single frequency measurement case, we assume a multi-frequency data set; that is, 
$
u_\infty(\hat{x}_n;k_p)$ for $n=1,\dots,N,\ p=1,\dots,P$,
where $\hat{x}_n$ are restricted in $\mathbb{S}^1_*\subset \mathbb{S}^1$ given by \eqref{MeasureRange}, and $k_p$ are ordered in increasing values. 
We define the MMSM indicator function as
\begin{equation}\label{MMSM}
	\calI_{\textrm{MMSM}}( z):=	
	\abs{\frac{1}{P}\,\sum_{p=1}^{P}\,\frac{\displaystyle\bkc{u_{\infty}(\xh,k_{p}),\rme^{-2\rmi k_{p}\xh\cdot z}}_{l^2(\mathbb{S}_*^1)}}
		{\disp\max_{z\in\Omg}\abs{\bkc{u_{\infty}(\xh,k_{p}),\rme^{-2\rmi k_{p}\xh\cdot z}}_{l^2(\mathbb{S}_*^1)}}}\,}
\end{equation}
for a point $z\in\RR^2$ in a test compact region $\Omega$.

\begin{theorem}\label{Thm2}
	Let $\left\{u_\infty(\hat{x}_n;k_{p})\,:\, n=1,\dots,N~\mbox{and}~p=1,\dots,P\right\}$ be given for $\hat{x}_n = [\cos(\theta_n),\sin(\theta_n)]^T\in\mathbb{S}^1_*$ for a given interval $I$. Set the weights $w_{m}=\alp_{m}^{2}(\eps_{m}-\eps_{0})|D_{m}|$. Then, for a sufficiently large $N$ and $P$ with $\theta_1,\theta_N,k_1,k_P$ fixed, we have
	\begin{equation}\label{MMSMresult}
		\calI_{\rm{MMSM}}(z)\approx 
		\left|\Phi_{\rm{MMSM}}(z) + \Lambda_{\rm{MMSM}}(z)\right|
	\end{equation}
	%
	with
	\begin{gather*}
		\Phi_{\rm{MMSM}}(z):=\sum_{m=1}^{M}w_{m}\, \widetilde{\BesselJ}_0(z-c_m;2k_1,2k_P),\\
		\Lambda_{\rm{MMSM}}(z): = \sum_{m=1}^{M}w_{m} \widetilde{R}(z-c_m;\,\theta_1,\theta_N,2k_1,2k_P).
	\end{gather*}
\end{theorem}

\begin{proof}
	From Theorem \ref{Thm1}, we have
	\begin{equation*}
		\frac{\displaystyle\bkc{u_{\infty}(\xh_n;k_{p}),\,\rme^{-2\rmi k_{p}\xh_n\cdot z}}_{l^{2}(\mathbb{S}^{1}_{*})}}
		{\disp\max_{z\in\Omg}\abs{\bkc{u_{\infty}(\xh_n;k_{p}),\,\rme^{-2\rmi k_{p}\xh_n\cdot z}}_{l^{2}(\mathbb{S}^{1}_{*})}}}
		\propto\Phi_{\textrm{MSM}}(z;k_{p}) + \Lambda_{\textrm{MSM}}(z;k_{p})
	\end{equation*}
	and, hence,
	\begin{equation*}
		\frac{1}{P}\sum_{p=1}^{P}\frac{\displaystyle\bkc{u_{\infty}(\xh,k_{p}),\rme^{-2\rmi k_{p}\xh\cdot z}}_{l^{2}(\mathbb{S}^{1}_{*})}}
		{\disp\max_{z\in\Omg}\abs{\bkc{u_{\infty}(\xh,k_{p}),\rme^{-2\rmi k_{p}\xh\cdot z}}_{l^{2}(\mathbb{S}^{1}_{*})}}}
		\propto \frac{1}{P}\sum_{p=1}^{P}\Big(\Phi_{\textrm{MSM}}(z;k_{p}) + \Lambda_{\textrm{MSM}}(z;k_{p})\Big).
	\end{equation*}
	
	First, we obtain
	\begin{equation*}
		\begin{split}
			\frac{1}{P}\sum_{p=1}^{P}\Phi_{\textrm{MSM}}(z;k_{p}) 
			&=\frac{1}{P}\sum_{p=1}^{P}\sum_{m=1}^{M}w_{m}\BesselJ_{0}(2k_p| z- c_{m}|)\\
			&\approx \sum_{m=1}^{M}w_{m}\frac{1}{k_{P}-k_{1}}\int_{k_{1}}^{k_{P}}\BesselJ_{0}(2k| z- c_{m}|) \,\rmd k=\Phi_{\textrm{MMSM}}(z).
		\end{split}
	\end{equation*}
	
	Second, we estimate the disturbing part. The Bessel functions satisfy that (see, for example, \cite[Section 3.4]{Colton_Kress1})
	\begin{equation}\label{Bessel:decay_s}
		J_s(t)=\frac{t^n}{2^s\, s!}\left(1+O\left(\frac{1}{s}\right)\right)\quad \mbox{as }s\rightarrow\infty
	\end{equation}
	uniformly on compact subsets of $\RR$. 
	Hence, for a given $\epsilon>0$, there exists a sufficiently large number $S$ independent of $p,m,z$ such that 
	\begin{align*}
		\bigg|\Lambda_{\textrm{MSM}}(z;k_p)-\sum_{m=1}^{M}2w_{m}&\sum_{s=1}^S{\rmi^{s}}\BesselJ_{s}(2k_p| z- c_{m}|)\cos\bke{\frac{s(\theta_{N}+\theta_{1}-2\varphi_m)}{2}}\sinc\bke{\frac{s(\theta_{N}-\theta_{1})}{2}}\bigg|<\epsilon
	\end{align*}
	with $\varphi_m =\varphi(z-c_m)$. 
	Similarly, for a sufficiently large $P$ independent of $s,m,z$, we have
	$$\left|\frac{1}{P}\sum_{p=1}^{P}\BesselJ_{s}(2k_{p}| z- c_{m}|) -\frac{1}{k_{P}-k_{1}}\int_{k_1}^{k_P}\BesselJ_{s}(2k_p| z- c_{m}|)\rmd k\right|<\ep.$$
	It then follows that (see also \eqref{def:tilde:J})
	\begin{align*}
		\Bigg|&\frac{1}{P}\sum_{p=1}^{P}\Lambda_{\textrm{MSM}}(z;k_{p})\\ 
		&\qquad-  \sum_{m=1}^{M}2w_{m}\sum_{s=1}^{S} {\rmi^{s}}\cos\bke{\frac{s(\theta_{N}+\theta_{1}-2\varphi_m)}{2}}\sinc\bke{\frac{s(\theta_{N}-\theta_{1})}{2}}
		\widetilde{\BesselJ}_s(z-c_m;2k_1,2k_P)\Bigg|<C\epsilon
	\end{align*}
	for some constant $C$. By again applying \eqref{Bessel:decay_s}, we conclude that 
	$$\frac{1}{P}\sum_{p=1}^{P}\Lambda_{\textrm{MSM}}(z;k_{p})\approx \Lambda_{\textrm{MMSM}}(z).$$
	This proves \eqref{MMSMresult}.
\end{proof}

\begin{figure}[t!]
	\centering
	\subfigure[$\Phi_{\textrm{MSM}}(z;k)$\label{Bessel}]{\includegraphics[width=0.4\textwidth]{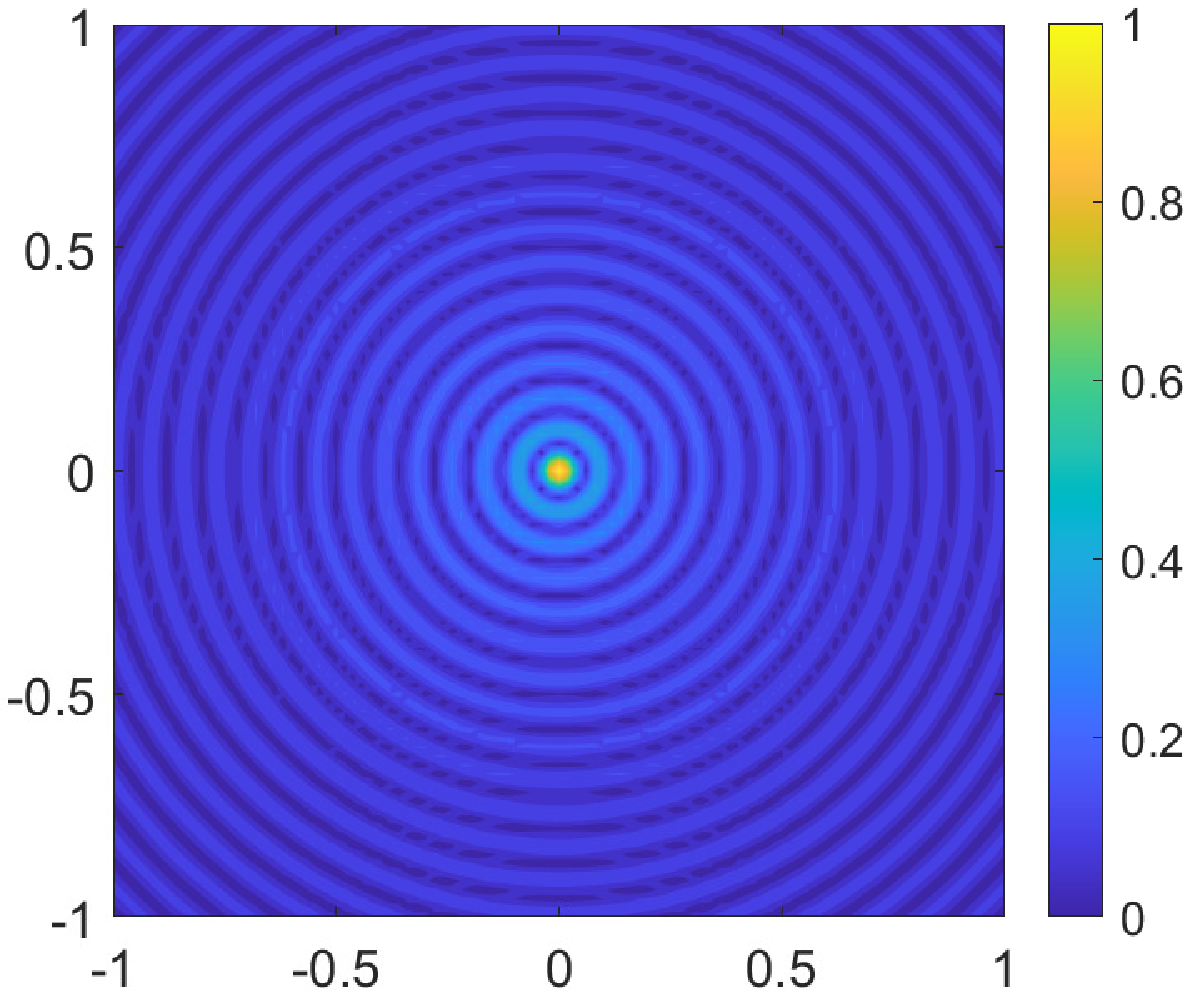}}
	\hskip .5cm
	\subfigure[$\Phi_{\textrm{MMSM}}(z)$\label{Struve1}]{\includegraphics[width=0.4\textwidth]{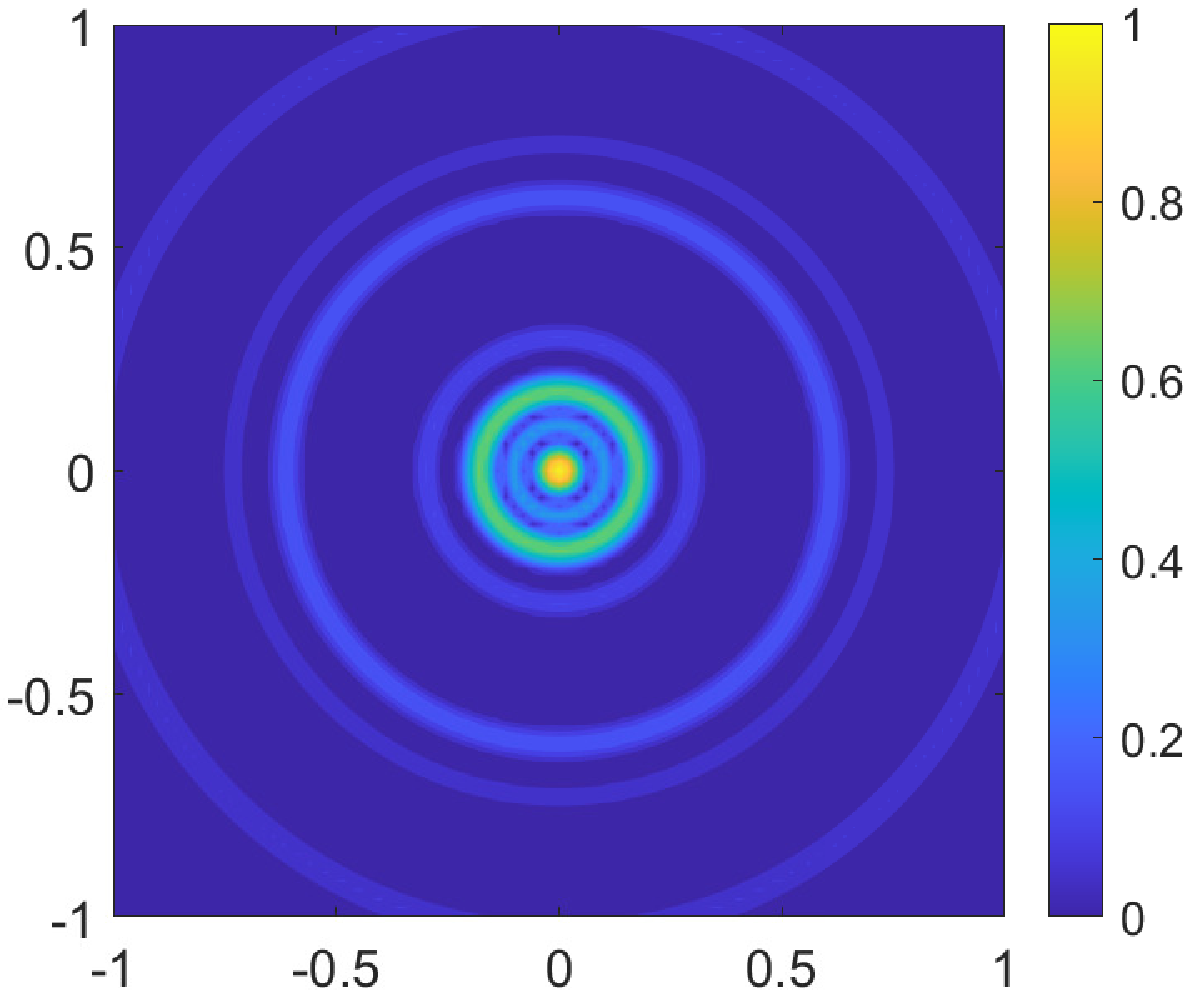}}
	\caption{Concentrating part of the indicator function in the presence of a small inhomogeneity centered at the origin for (a) MSM with the frequency \SI{1}{\GHz}, and (b) MMSM with $7$ frequencies from \SI{700}{\MHz} to \SI{1.3}{\GHz}. Each figure is normalized.}
	\label{Comparison_MSM_MMSM}
\end{figure}

\paragraph{Properties of the indicator function with multiple frequencies measurement} From the asymptotic structure derived in Theorem \ref{Thm2}, we can observe the properties of the MMSM in limited-aperture configuration, assuming that the inhomogeneities are small and well separated and that the far-field data are obtained from the measurements, as follows:
\begin{enumerate}[(i)]
	\item The indicator function of the MMSM can be decomposed into the concentrating term $\Phi_{\textrm{MMSM}}(z)$ and the disturbing term $\Lambda_{\textrm{MMSM}}(z)$ that have exactly the same form as those of the MSM, except that the Bessel function $J_s(2k|z-c_m|)$ is now replaced with $\widetilde{J}_s(z;k_1,k_P)$ (the mean value of the Bessel function with respect to the frequency). 
	
	\item As shown in Figure \ref{Comparison_MSM_MMSM}, the concentrating term of the multi-frequency measurement case is better concentrated than that of the single frequency near the location of the inhomogeneity. Hence, we conclude that the MMSM is an improved version of the MSM for imaging small inhomogeneities.

	\item Since $\Lambda_{\textrm{MMSM}}(z)$ is given in terms multiplied by the same cosine function and sinc function terms as $\Lambda_{\textrm{MSM}}(z;k)$, it has a similar dependence on the measurement interval $[\theta_1,\theta_N]$. In particular, the MMSM with $\theta_N-\theta_1=2\pi$ (full-aperture case) is more accurate than the limited case, and it shows worse results with smaller $\theta_N-\theta_1$. Similar to the single frequency measurement, the artifacts in the MMSM are significantly reduced for $\theta_N-\theta_1\geq \pi$ (see Figure \ref{Ex2-MSMResult} in Section \ref{sec:5}). For the case $\theta_N-\theta_1<\pi$, the artifact in $\calI_{\textrm{MMSM}}(z)$ is small for $z$ satisfying \eqref{cond:varphi}. 
\end{enumerate}

\begin{remark}
	Thanks to the hypothesis of sufficiently large $P$ and the indefinite integral formula of the Bessel function \cite[p.7]{integral_bessel}:
	\begin{equation*}
		\int \BesselJ_{0}(t)\rmd t = t\BesselJ_{0}(t) + \frac{\pi t}{2}\Big(\BesselJ_{1}(t)\Struve_{0}(t) - \BesselJ_{0}(t)\Struve_{1}(t)\Big),
	\end{equation*}
	one can easily find that
	$$\widetilde{J}_0(z;k_1,k_p)=\frac{1}{k_P-k_1} \left(\mathcal{C}(z;k_P)-\mathcal{C}(z;k_1)\right)$$
	with
	\begin{multline}\label{MMSM_Contributing1}
		\mathcal{C}(z;k):= k\BesselJ_{0}(2k| z- c_{m}|)
		+ \frac{\pi k}{2}\Big(\BesselJ_{1}\big(2k| z- c_{m}|\big)\Struve_{0}\big(2k| z- c_{m}|\big)
		- \BesselJ_{0}\big(2k| z- c_{m}|\big)\Struve_{1}\big(2k| z- c_{m}|\big)\Big),
	\end{multline}
	where $\Struve_{s}$ indicates the Struve function of integer order $s$ (see \cite[Chapter 12]{bessel_book} for the definition and properties of the Struve function).
\end{remark}

\section{Comparison of MSM and Classical DSM in Limited-Aperture Configuration}\label{sec:4}
Recall that we define the classical DSM, \eqref{DSM}, by using a fixed incoming wave direction $\vartheta$ and frequency $k$ in full-aperture configuration. The same definition (and its extension to the multiple frequencies measurement case) can be adopted in limited-aperture configuration, and the only difference is that $\hat{x}_n$ are now restricted in a subset $\mathbb{S}^1_*$ of $\mathbb{S}^1$.
In \cite{Kang_mfDSM_limited}, the asymptotic features of the classical DSM with single- and multi--frequency measurement in limited-aperture configuration was investigated.  In Subsection \ref{subsec:review:Park}, we review the results obtained in \cite{Kang_mfDSM_limited}, and we then compare our results on the MSM (that is, Theorems \ref{Thm1} and \ref{Thm2}) and the classical DSM in Subsection \ref{subsec:comparison}, given the same amount of measurement data.

\subsection{Classical DSM in Limited-Aperture Configuration}\label{subsec:review:Park}
As in the previous section, we set $\hat{x}_n=[\cos(\theta_n),\sin(\theta_n)]^T$ for $n=1,\dots,N$, where $I=[\theta_1,\theta_N]$ is a given fixed subinterval of $[0,2\pi]$. We accordingly define $\mathbb{S}^1_*$. We also denote by $\Omg$ a test region. 
\begin{theorem}[\cite{Kang_mfDSM_limited}, single frequency measurement]\label{thm:review:single}
	Let the measurement data $\{u_\infty(\hat{x}_n,\vartheta;k): n=1,\dots,N\}$ be given, where $k$ and $\vartheta$ are fixed. For a sufficiently large $N$, the indicator function of the classical DSM, defined by \eqref{DSM}, has the following asymptotic formula: for $z\in\Omg$,
	\begin{equation}\label{DSMresult}
		\mathcal{I}_{\mathrm{DSM}}(z,\vartheta;k)\approx\frac{\left|\Phi_{\mathrm{DSM}}(z,\vartheta;k) + {\Lambda}_{\mathrm{DSM}}(z,\vartheta;k)\right|}{\disp\max_{z\in\Omega}\left|\Phi_{\mathrm{DSM}}(z,\vartheta;k) + {\Lambda}_{\mathrm{DSM}}(z,\vartheta;k)\right|}
	\end{equation}
	with
	\begin{gather}\label{DSMcontri}
		\Phi_{\mathrm{DSM}}(z,\vartheta;k) = \sum_{m=1}^{M}w_{m}\,\rme^{\rmi k\vartheta\cdot c_{m}}\BesselJ_{0}(k\abs{z-c_{m}}),\\
		\label{DSMdistrub}
		{\Lambda}_{\mathrm{DSM}}(z,\vartheta,k)=\sum_{m=1}^{M}w_{m}\,\rme^{\rmi k\vartheta\cdot c_{m}} R(z-c_m,k;\theta_1,\theta_N).\end{gather}
\end{theorem}
\begin{proof}
	One can prove the theorem from the asymptotic expression of the far-field pattern in \eqref{FarFieldAsymptoticFormula} and \eqref{R:int:1}.
\end{proof}


With the data corresponding to multiple frequencies $k_1$,\dots,$k_P$, we now define the indicator function 
\begin{equation}\label{MDSM}
	\mathcal{I}_{\mathrm{MDSM}}(z,\vartheta)
	=\abs{\frac{1}{P}\displaystyle\sum_{p=1}^{P}\frac{\langle u_{\infty}(\hat{x}_{n},\vartheta;k_{p}),\rme^{-\rmi k_{p}\hat{x}_{n}\cdot z}\rangle}
		{\disp\max_{z\in\Omg}\abs{\langle u_{\infty}(\hat{x}_{n},\vartheta;k_{p}),\rme^{-\rmi k_{p}\hat{x}_{n}\cdot z}\rangle}}\,},
\end{equation}
which satisfies from Theorem \ref{thm:review:single} that
\begin{equation}\label{lemma:MDSM}
	\mathcal{I}_{\mathrm{MDSM}}(z,\vartheta)\approx {\left|{\widetilde{\Phi}}_{\mathrm{MDSM}}(z,\vartheta) + {\widetilde{\Lambda}}_{\mathrm{MDSM}}(z,\vartheta)\right|}
\end{equation}
with
\begin{align*}
	{\widetilde{\Phi}}_{\mathrm{MDSM}}(z,\vartheta)&=\frac{1}{P}\sum_{p=1}^P \Phi_{\mathrm{DSM}}(z,\vartheta;k_p)
	= \sum_{m=1}^{M}w_{m}\frac{1}{P}\sum_{p=1}^P\rme^{\rmi k_p\vartheta\cdot c_{m}}\BesselJ_{0}(k_p\abs{z-c_{m}}),\\
	{\widetilde{\Lambda}}_{\mathrm{MDSM}}(z,\vartheta)&=\frac{1}{P}\sum_{p=1}^P \Lambda_{\mathrm{DSM}}(z,\vartheta;k_p)
	=\sum_{m=1}^{M}w_{m}\frac{1}{P}\sum_{p=1}^P\rme^{\rmi k_p\vartheta\cdot c_m} R(z-c_m,k_p;\theta_1,\theta_N).
\end{align*}

For a sufficiently large $P$, we can approximate these summations in $p$ by using the notations
\begin{equation}\label{def:doubletildeJ}
	\doublewidetilde{\BesselJ}_{s}(z;\vartheta,k_{1},k_{P}):=\frac{1}{k_{P}-k_{1}}\int_{k_{1}}^{k_{P}} \rme^{\rmi k \vartheta\cdot c_{m}}\BesselJ_{s}(k|z|)\rmd k,\quad s=1,2\dots,
\end{equation}
\begin{equation}
	\doublewidetilde{R}(z;\theta_1,\theta_N,k_1,k_P): = 2\sum_{s=1}^{\infty}{\rmi^{s}}\,\doublewidetilde{\BesselJ}_{s}(z;\vartheta,k_1,k_P)\cos\bke{\frac{s(\theta_{N}+\theta_{1}-2\varphi(z))}{2}}\sinc\bke{\frac{s\left(\theta_{N}-\theta_{1}\right)}{2}}.
\end{equation}
Then, \eqref{lemma:MDSM} leads us to the following theorem.
\begin{theorem}[\cite{Kang_mfDSM_limited}, multiple frequencies measurement]\label{thm:review:multiple}
	Let the measurement data $\{u_\infty(\hat{x}_n,\vartheta;k_p): n=1,\dots,N~\mbox{and}~p=1,\dots,P\}$ be given, where $\vartheta$ is fixed and $k_p$, $p=1,\dots,P$, are sample points of a fixed interval $[k_1,k_P]$. For a sufficiently large $N$ and $P$, the indicator function of the multi-frequency DSM, defined by \eqref{MDSM}, satisfies the asymptotic formula: for $z\in\Omg$,
	\begin{equation}\notag
		\mathcal{I}_{\mathrm{MDSM}}(z,\vartheta)\approx{\left|\Phi_{\mathrm{MDSM}}(z,\vartheta) + \Lambda_{\mathrm{MDSM}}(z,\vartheta)\right|}
	\end{equation}
	with	
	\begin{equation*}
		\begin{split}
			\Phi_{\mathrm{MDSM}}(z,\vartheta)&=\sum_{m=1}^{M}w_{m}\doublewidetilde{\BesselJ}_0(z-c_{m};k_1,k_P),\\
			\Lambda_{\mathrm{MDSM}}(z,\vartheta)& = \sum_{m=1}^{M}w_{m}\doublewidetilde{R}(z-c_{m};\,\theta_1,\theta_N,k_1,k_P).
		\end{split}
	\end{equation*}
	
\end{theorem}

Recall that $\vartheta$ is fixed. We set $\vartheta\cdot c_m = |c_m|\cos(\psi_m)$ for each location of inhomogeneities. In \cite{Kang_mfDSM_limited}, it was shown that at $z=c_m$, 
the concentrating term $\Phi_{\mathrm{MDSM}}(z;\vartheta)$ satisfies
\begin{equation}\notag
	\Phi_\textrm{MDSM}(c_{m};\vartheta) = \frac{1}{k_P-k_1}\left(\Phi^{(1)}+ \Phi^{(2)}\right)
\end{equation}
with 
\begin{align}
	\begin{aligned}\label{LimitationMDSM_temp1}
		\Phi^{(1)}=
		&{_1 \mathrm{F}_2}\bke{\frac{1}{2};1,\frac{3}{2};\frac{|c_m|^{2}k_{P}^{2}}{4}}k_{P}-{_1 \mathrm{F}_2}\bke{\frac{1}{2};1,\frac{3}{2};\frac{|c_m|^{2}k_{1}^{2}}{4}}k_{1},\\
		\Phi^{(2)}=&\sum_{t=1}^{\infty}\frac{\cos(t\psi_{m})}{2^{t-1}\Gamma(t+2)}\Bigg\{{_1\mathrm{F}_{2}}\bke{\frac{t+1}{2};t+1,\frac{t+3}{2};-\frac{1}{4}k_{P}^{2}|c_{m}|^{2}}k_{P}\\
		&\qquad\qquad\qquad\qquad- {_1 \mathrm{F}_{2}}\bke{\frac{t+1}{2};t+1,\frac{t+3}{2};-\frac{1}{4}k_{1}^{2}|c_{m}|^{2}}k_{1}\Bigg\},
	\end{aligned}
\end{align}
where ${_a \mathrm{F}_{b}}$ denotes the generalized hypergeometric function of orders $a$ and $b$.

Due to properties of the hypergeometric functions, the indicator function $\mathcal{I}_{\textrm{MDSM}}(z;\vartheta)$ at $c_m$, $m=1,\dots,M$, may have significant differences in magnitude depending on the location of $c_m$ even though other characteristics (e.g., permittivity, size, shape) are the same. Hence,  it is in general impossible to detect all multiple inhomogeneities via MDSM.

\subsection{Comparison of Indicator Functions of MSM and Classical DSM}\label{subsec:comparison}

We compare the asymptotic features of the MSM and classical DSM with single- and multi--frequency measurement in limited measurement environments.

\paragraph{Single frequency measurement}
From Theorem \ref{Thm1} and Theorem \ref{thm:review:single}, the indicator functions of the MSM and DSM asymptotically have the concentrating terms 
\begin{gather*}
	\Phi_{\textrm{MSM}}(z;k)= \sum_{m=1}^{M}w_{m}\BesselJ_{0}(2k| z- c_{m}|),\\
	\Phi_{\textrm{DSM}}(z,\vartheta;k)= \sum_{m=1}^{M}w_{m}\,\rme^{\rmi k\vartheta\cdot c_{m}}\BesselJ_{0}(k\abs{z-c_{m}}),
\end{gather*}
where the frequency $k$ and the direction $\vartheta$ are fixed. 
The asymptotic features of the MSM and classical DSM with single frequency have similarities and differences, which are as follows:
\begin{enumerate}[(i)]
	\item The functions $\BesselJ_{0}(2k|z-c_{m}|)$ (in MSM) and $\BesselJ_{0}(k|z-c_{m}|)$ (in DSM), that are defined using the Bessel functions of the first kind, are essential to image the inhomogeneities. The indicator function of the MSM has more oscillations than that of the DSM.
	
	\item The exponential term $\rme^{\rmi k\vartheta\cdot c_{m}}$ is multiplied only for the DSM. For the single frequency case, $k$ is fixed and, thus, $\rme^{\rmi k\vartheta\cdot c_{m}}$ is a constant for each inhomogeneity $\tau_m$. The MSM and DSM perform similarly for both single and multiple inhomogeneities imaging.

	\item  Similar to the discussion in Section \ref{subsec:single}, for both the MSM and DSM, the effect caused by the disturbing term is significantly decreased when $\theta_{n}-\theta_{1}\geq\pi$.

\end{enumerate}

\paragraph{Multiple frequencies measurement} For the multiple frequencies case, we compare $\Phi_{\textrm{MMSM}}(z)$ 
and $\Phi_{\textrm{MDSM}}(z,\vartheta)$, where the wavenumbers $k=k_1,\dots,k_p$ are used and $\vartheta$ is fixed. 
The asymptotic features of the MMSM and MDSM are determined by
$$	\frac{1}{P}\sum_{p=1}^{P}\Phi_{\textrm{MSM}}(z;k_{p}) \quad\mbox{and}\quad
\frac{1}{P}\sum_{p=1}^P \Phi_{\mathrm{DSM}}(z,\vartheta;k_p),$$ 
respectively. 
They have different features as follows.
\begin{enumerate}[(i)]
	\item 
	The MMSM shows an improved imaging performance for both single and multiple inhomogeneities. However, the MDSM may not detect some of the inhomogeneities for the multiple inhomogeneities case. The unexpected phenomenon of the MDSM comes from the effect of source term $\rme^{\rmi k \vartheta\cdot c_{m}}$ in $\Phi_{\textrm{DSM}}$. However, the MSM simultaneously tests the effects of the source and receiver directions. This leads that the concentrating term of the MMSM does not include $\rme^{\rmi k \vartheta\cdot c_{m}}$ or similar terms.

	\item For both the MMSM and MDSM, the effect caused by this disturbing term is significantly decreased when $\theta_{n}-\theta_{1}\geq\pi$.
\end{enumerate}

\section{Numerical Simulation}\label{sec:5}
In this section, we perform numerical simulations to validate our theoretical results.  
In Subsection \ref{sec:numerical:MSM}, we illustrate the indicator functions in monostatic configuration (i.e., the MSM and MMSM); see \eqref{MSM2} and \eqref{MMSM}. In Subsection \ref{sec:numerical:DSM}, we compare the results with the classical DSM. 

For the MSM, we use the the measurement data $\left\{u_\infty(\hat{x}_n;k)\,:\, n=1,\dots,N\right\}$ with $k$ fixed, where $\hat{x}_n=[\cos(\theta_n),\sin(\theta_n)]^T$ and the angle $\theta_n$ is contained in an interval. For the MMSM, we use the data with multiple frequencies $\left\{u_\infty(\hat{x}_n;k_{p})\,:\, n=1,\dots,N,\ p=1,\dots,P\right\}$.

To compare the methods with the classical DSM, for the same inhomogeneities as in Subsection \ref{sec:numerical:MSM}, we show the imaging results obtained with the indicator functions of the classical DSM given by \eqref{DSM} and \eqref{MDSM}.
For the DSM, we use $\{u_\infty(\hat{x}_n,\vartheta;k): n=1,\dots,N\}$ with fixed $k$, $\vartheta$.
For the MDSM, we use $\{u_\infty(\hat{x}_n,\vartheta;k_p): n=1,\dots,N,\ p=1,\dots,P\}$ with fixed $\vartheta$.

The following three cases are considered: single small inhomogeneity, multiple small inhomogeneities, and extended target.
For all examples, the far-field pattern is obtained by using FEKO, a commercial EM simulation software. For the single frequency case, we set $f=\SI{1}{\GHz}=c_{0}/\lambda_{0}$, where $c_{0}=1/\sqrt{\eps_{0}\mu_{0}}\approx\SI{299792458}{\meter/\second}$ is the speed of light and wavelength $\lam_{0}={2\pi}/{k}=\SI{0.2997925}{\meter}\approx\SI{0.3}{\meter}$. For the multi-frequency case, we use the measurements with $7$ frequencies from \SI{700}{\MHz} to \SI{1.3}{\GHz} with $\SI{100}{\MHz}$ step size. The $k_{p}$ is the corresponding wavenumber for each frequency. As the region of interest $\Omg$, a square domain with sides of length $\SI{2}{\meter}(\approx\frac{20}{3}\lam_{0})$ is used with $101\times101$ discretization.
To show the robustness of our numerical schemes, we add $\SI{20}{\dB}$ white Gaussian random noise on the far-field pattern in all examples except Figure \ref{Ex2-1-MSMResult} using the \textit{awgn} command in \textit{Matlab}. In Figure \ref{Ex2-1-MSMResult}, we add $\SI{10}{\dB}$ noise to the unperturbed data.

\subsection{Results of the MSM and MMSM}\label{sec:numerical:MSM}
\subsubsection{Single Small dielectric Disk}\label{sec:Ex1}

We consider a single small dielectric disk with radius $\alp=0.1\lambda_{0}$ and permittivity $\eps=3\eps_{0}$ ($\mu=\mu_0$). Figure \ref{Ex1-MSMResult} shows that the inhomogeneity can be identified via the MSM in limited-aperture configurations even with a narrow range of measurement. As expected, the MMSM has better results than the MSM. 

As one of the main features of the MSM, the narrow measurement angle increases the effect of the disturbing term (that is, $\Lambda_{\textrm{MSM}}(z;k)$), as shown in Figure \ref{MSM1-1}. For the wide range of observation directions with $\theta_{N}-\theta_{1}\geq\pi$, the effect due to the disturbing term is decreased. However, large oscillations are still observed in Figures \ref{MSM1-2} and \ref{MSM1-3} due to the oscillation properties of the Bessel function, which is the concentrating term in the indicator function. As discussed in Subsection \ref{sec:4}, the concentrating term of the MMSM is better localized near the inhomogeneities and, hence, shows significantly reduced oscillations compared to the MSM.


\subsubsection{Multiple Small Dielectric Disks}\label{sec:Ex2}
We provide the imaging results when three small dielectric circular inhomogeneities are embedded in the background medium. We assume that the sizes ($\alp_{m}=0.1\lambda_{0}$) and permittivities ($\eps_{m}=3\eps_{0}$) are identical for all $m=1,2,3$. Again, we assume that $\mu_m=\mu_0$. Figure \ref{Ex2-MSMResult} shows the indicator functions of the MSM and MMSM with various ranges of angles in monostatic configurations. The MSM can localize all three inhomogeneities by using the data with $\theta_{N}-\theta_{1}\geq\pi$. The MMSM shows significantly improved results, and it identifies all three inhomogeneities even when $\theta_{N}-\theta_{1}=\frac{\pi}{2}$.

\subsubsection{Noise Robustness}\label{sec:Ex2-1}
	From the previous studies including the results in \cite{dsm2d_farfield,Kang_3DDSM,PARK2018648}, it is well known that the DSM is a noise robust technique for locating inhomogeneities in various inverse scattering problems. 
	Figure \ref{Ex2-1-MSMResult} presents the maps of the MSM and MMSM for the multiple small dielectric disks (the same example as in Figure \ref{Ex2-MSMResult}) with \SI{10}{\dB} white Gaussian random noise added to the unperturbed data. The results in Figure \ref{Ex2-MSMResult} (with \SI{20}{\dB} noise) and Figure \ref{Ex2-1-MSMResult} (with \SI{10}{\dB} noise) show very similar results, i.e., the noise level of data has little effect on the imaging performance of the method. The MSM and MMSM are strong to the noise of input data.

\subsubsection{Extended Target}
Our theoretical results are based on the small volume hypothesis of inhomogeneities.
In this example, we test our proposed methods in imaging an extended target, which has a size bigger than $\frac{\pi}{k_{p}}=\frac{\lambda_{p}}{2}$ for all $p=1,2,\dots,7$. 
Figure \ref{Ex3-MSMResult} visualizes the indicator functions of the MSM and MMSM with an extended target, which is a disk and is indicated in white. Using the MSM (that is, with single frequency data), one can identify the center of the inhomogeneity. However, non-center points cannot be detected even with a wide range of monostatic measurement systems. Interestingly, the indicator function of the MMSM attains high magnitude at the points near the boundary points of the extended target located in the direction of the measurements.

\subsection{Comparison with the Results of the Classical DSM}\label{sec:numerical:DSM}

\smallskip 

Figure \ref{Ex1-DSMResult} shows the indicator functions of the DSM and MDSM for the same inhomogeneity as in Figure \ref{Ex1-MSMResult}. 
As discussed in Section \ref{subsec:comparison}, for the single inhomogeneity imaging, the DSM and MDSM show the imaging performances similar to those of the MSM and MMSM.  Also, the MDSM (using the data with multi-frequency measurements) shows better results than the DSM (with single-frequency measurements). Oscillations appear with higher spatial frequency in Figure \ref{Ex1-MSMResult} than in Figure \ref{Ex1-DSMResult}, which can be expected from Figure \ref{Comparison_MSM_MMSM}. 

Figure \ref{Ex2-DSMResult} shows that the DSM cannot image multiple inhomogeneities. The MDSM detects some, but not all, inhomogeneities even with the wide range of observations. 
However, all three inhomogeneities are identified by the MMSM (see Figure \ref{Ex2-MSMResult}). We conclude that the MMSM performs with higher accuracy in detecting multiple inhomogeneities than the MDSM.
Similarly, the extended circular target can not be recognized via the DSM and MDSM, shown in Figure \ref{Ex3-DSMResult}. This will be addressed in future research.

\begin{figure}[h!]
	\centering
	\subfigure[\label{MSM1-1}$\theta_{1}=0$ and $\theta_{N}=\frac{\pi}{2}$]{\includegraphics[width=0.32\textwidth]{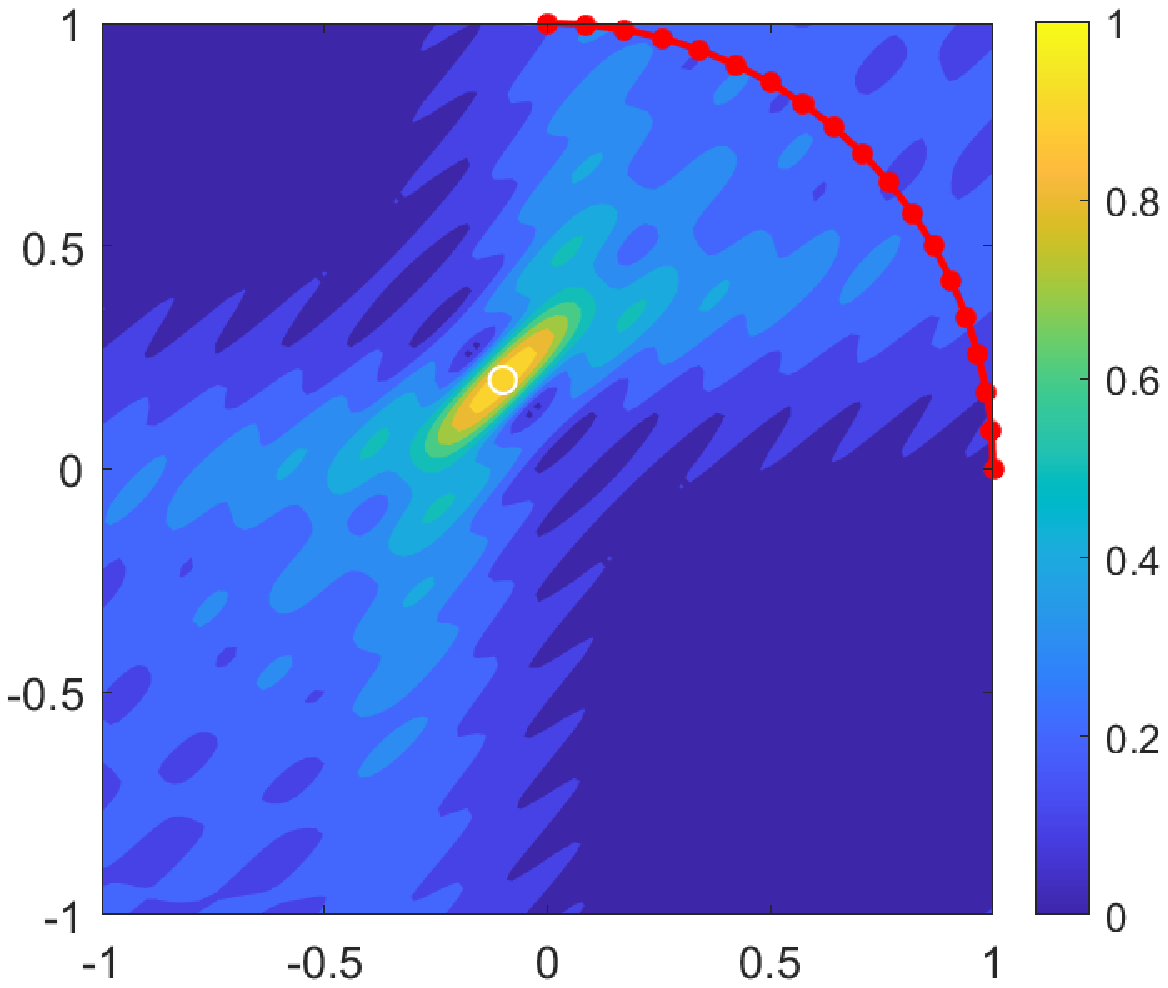}}
	\subfigure[\label{MSM1-2}$\theta_{1}=0$ and $\theta_{N}=\pi$]{\includegraphics[width=0.32\textwidth]{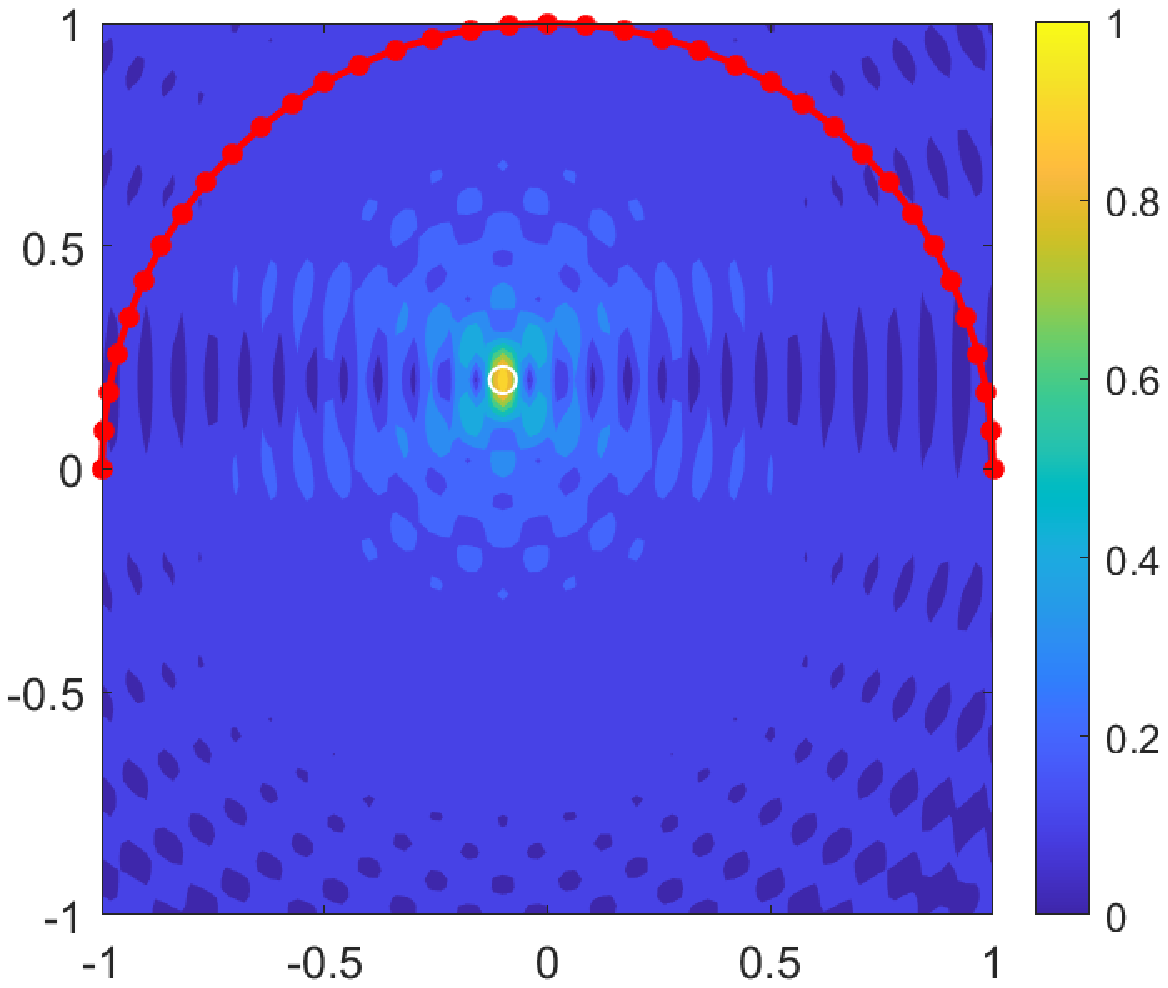}}
	\subfigure[\label{MSM1-3}$\theta_{1}=0$ and $\theta_{N}=\frac{3}{2}\pi$]{\includegraphics[width=0.32\textwidth]{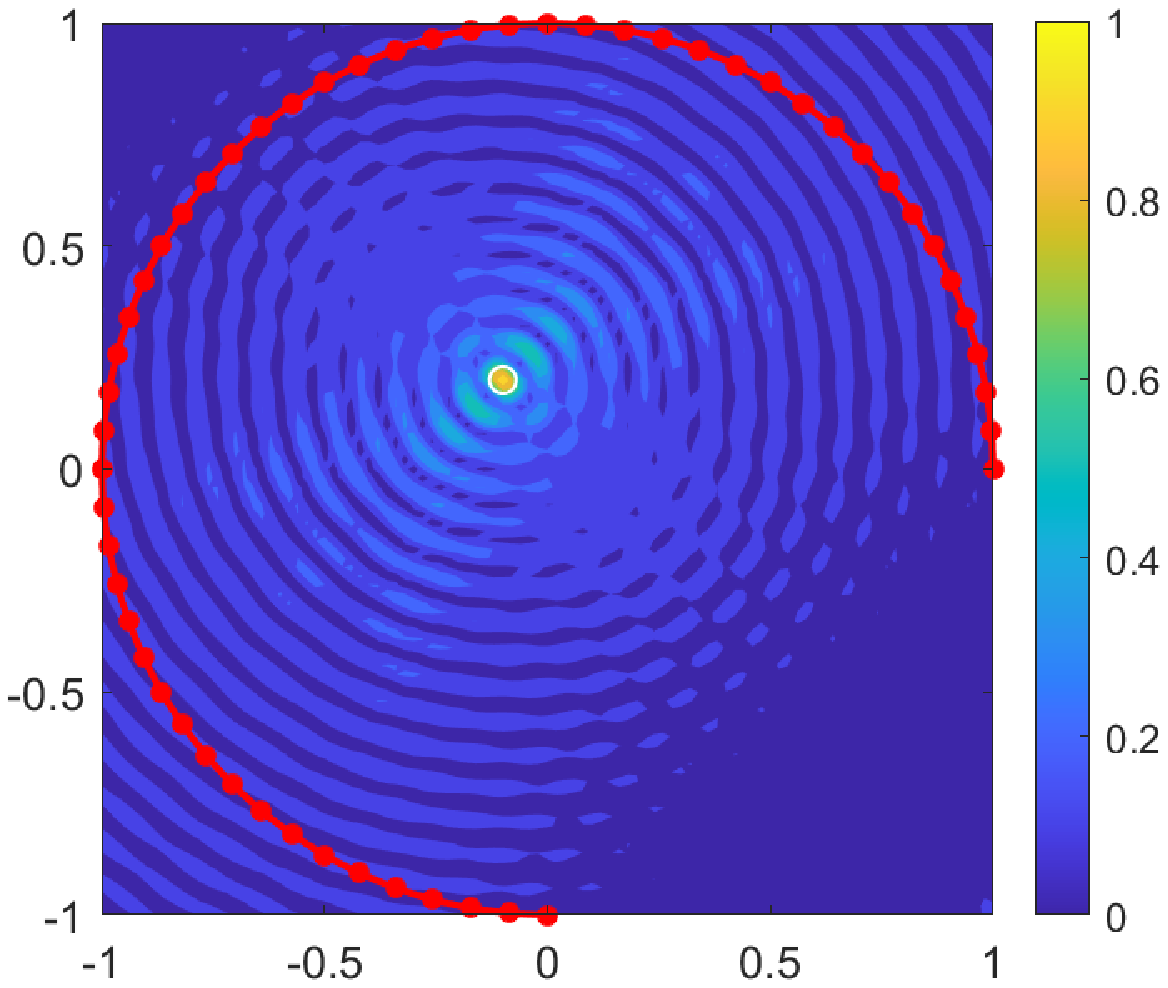}}
	\subfigure[\label{MMSM1-1}$\theta_{1}=0$ and $\theta_{N}=\frac{\pi}{2}$]{\includegraphics[width=0.32\textwidth]{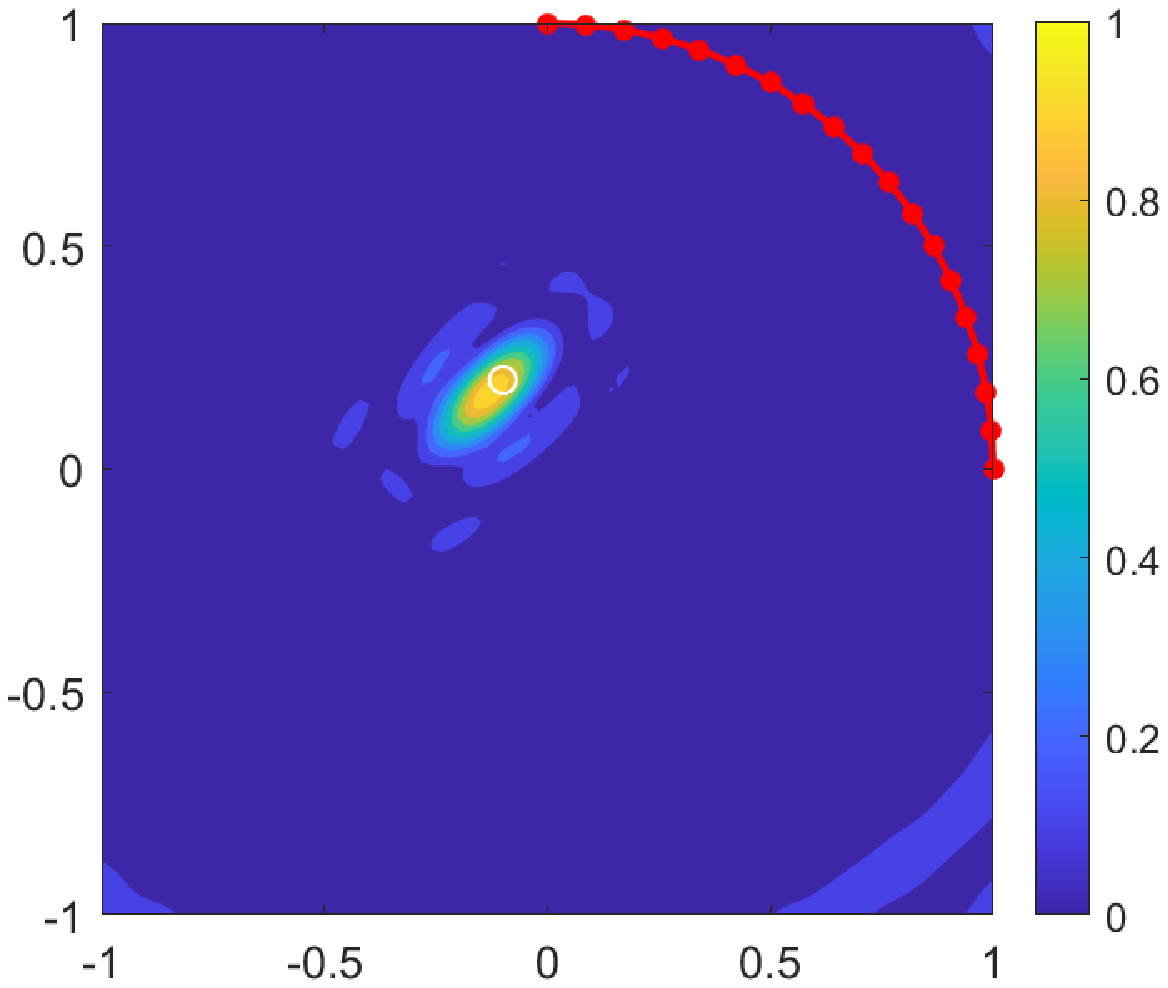}}
	\subfigure[\label{MMSM1-2}$\theta_{1}=0$ and $\theta_{N}=\pi$]{\includegraphics[width=0.32\textwidth]{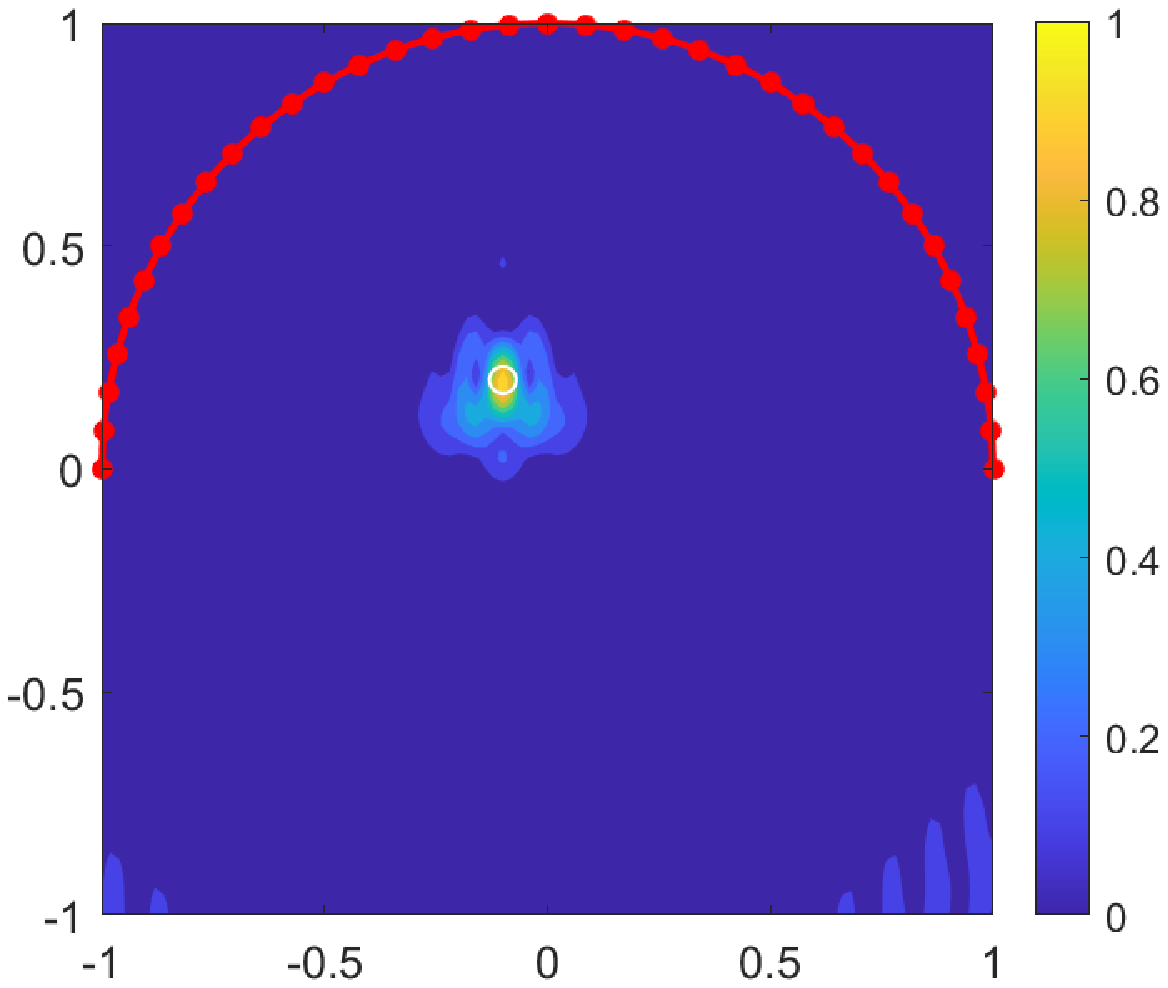}}
	\subfigure[\label{MMSM1-3}$\theta_{1}=0$ and $\theta_{N}=\frac{3}{2}\pi$]{\includegraphics[width=0.32\textwidth]{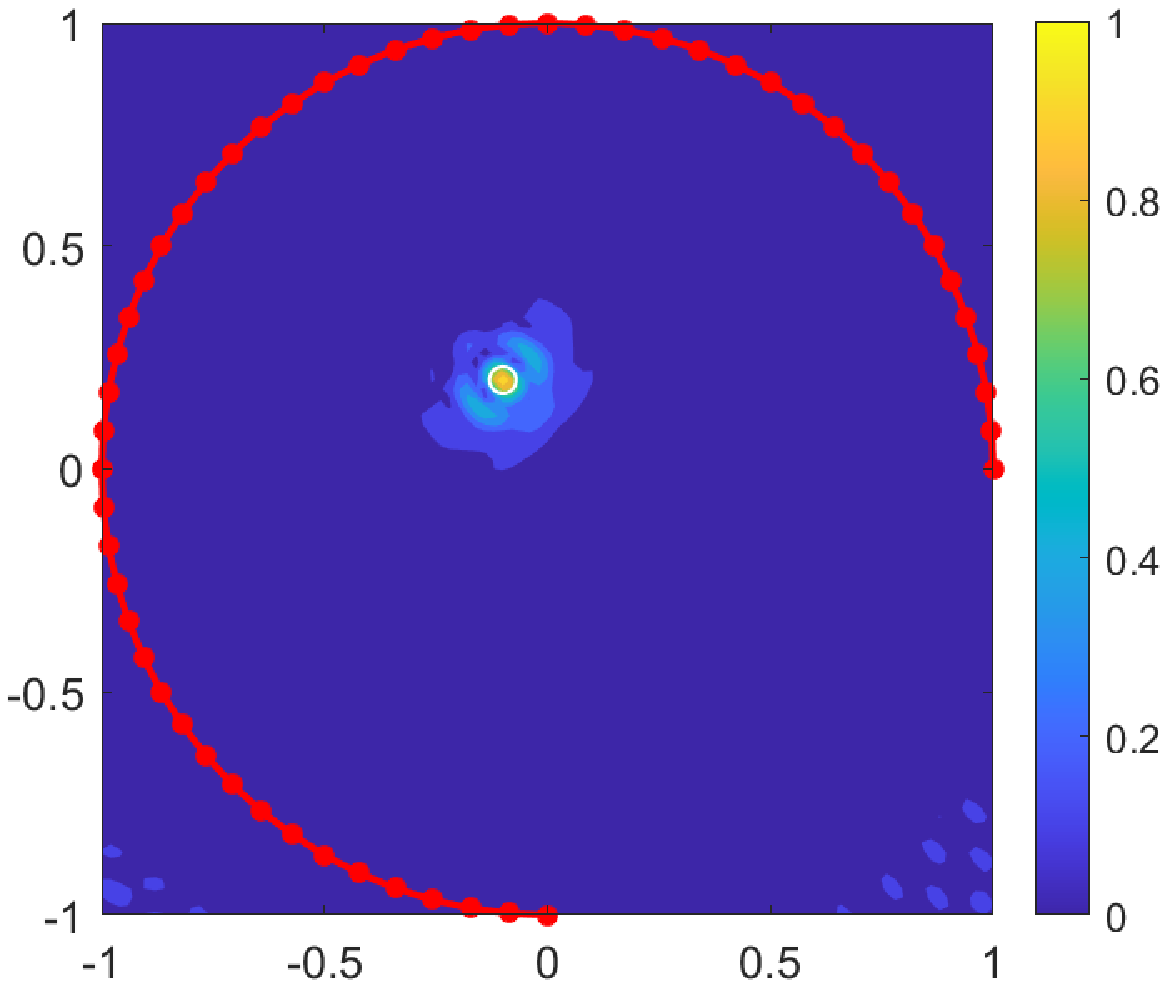}}
	\caption{MSM (top) and MMSM (bottom) for a single small inhomogeneity. The boundary of the target inhomogeneity is indicated in white, and the red dots indicate the measurement angles $\theta_n$. The MSM (with the measurement angle $\geq \pi$) and MMSM successfully detect the inhomogeneity. The MMSM shows less oscillation in the results.}
	\label{Ex1-MSMResult}
	%
	%
	\vskip 1cm
	%
	\centering
	\subfigure[\label{MSM2-1}$\theta_{1}=0$ and $\theta_{N}=\frac{\pi}{2}$]{\includegraphics[width=0.32\textwidth]{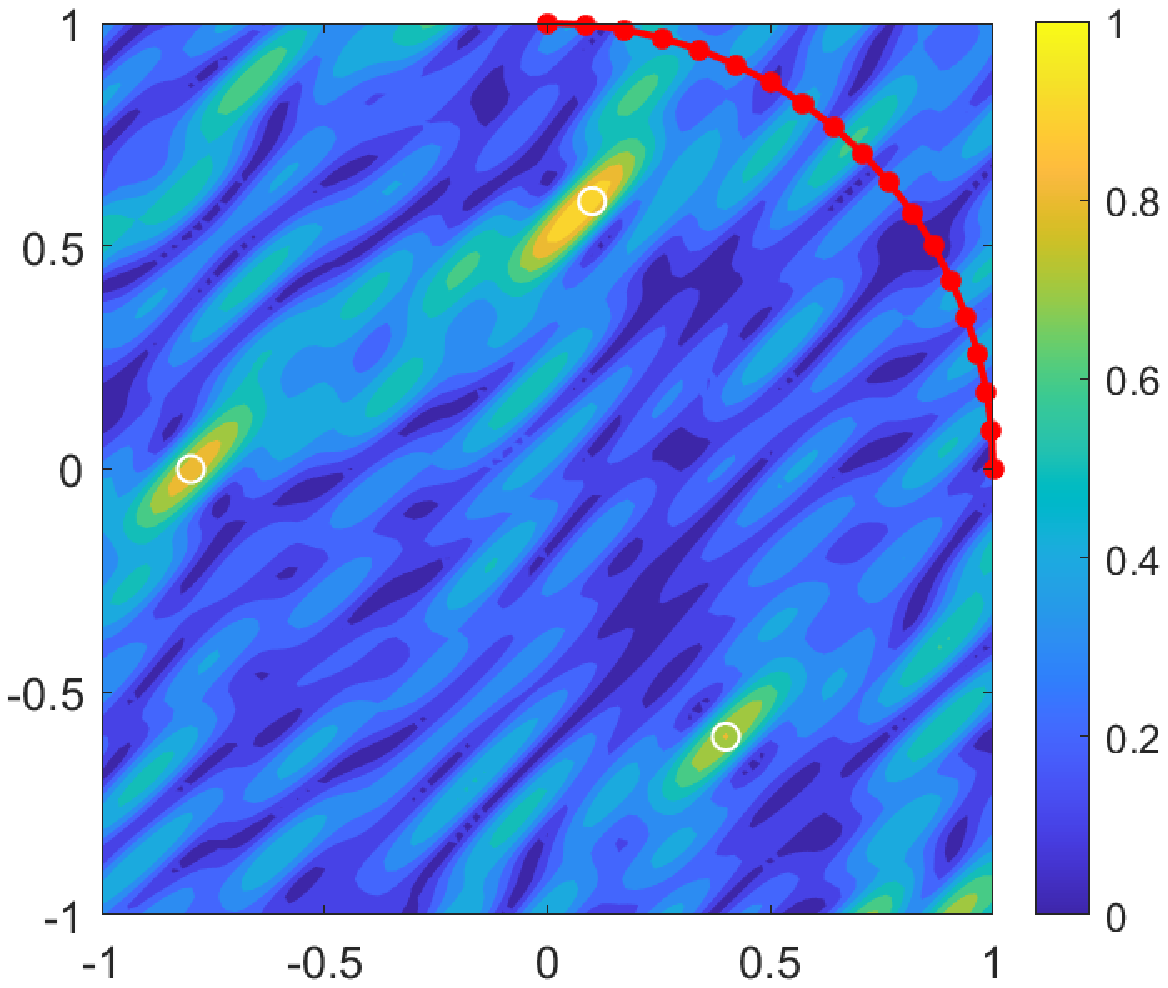}}
	\subfigure[\label{MSM2-2}$\theta_{1}=0$ and $\theta_{N}=\pi$]{\includegraphics[width=0.32\textwidth]{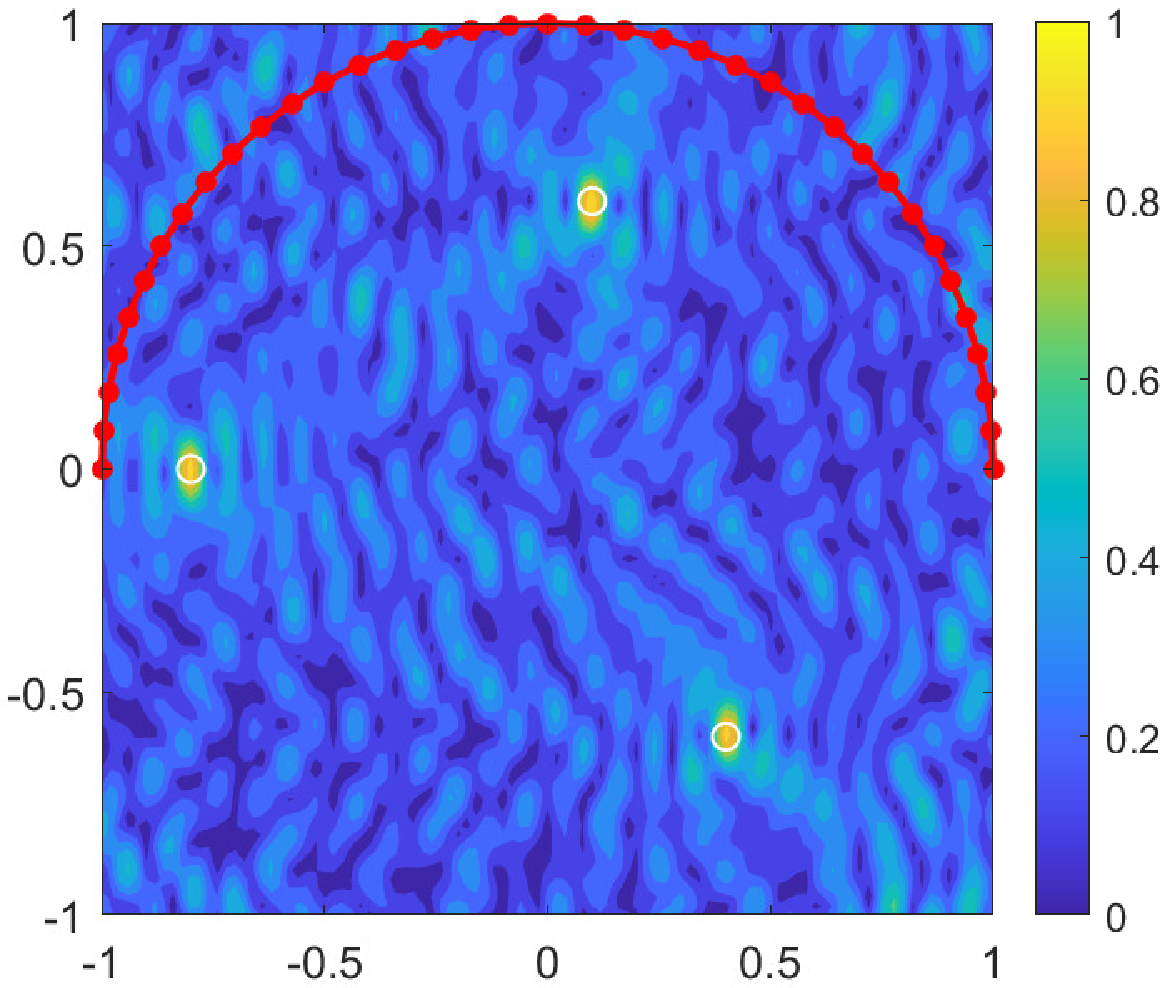}}
	\subfigure[\label{MSM2-3}$\theta_{1}=0$ and $\theta_{N}=\frac{3}{2}\pi$]{\includegraphics[width=0.32\textwidth]{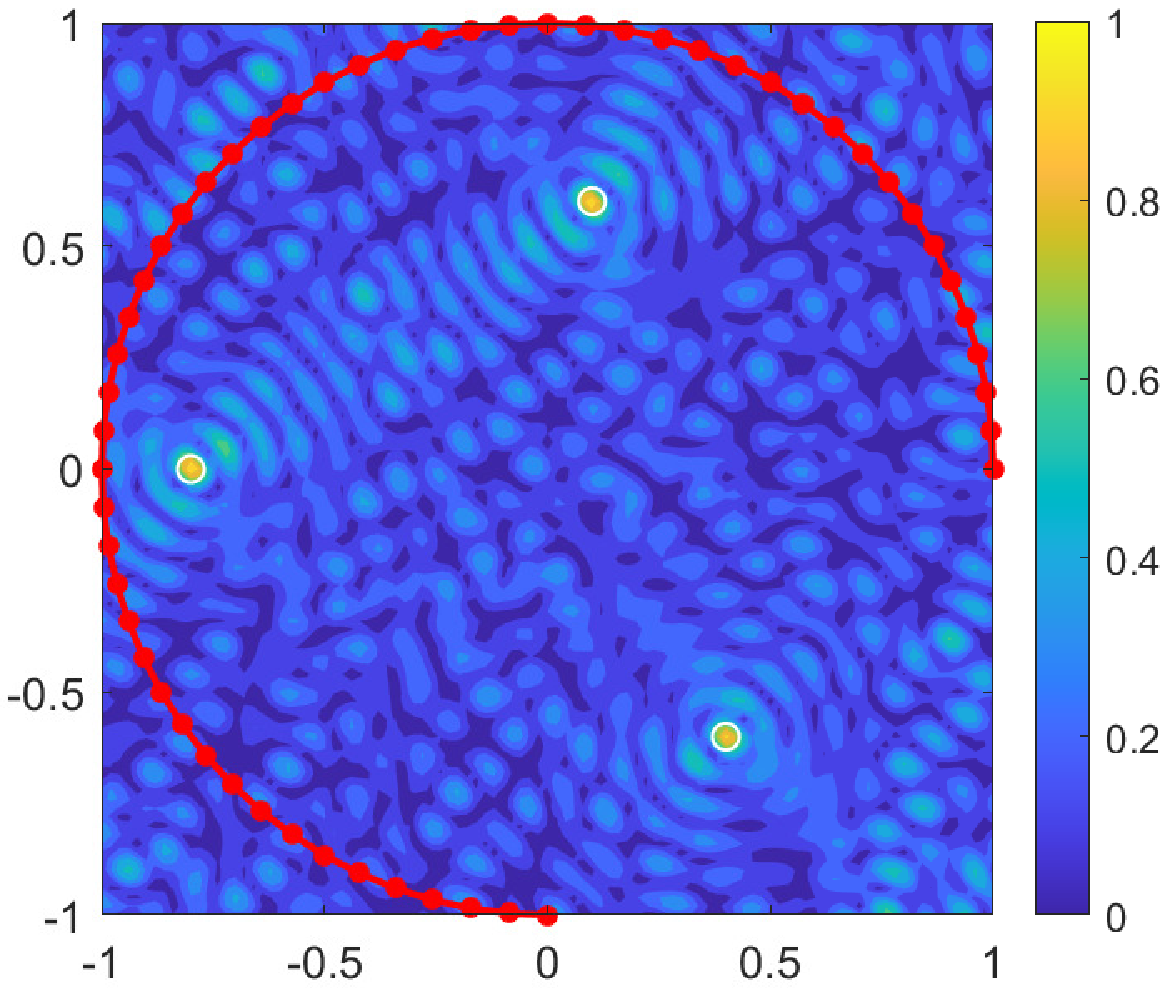}}
	\subfigure[\label{MMSM2-1}$\theta_{1}=0$ and $\theta_{N}=\frac{\pi}{2}$]{\includegraphics[width=0.32\textwidth]{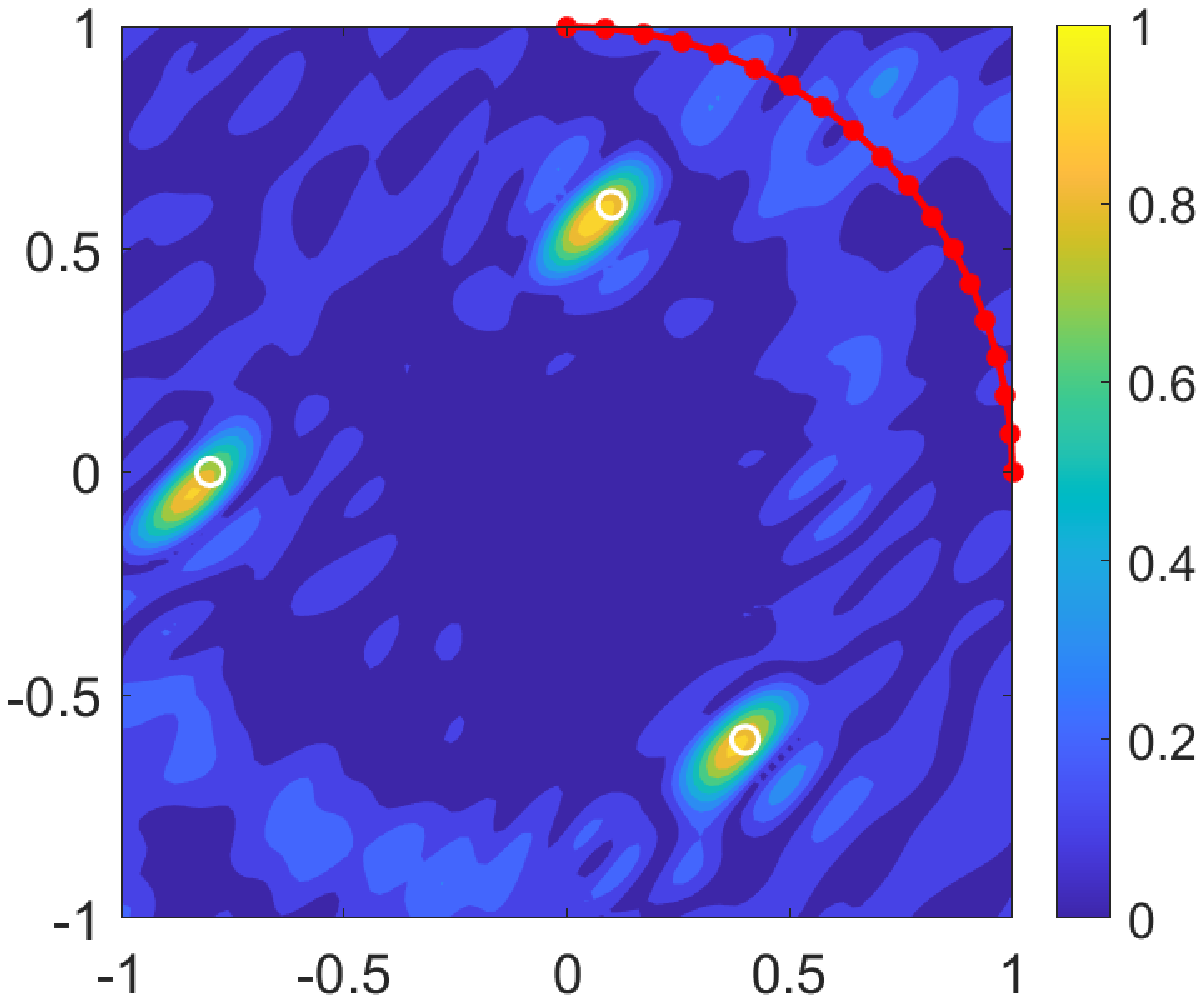}}
	\subfigure[\label{MMSM2-2}$\theta_{1}=0$ and $\theta_{N}=\pi$]{\includegraphics[width=0.32\textwidth]{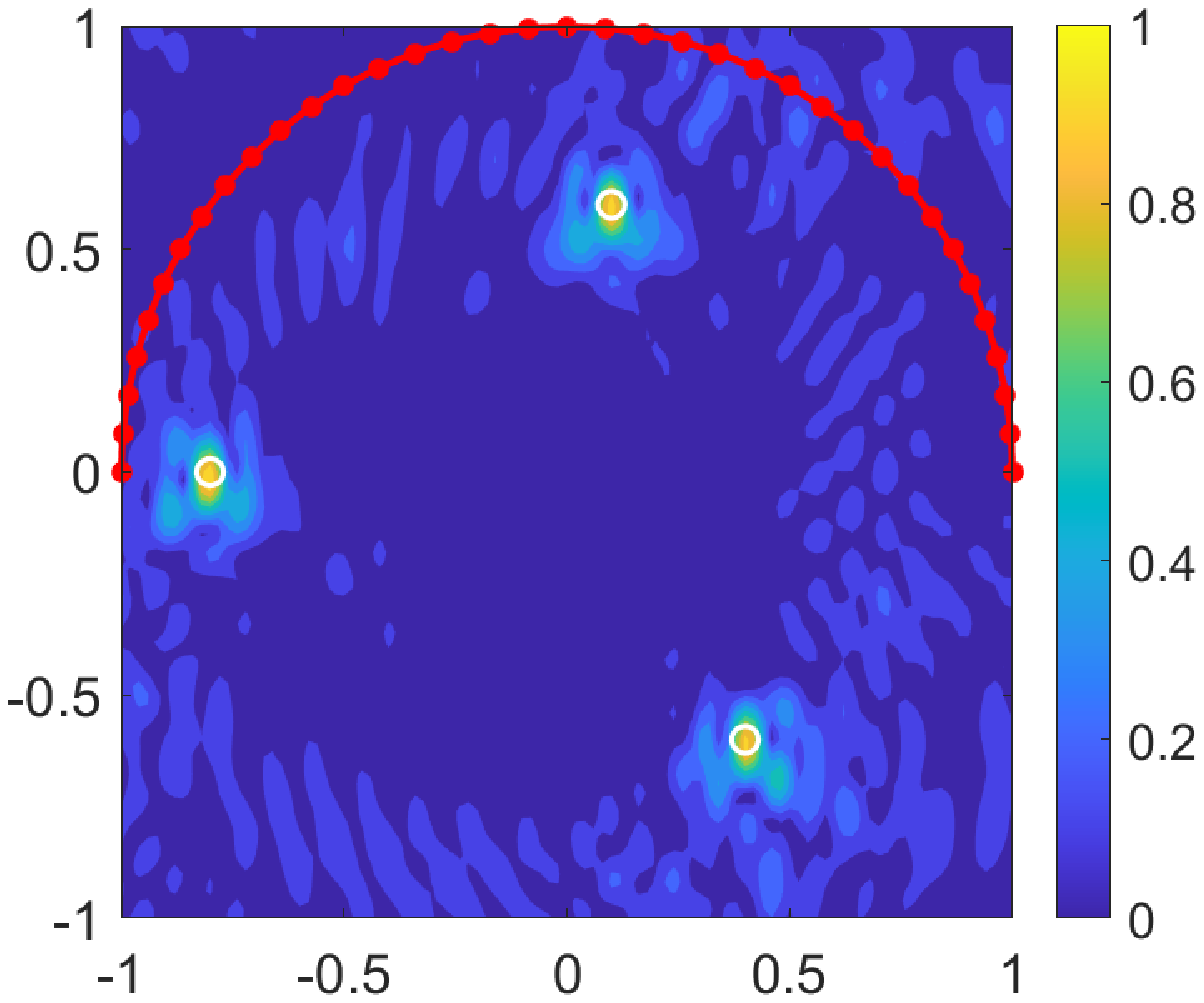}}
	\subfigure[\label{MMSM2-3}$\theta_{1}=0$ and $\theta_{N}=\frac{3}{2}\pi$]{\includegraphics[width=0.32\textwidth]{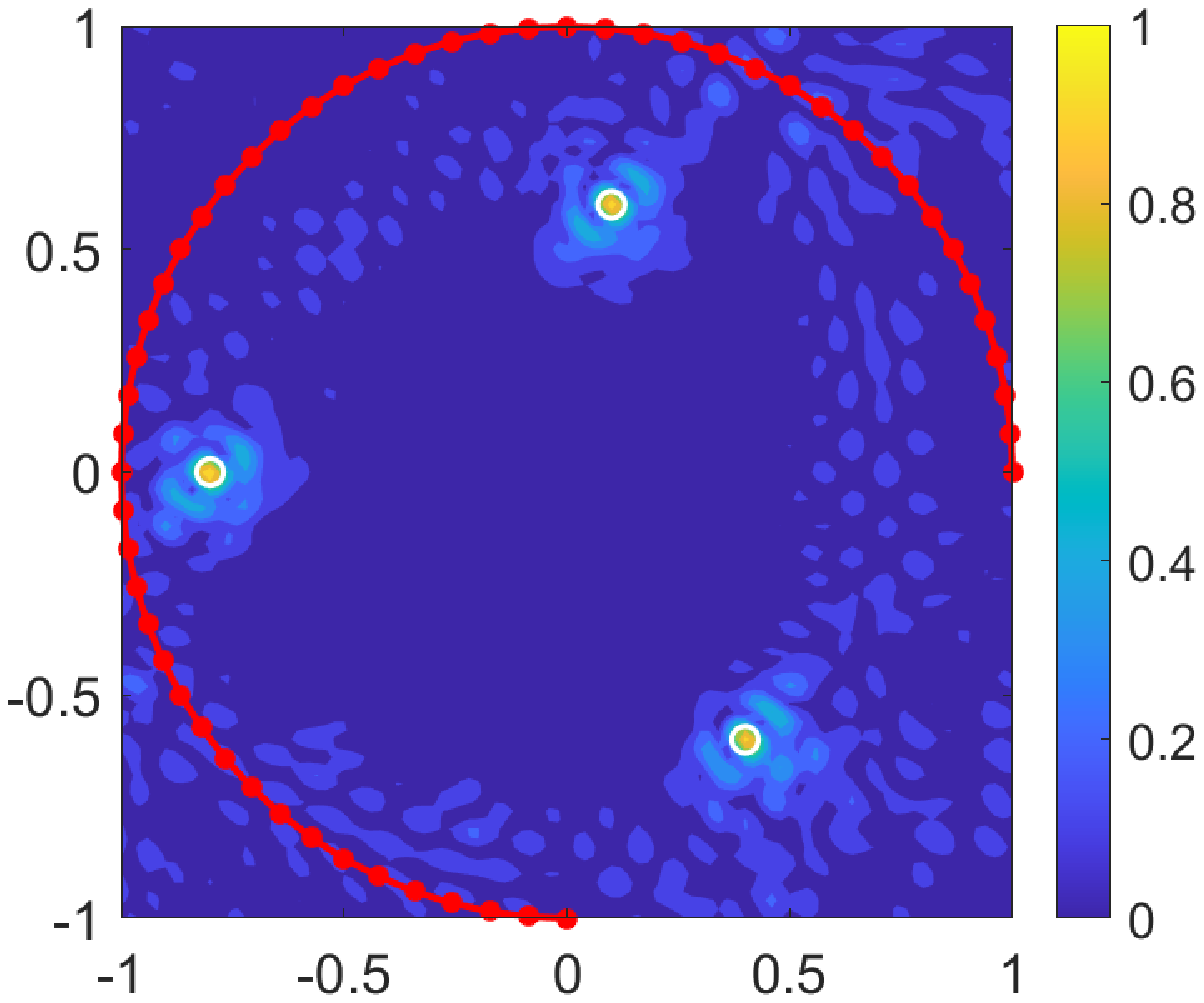}}
	\caption{MSM (top) and MMSM (bottom) for multiple small inhomogeneities. The MMSM clearly detects all three inhomogeneities.}
	\label{Ex2-MSMResult}
\end{figure}

\begin{figure}[h!]
	\centering
	\subfigure[\label{MSM2-1-1}$\theta_{1}=0$ and $\theta_{N}=\frac{\pi}{2}$]{\includegraphics[width=0.32\textwidth]{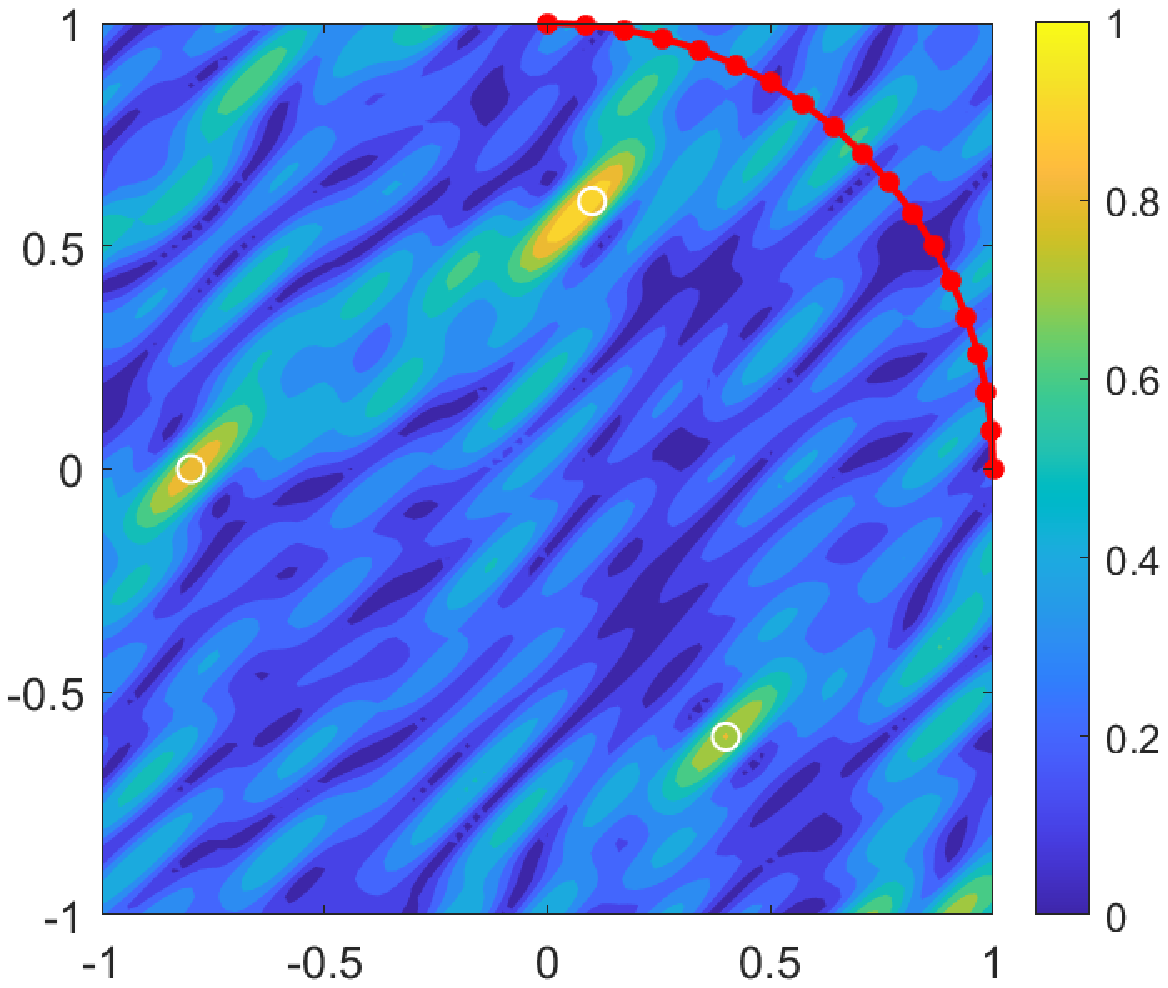}}
	\subfigure[\label{MSM2-1-2}$\theta_{1}=0$ and $\theta_{N}=\pi$]{\includegraphics[width=0.32\textwidth]{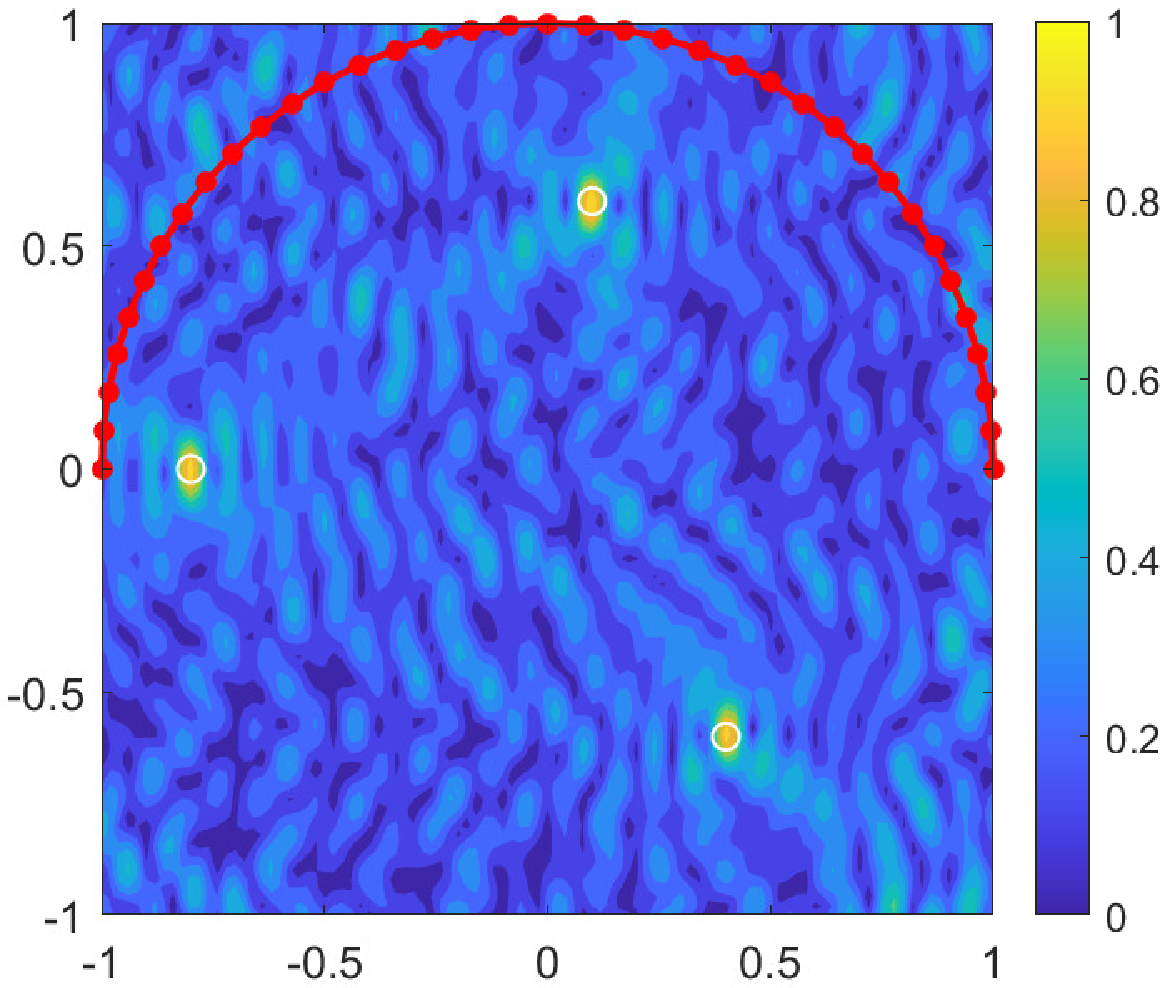}}
	\subfigure[\label{MSM2-1-3}$\theta_{1}=0$ and $\theta_{N}=\frac{3}{2}\pi$]{\includegraphics[width=0.32\textwidth]{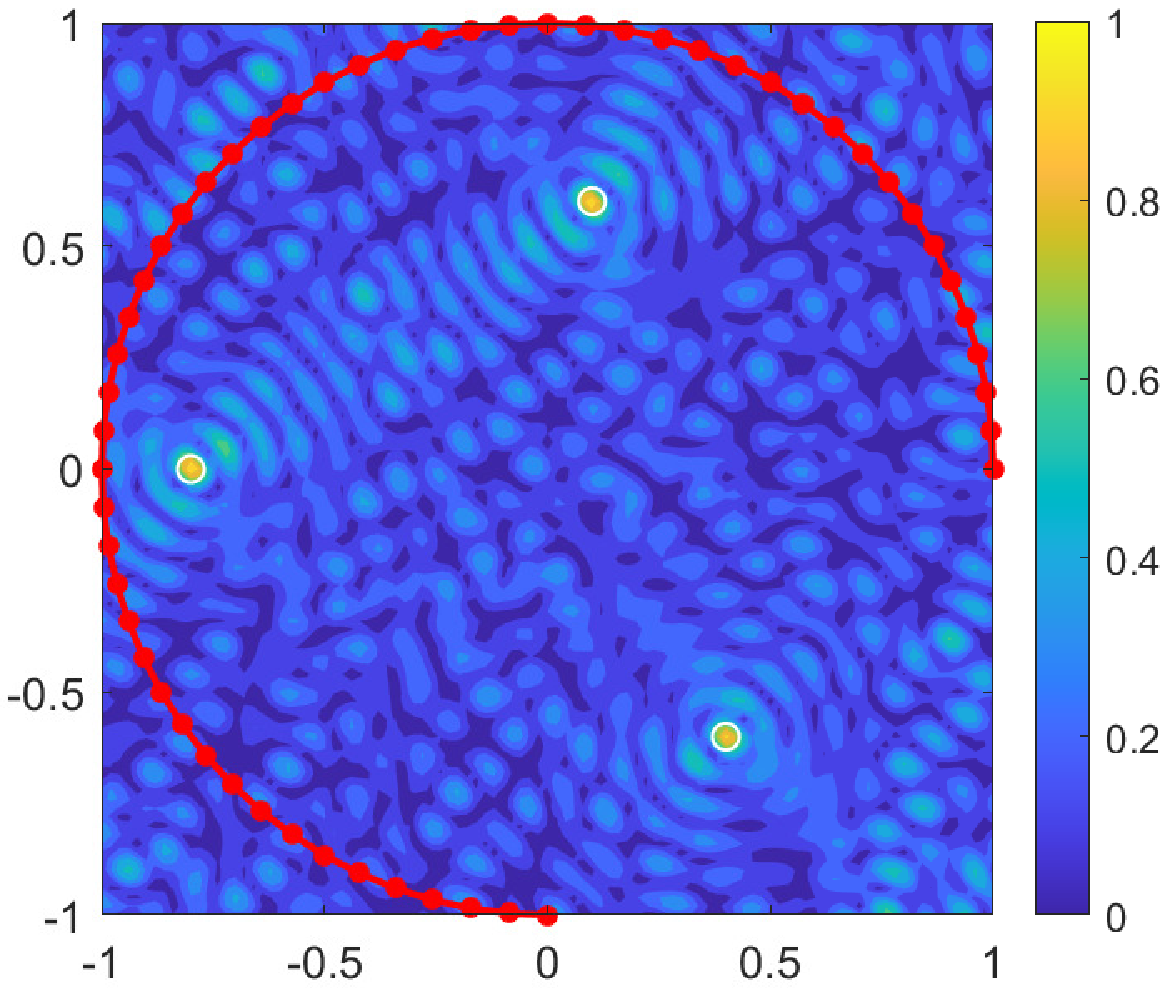}}
	\subfigure[\label{MMSM2-1-1}$\theta_{1}=0$ and $\theta_{N}=\frac{\pi}{2}$]{\includegraphics[width=0.32\textwidth]{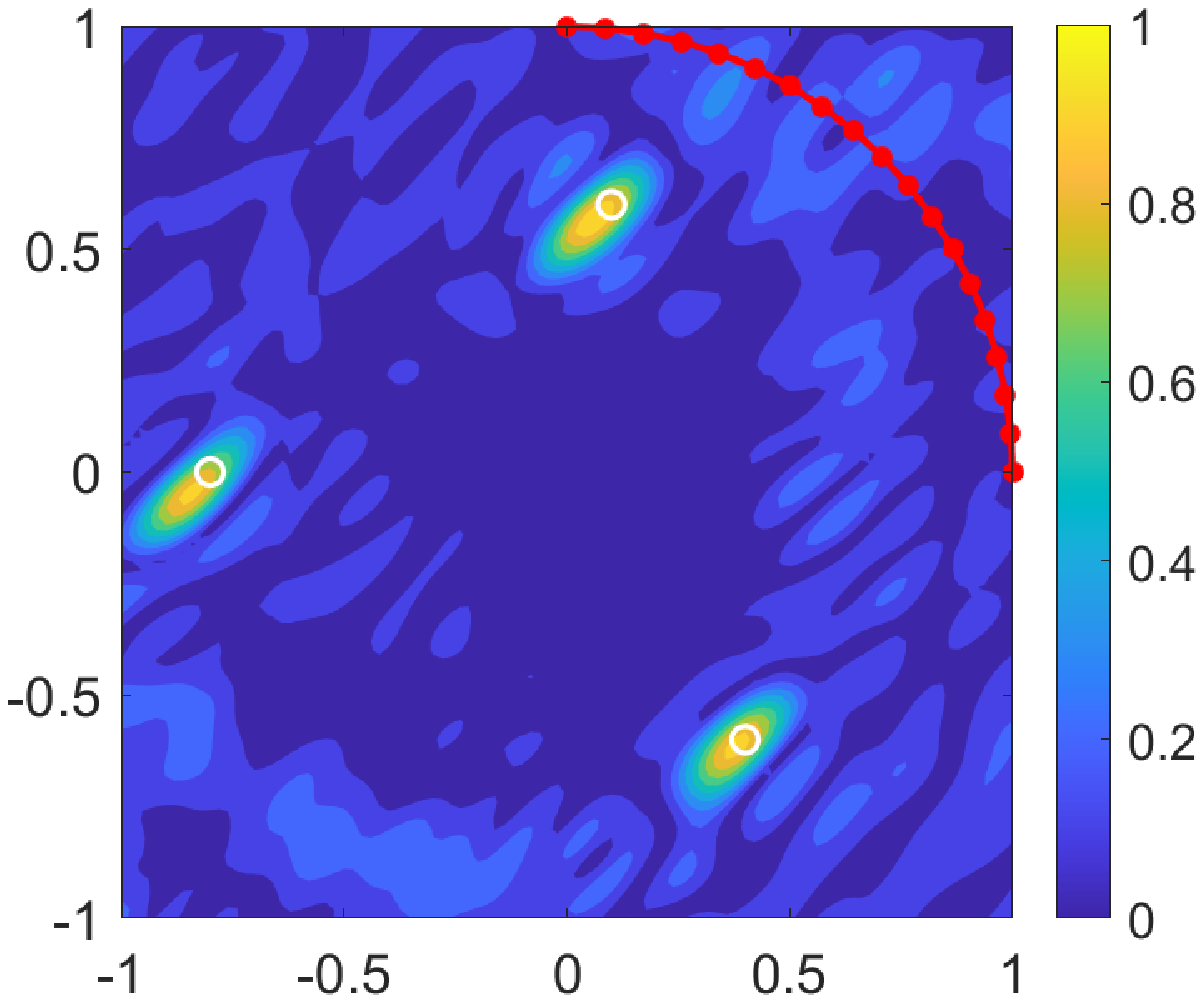}}
	\subfigure[\label{MMSM2-1-2}$\theta_{1}=0$ and $\theta_{N}=\pi$]{\includegraphics[width=0.32\textwidth]{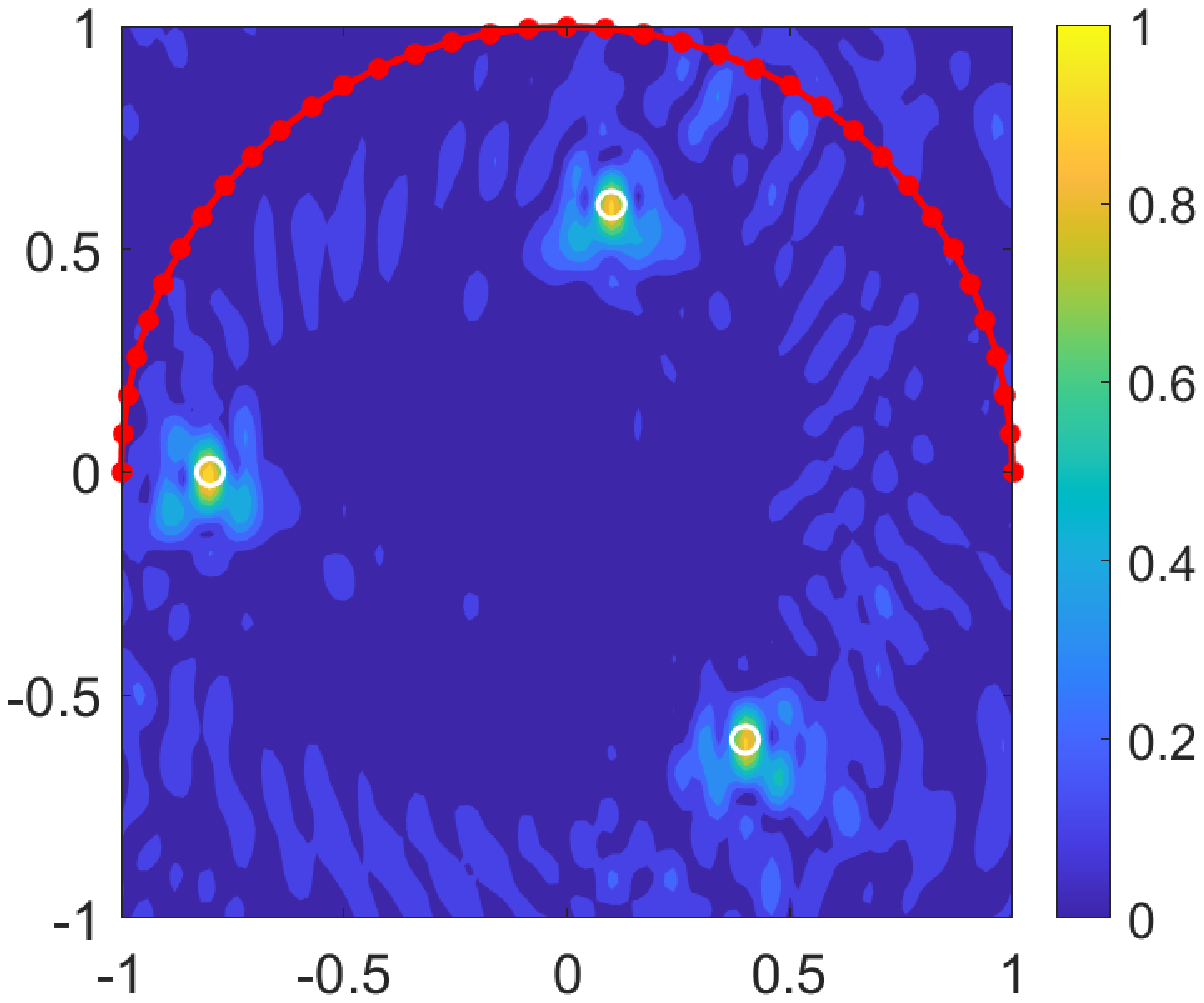}}
	\subfigure[\label{MMSM2-1-3}$\theta_{1}=0$ and $\theta_{N}=\frac{3}{2}\pi$]{\includegraphics[width=0.32\textwidth]{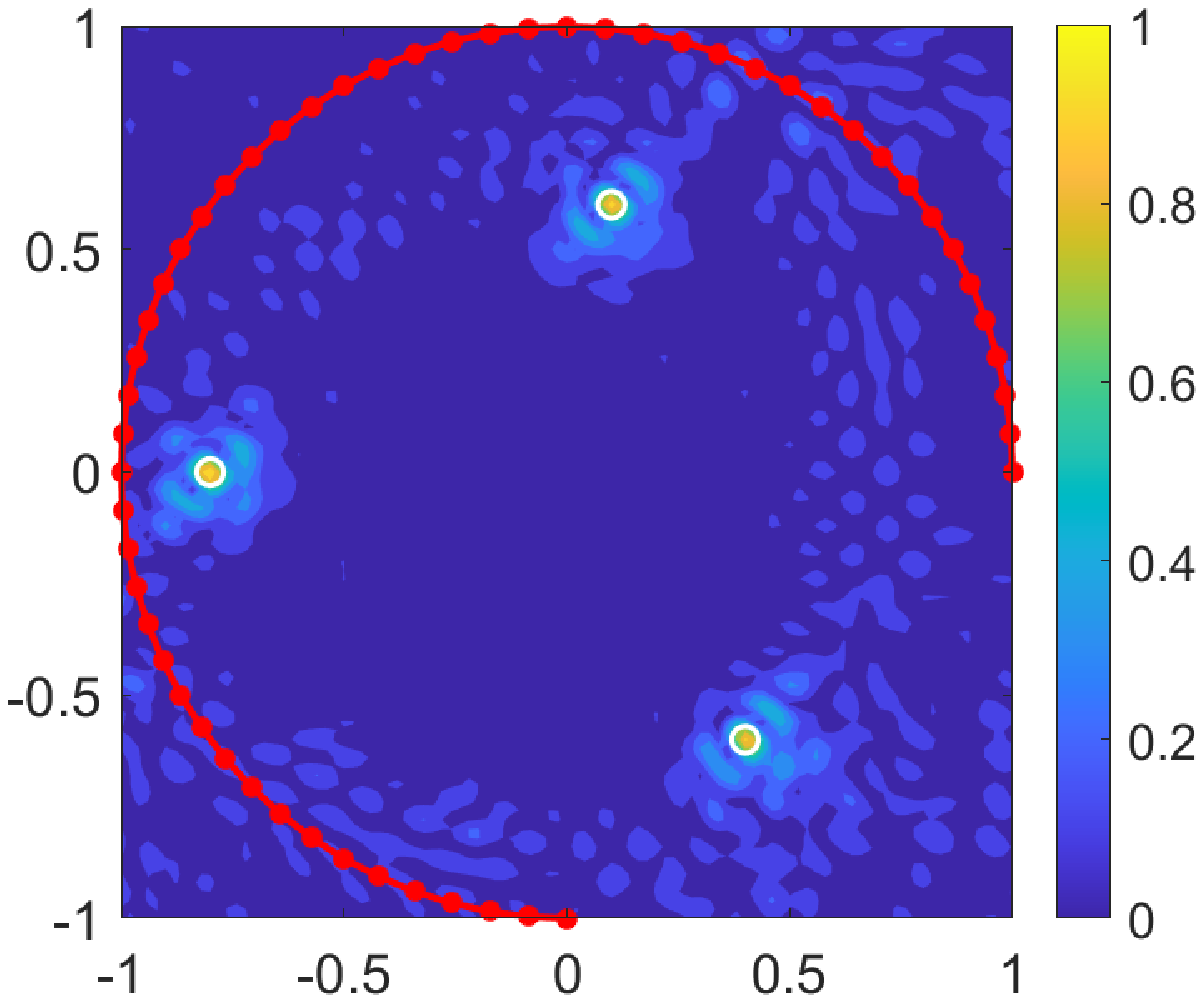}}
	\caption{MSM (top) and MMSM (bottom) for multiple small inhomogeneities with \SI{10}{\dB} Gaussian random noise. The MSM and MMSM have the similar imaging performance compared to Figure \ref{Ex2-MSMResult} (\SI{20}{\dB} noise)}
	\label{Ex2-1-MSMResult}
	%
	%
	\vskip 1cm
	
	\centering
	\subfigure[\label{MSM3-1}$\theta_{1}=0$ and $\theta_{N}=\frac{\pi}{2}$]{\includegraphics[width=0.32\textwidth]{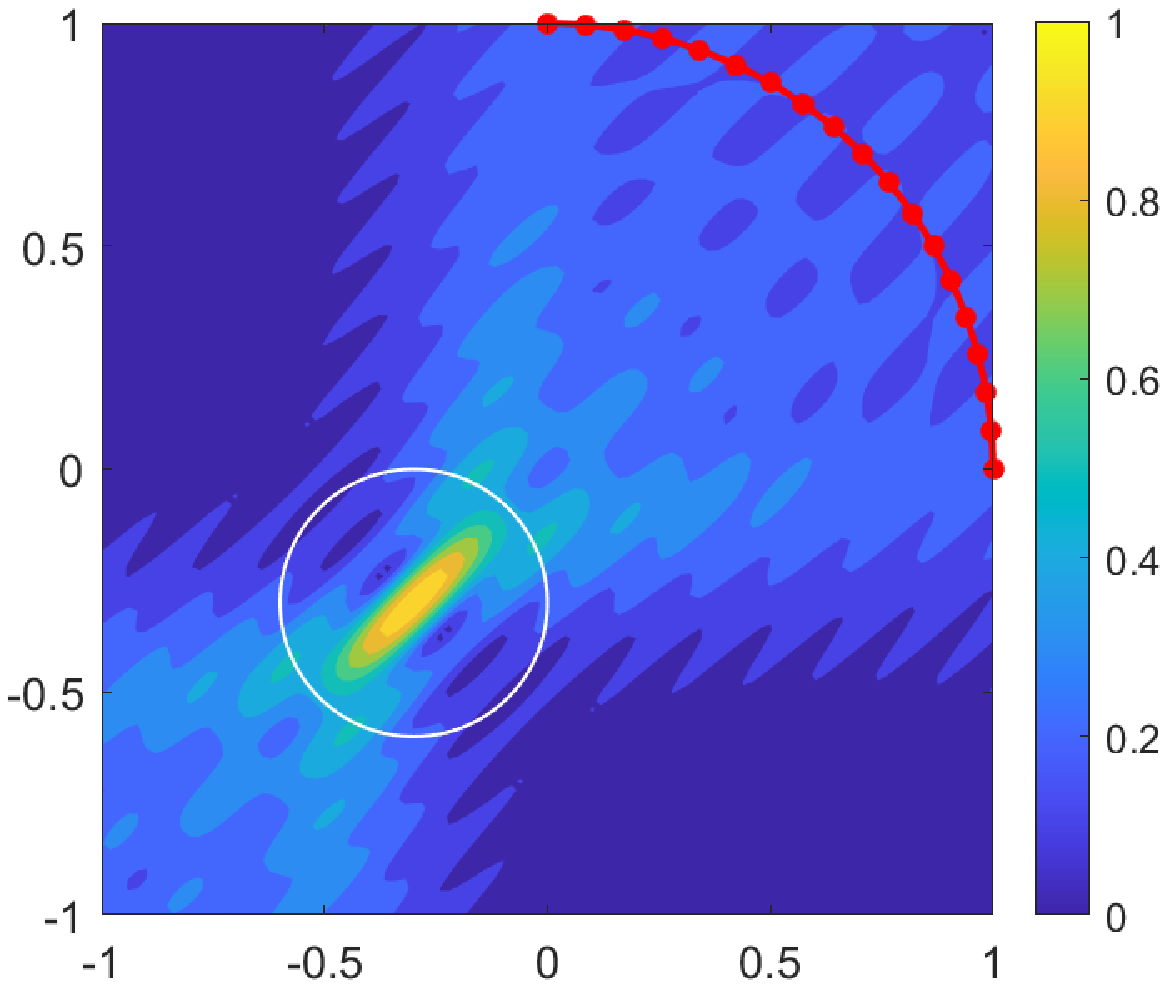}}
	\subfigure[\label{MSM3-2}$\theta_{1}=0$ and $\theta_{N}=\pi$]{\includegraphics[width=0.32\textwidth]{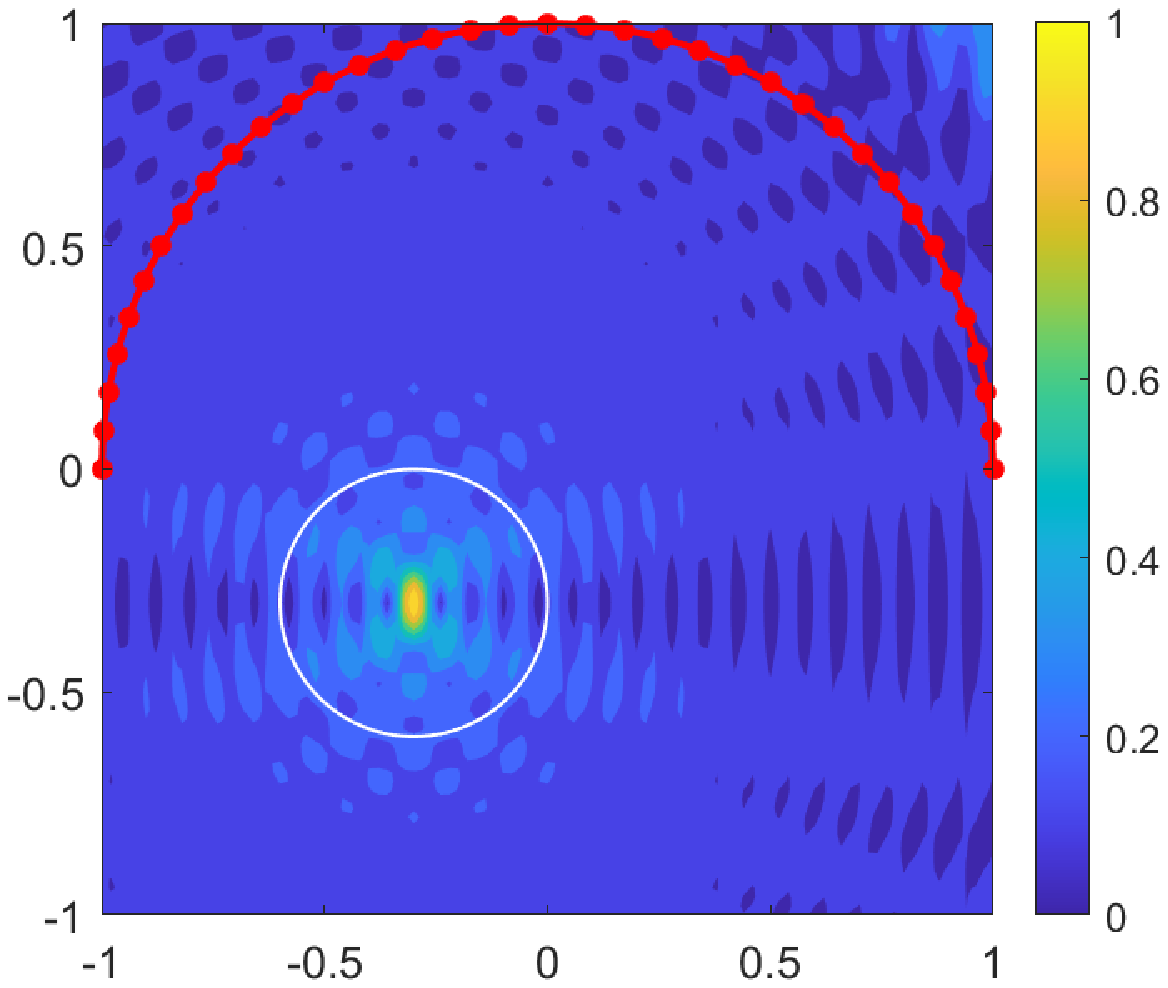}}
	\subfigure[\label{MSM3-3}$\theta_{1}=0$ and $\theta_{N}=\frac{3}{2}\pi$]{\includegraphics[width=0.32\textwidth]{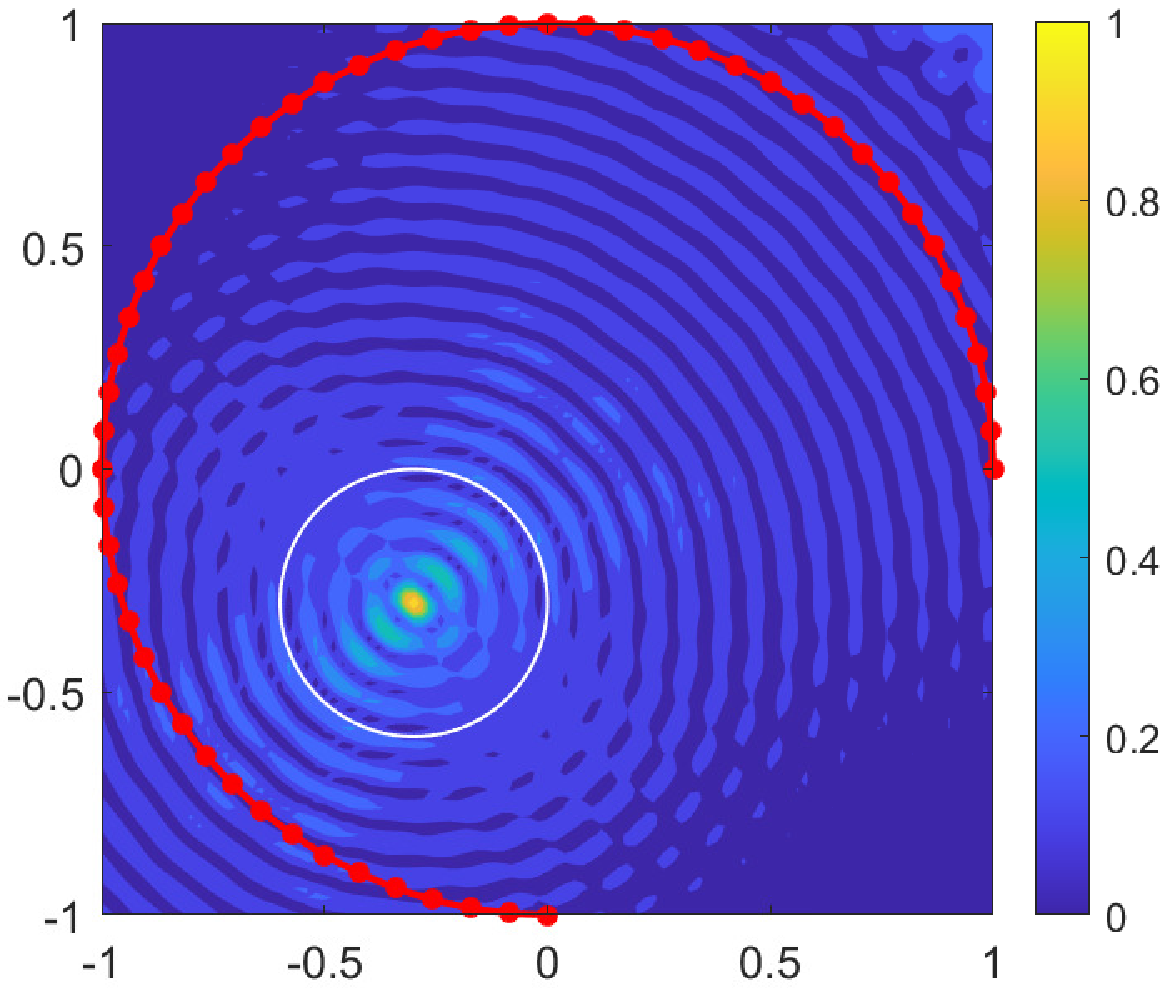}}
	\subfigure[\label{MMSM3-1}$\theta_{1}=0$ and $\theta_{N}=\frac{\pi}{2}$]{\includegraphics[width=0.32\textwidth]{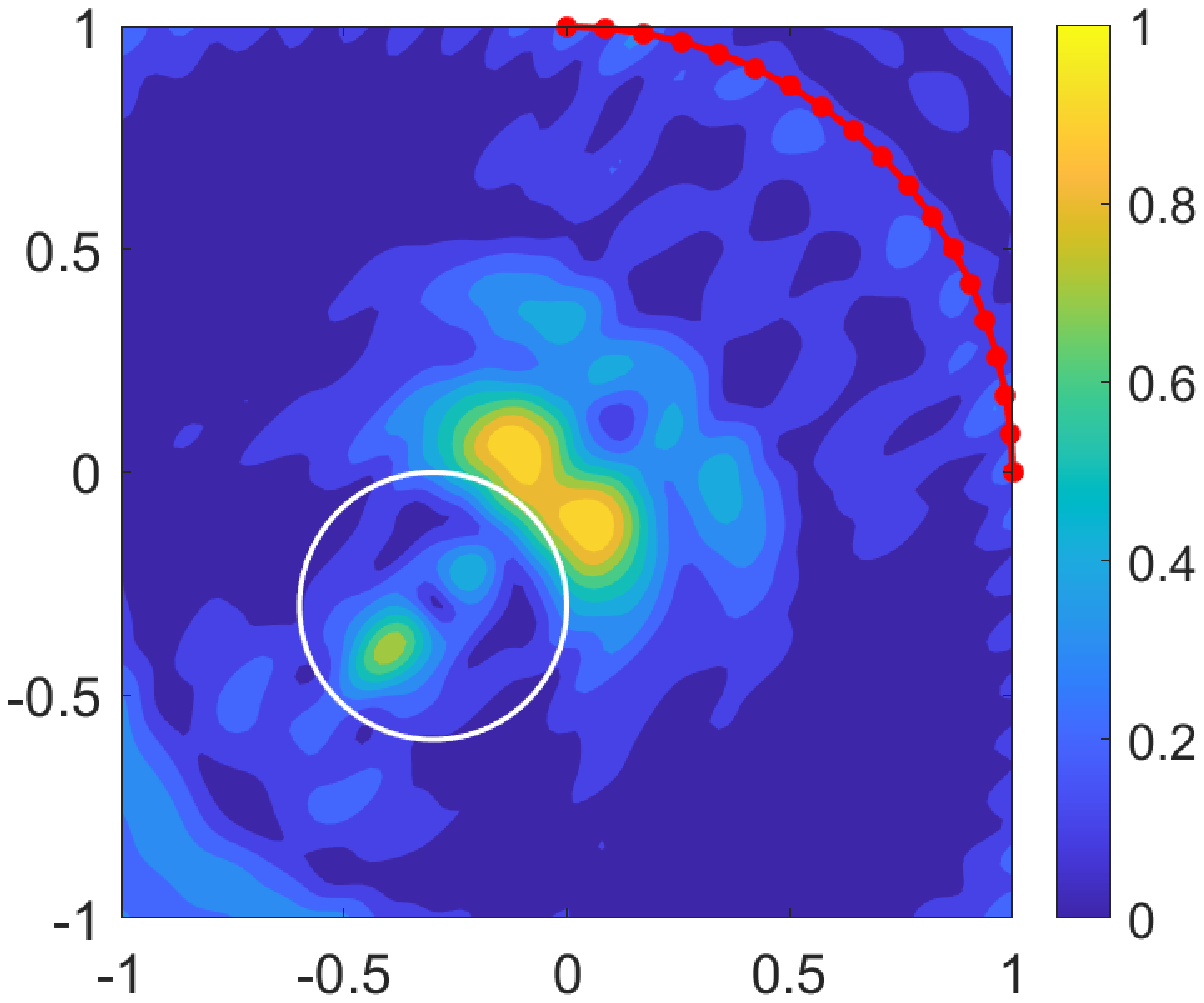}}
	\subfigure[\label{MMSM3-2}$\theta_{1}=0$ and $\theta_{N}=\pi$]{\includegraphics[width=0.32\textwidth]{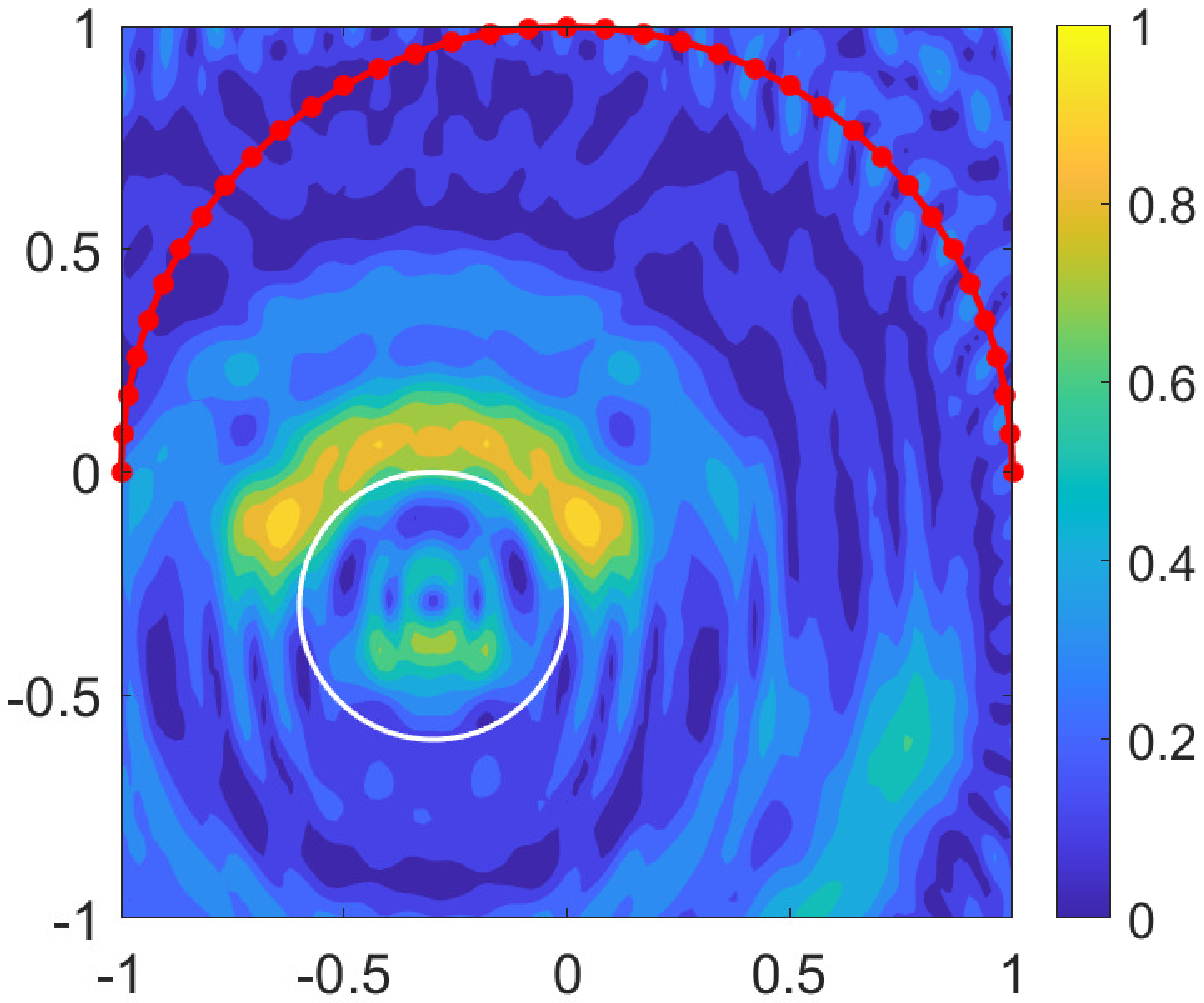}}
	\subfigure[\label{MMSM3-3}$\theta_{1}=0$ and $\theta_{N}=\frac{3}{2}\pi$]{\includegraphics[width=0.32\textwidth]{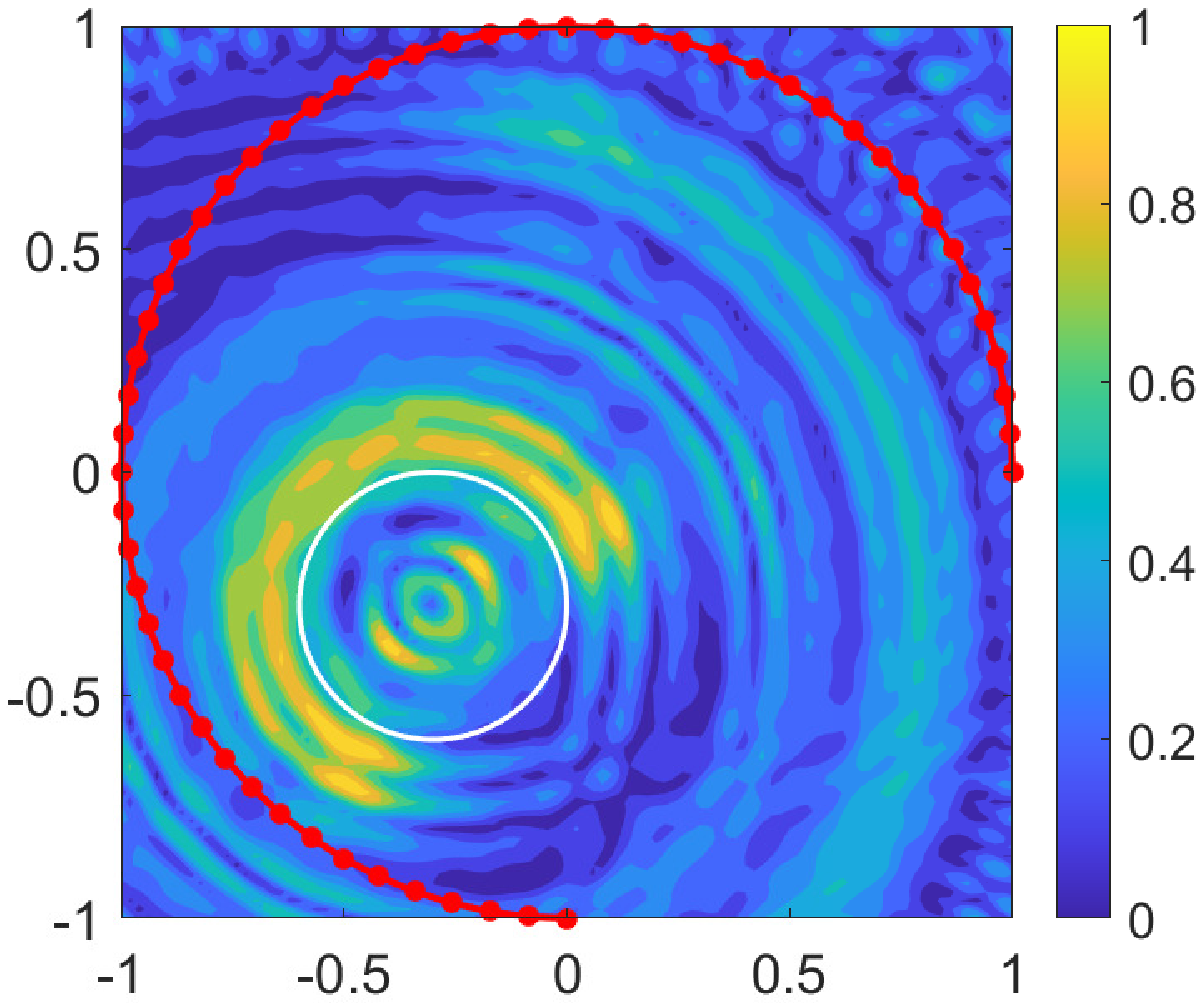}}
	\caption{MSM (top) and MMSM (bottom) for an extended target. The MSM (with single frequency measurement) can identify the center of the inhomogeneity.}
	\label{Ex3-MSMResult}
\end{figure}

\begin{figure}[h!]
	\centering
	\subfigure[\label{DSM1-1}$\theta_{1}=0$ and $\theta_{N}=\frac{\pi}{2}$]{\includegraphics[width=0.32\textwidth]{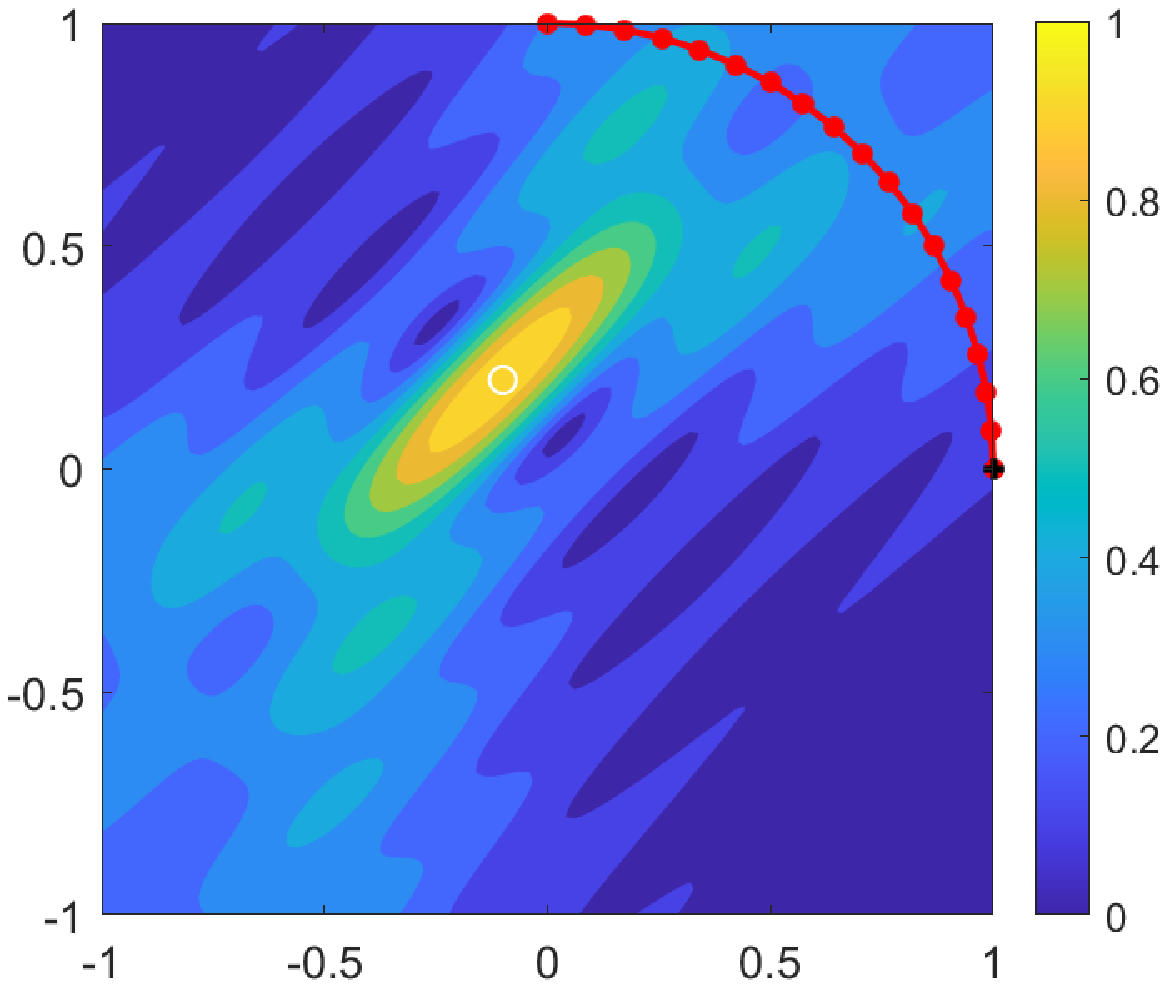}}
	\subfigure[\label{DSM1-2}$\theta_{1}=0$ and $\theta_{N}=\pi$]{\includegraphics[width=0.32\textwidth]{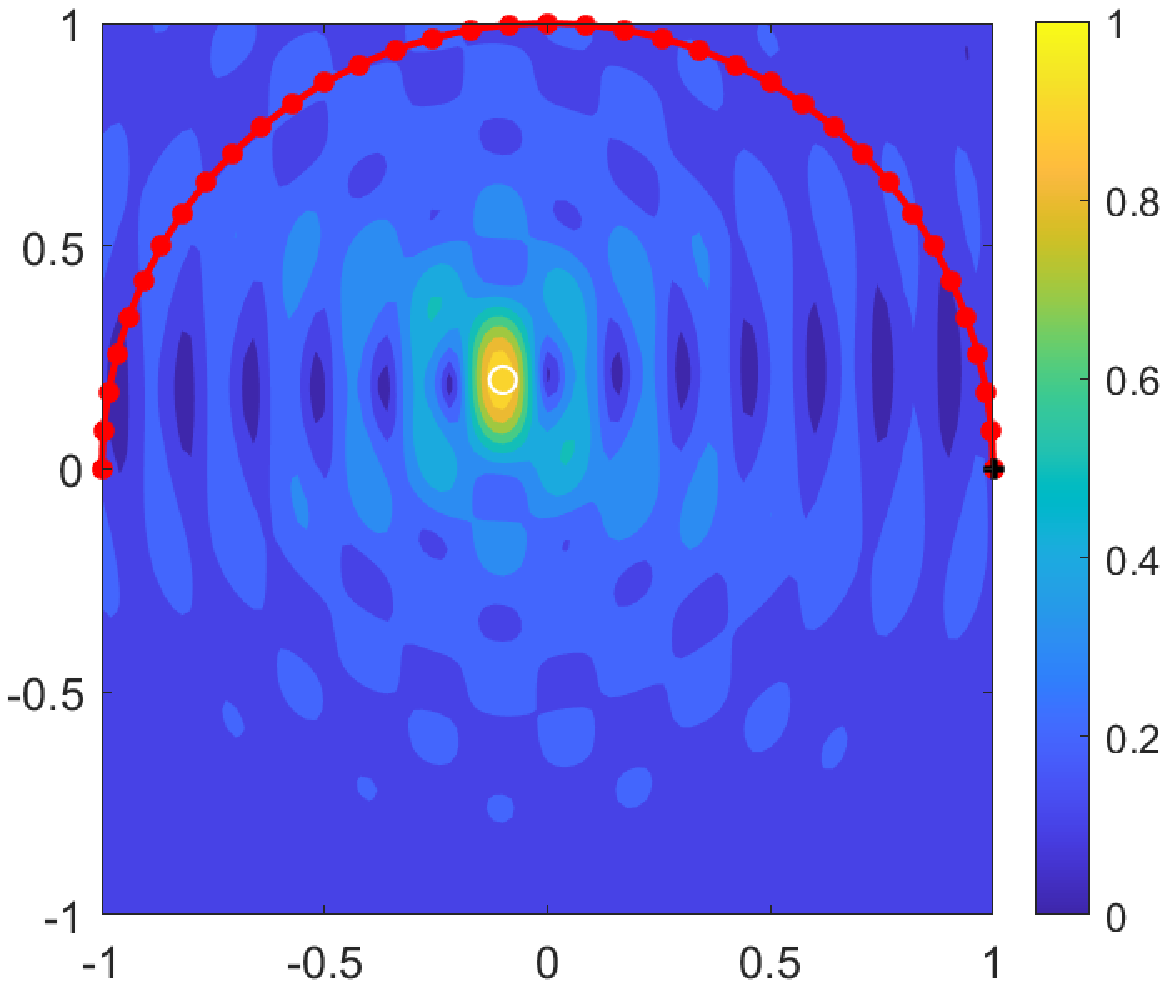}}
	\subfigure[\label{DSM1-3}$\theta_{1}=0$ and $\theta_{N}=\frac{3}{2}\pi$]{\includegraphics[width=0.32\textwidth]{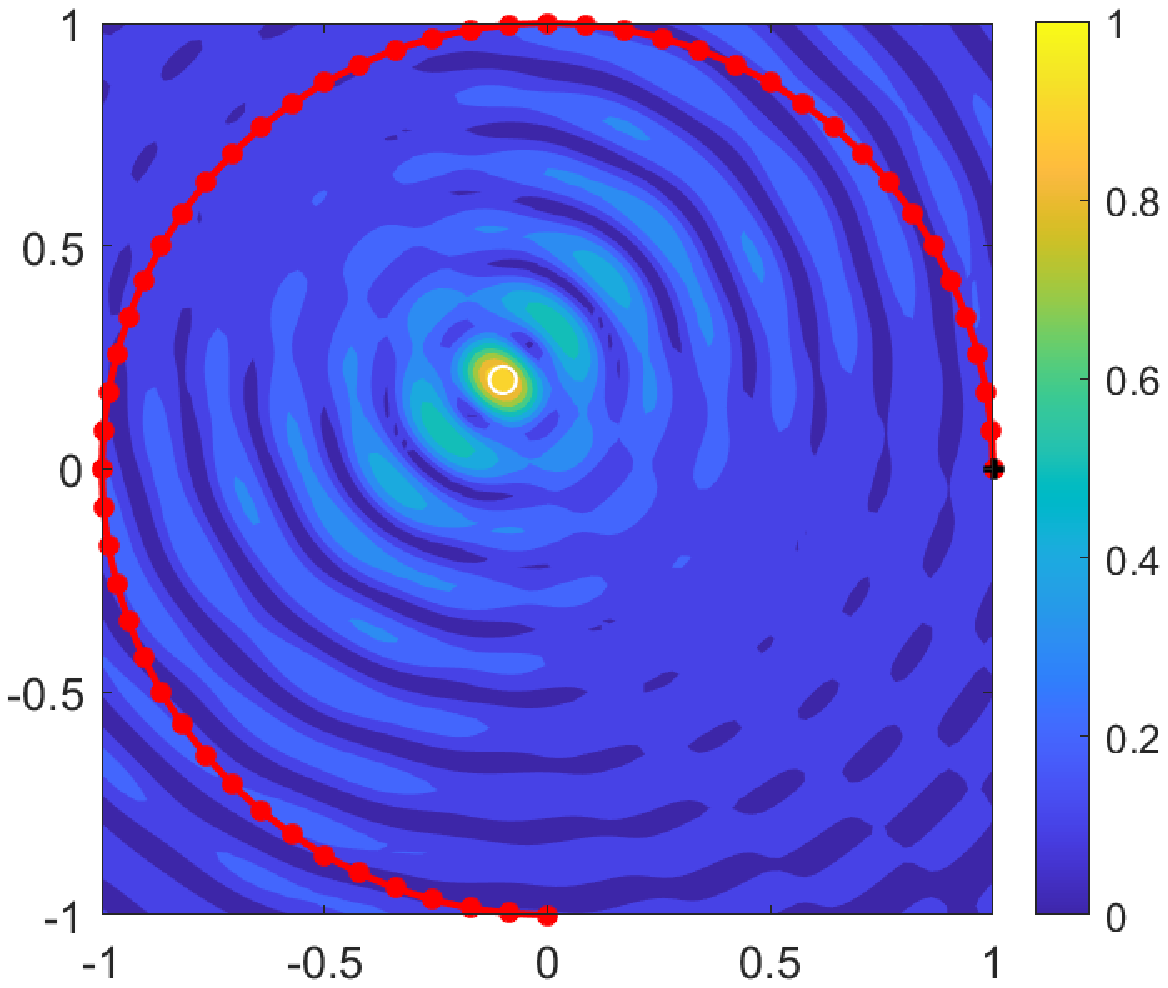}}
	\subfigure[\label{MDSM1-1}$\theta_{1}=0$ and $\theta_{N}=\frac{\pi}{2}$]{\includegraphics[width=0.32\textwidth]{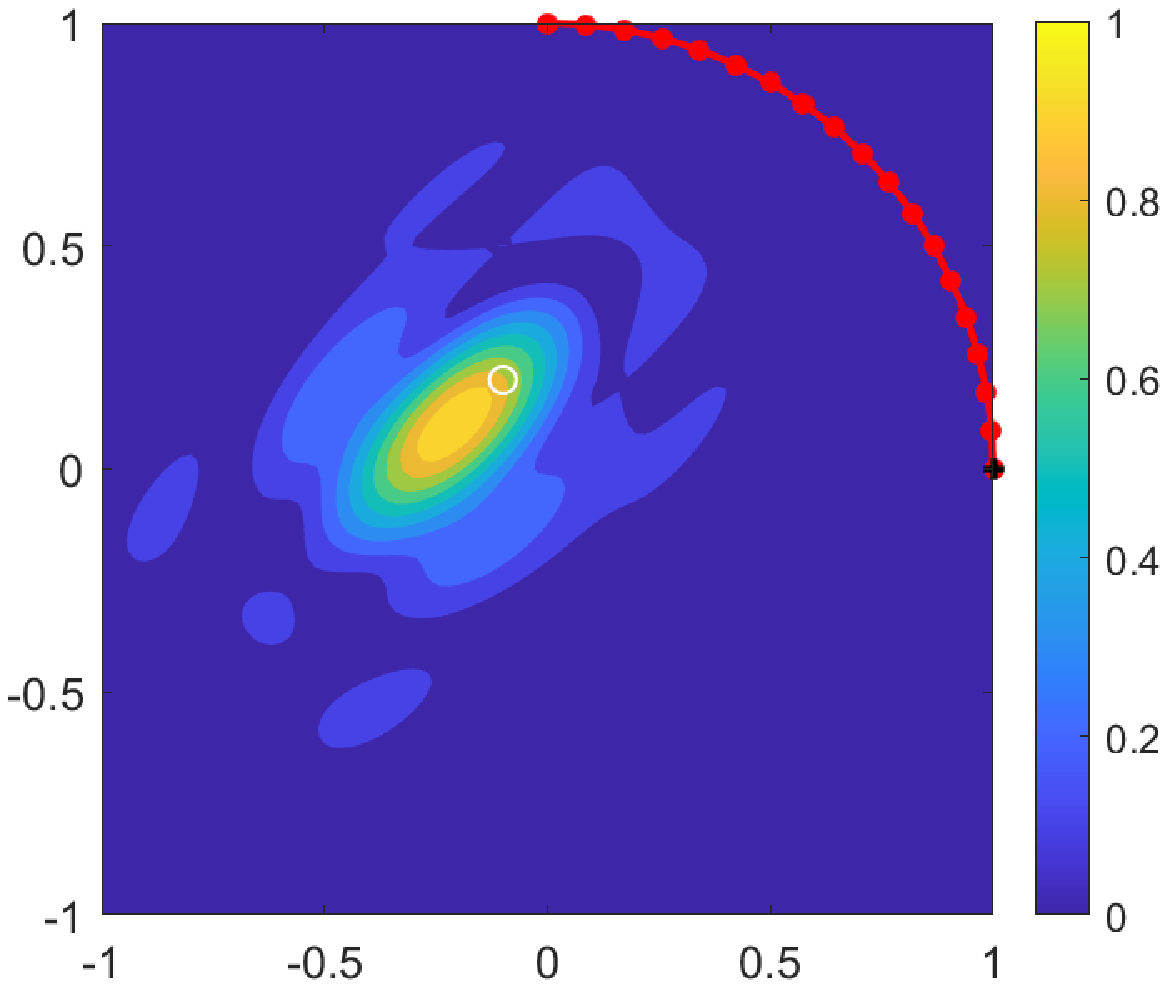}}
	\subfigure[\label{MDSM1-2}$\theta_{1}=0$ and $\theta_{N}=\pi$]{\includegraphics[width=0.32\textwidth]{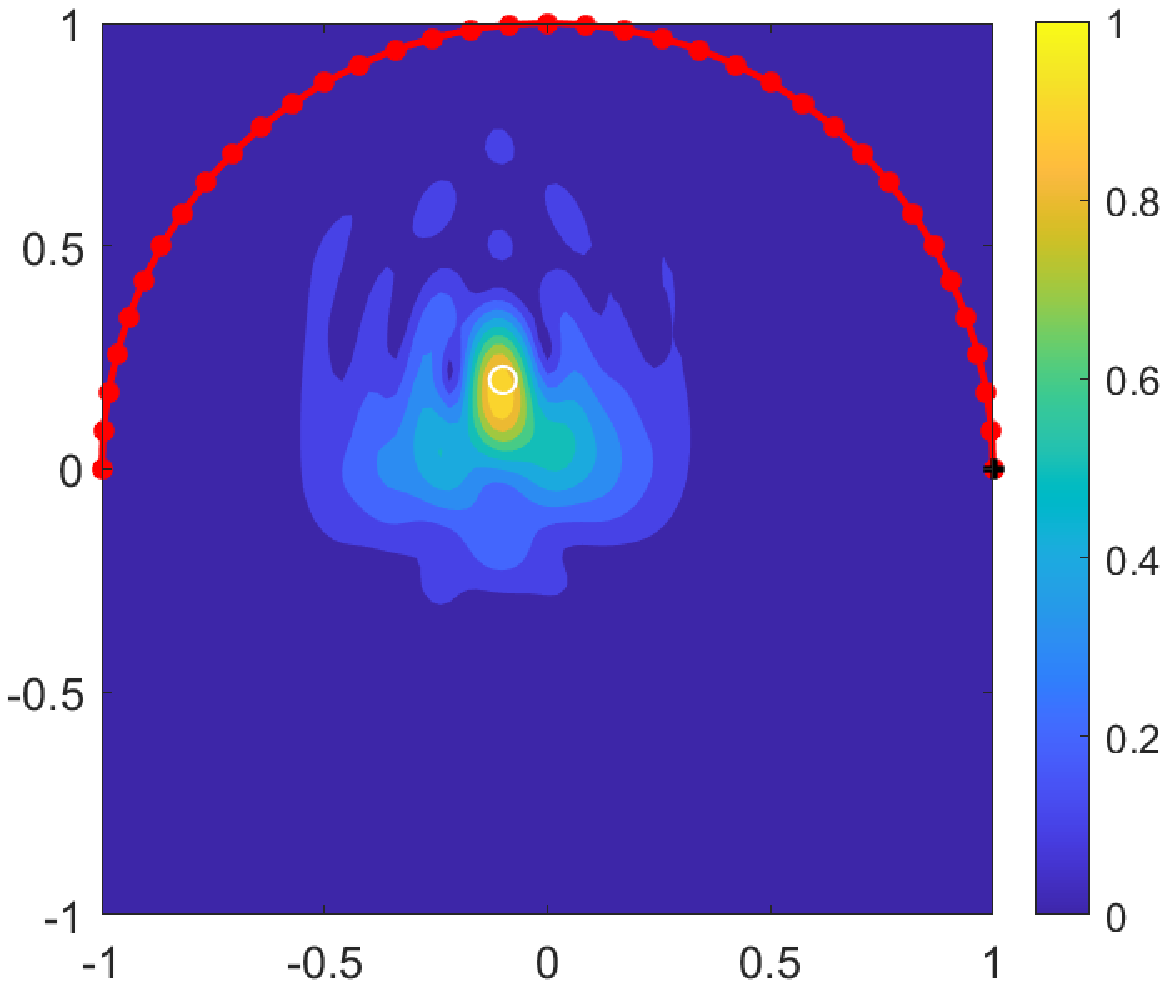}}
	\subfigure[\label{MDSM1-3}$\theta_{1}=0$ and $\theta_{N}=\frac{3}{2}\pi$]{\includegraphics[width=0.32\textwidth]{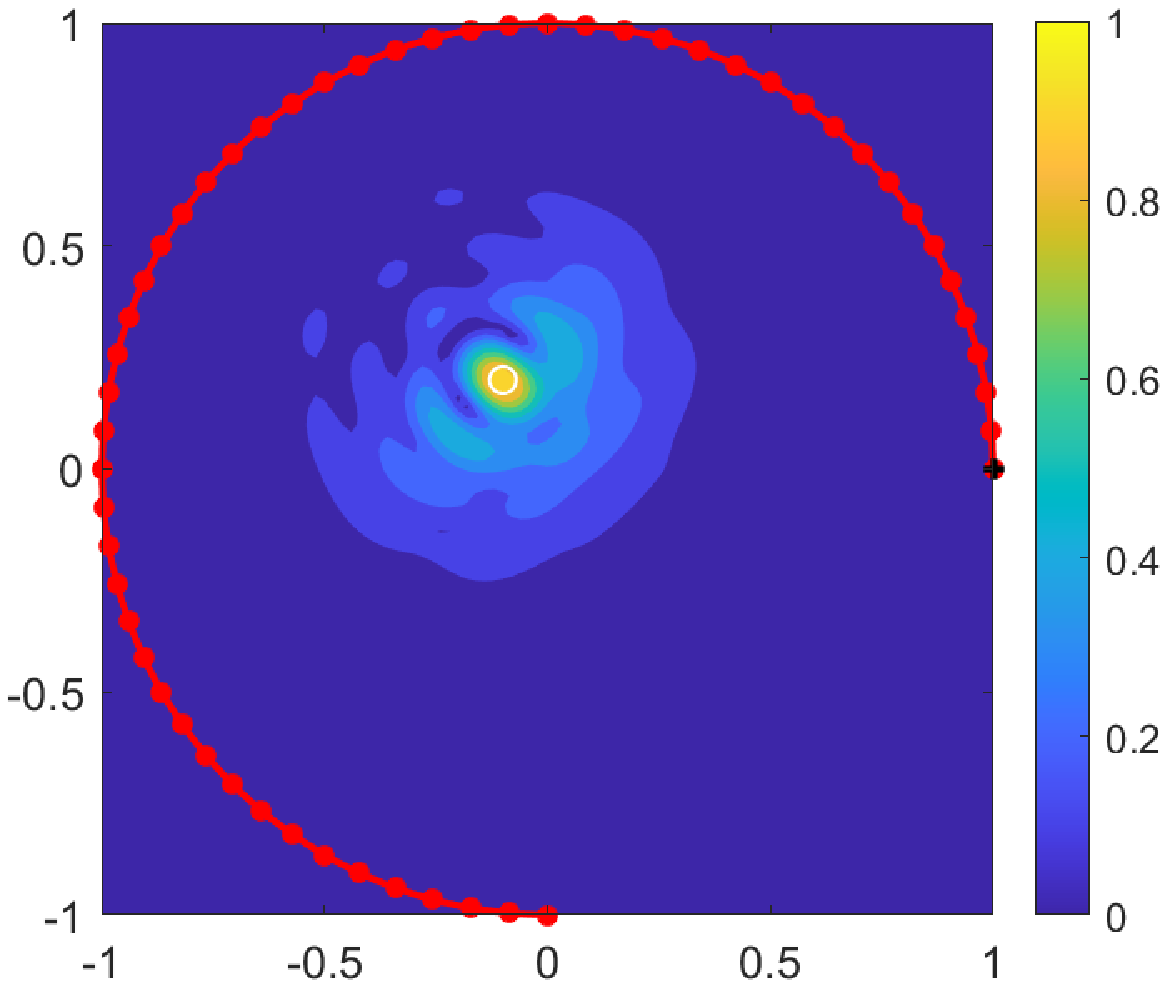}}
	\caption{ DSM (top) and MDSM (bottom) for a single small inhomogeneity. Results are similar to those with the MSM and MDSM, respectively (see Figure \ref{Ex1-MSMResult}).
	}
	\label{Ex1-DSMResult}
	%
	\vskip 1cm
	\centering
	\subfigure[\label{DSM2-1}$\theta_{1}=0$ and $\theta_{N}=\frac{\pi}{2}$]{\includegraphics[width=0.32\textwidth]{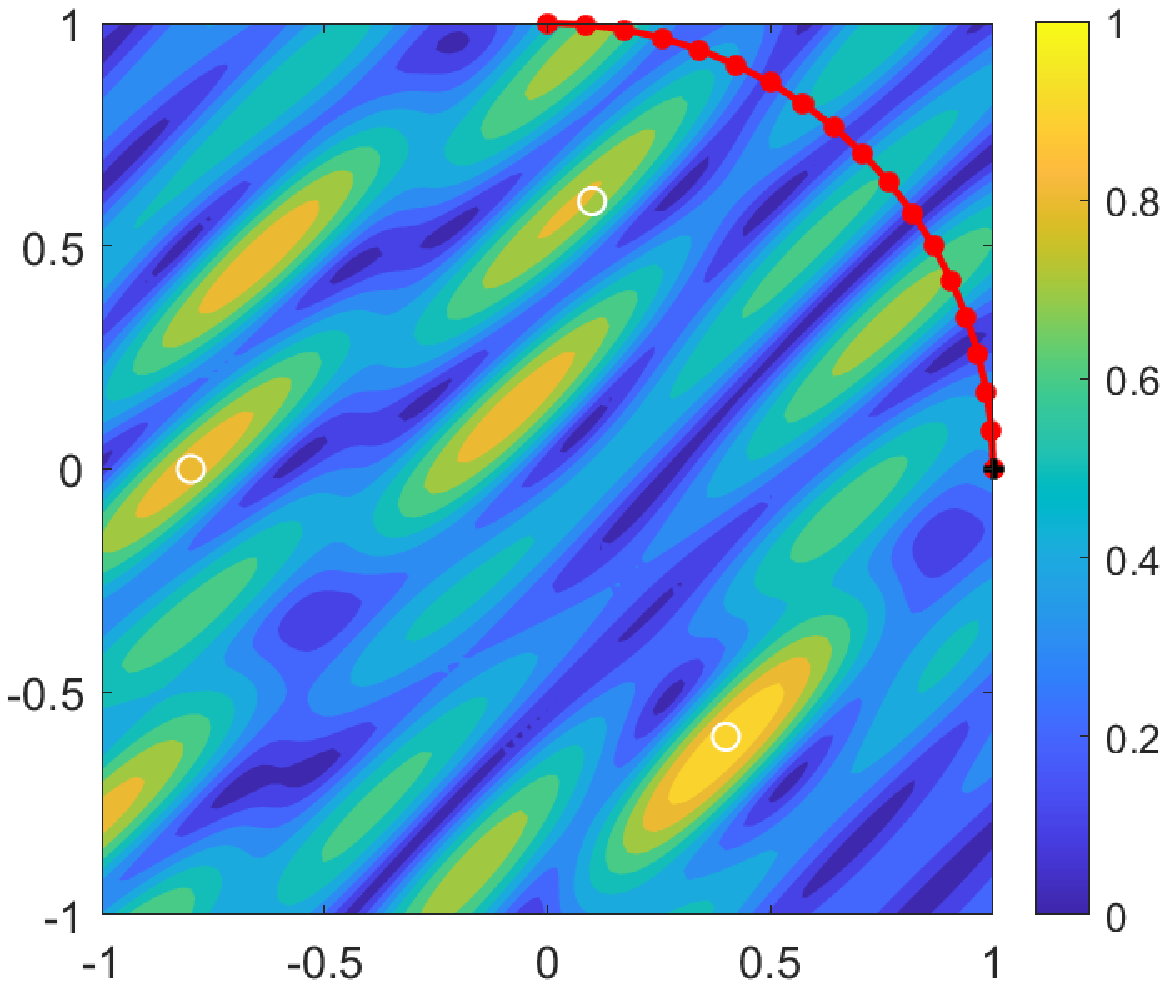}}
	\subfigure[\label{DSM2-2}$\theta_{1}=0$ and $\theta_{N}=\pi$]{\includegraphics[width=0.32\textwidth]{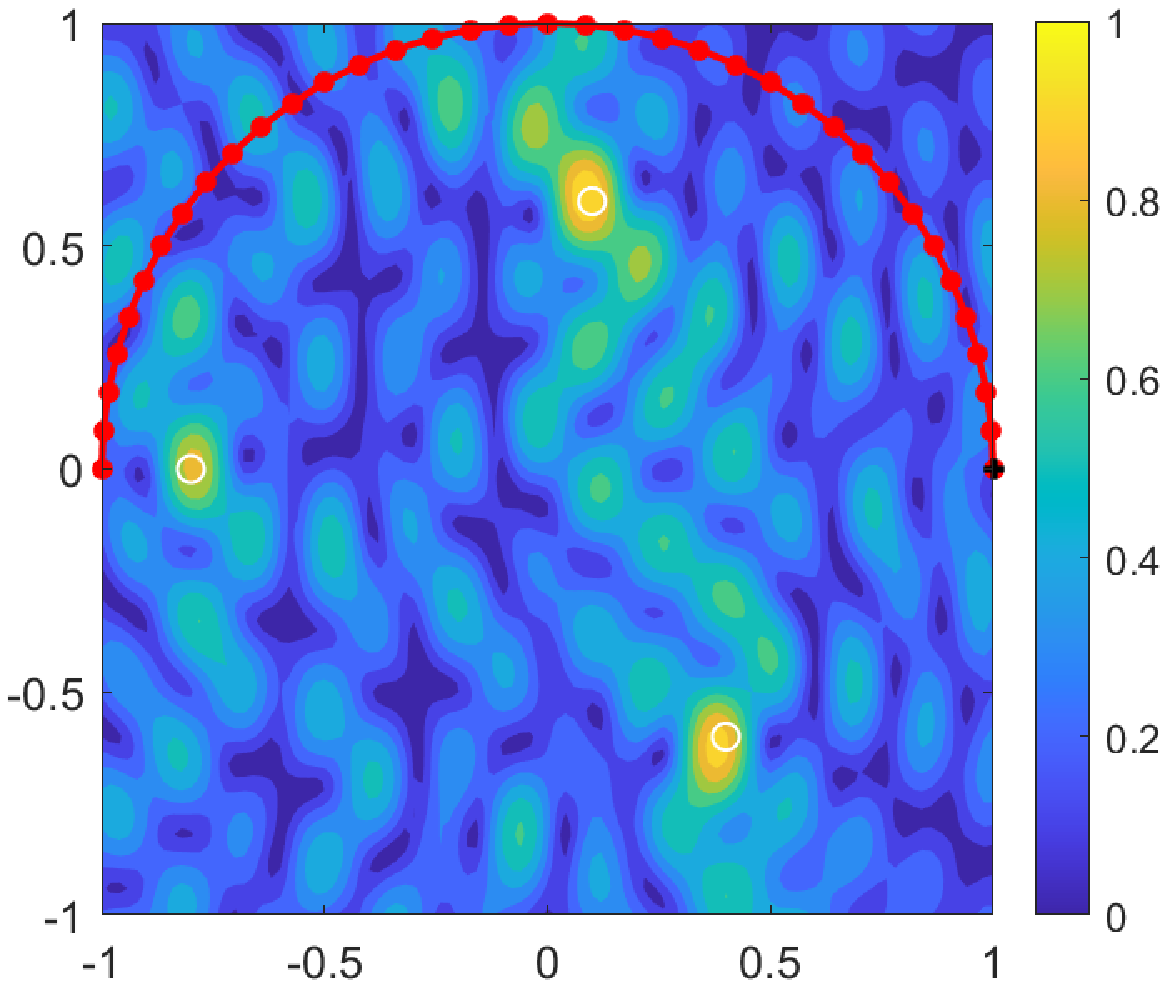}}
	\subfigure[\label{DSM2-3}$\theta_{1}=0$ and $\theta_{N}=\frac{3}{2}\pi$]{\includegraphics[width=0.32\textwidth]{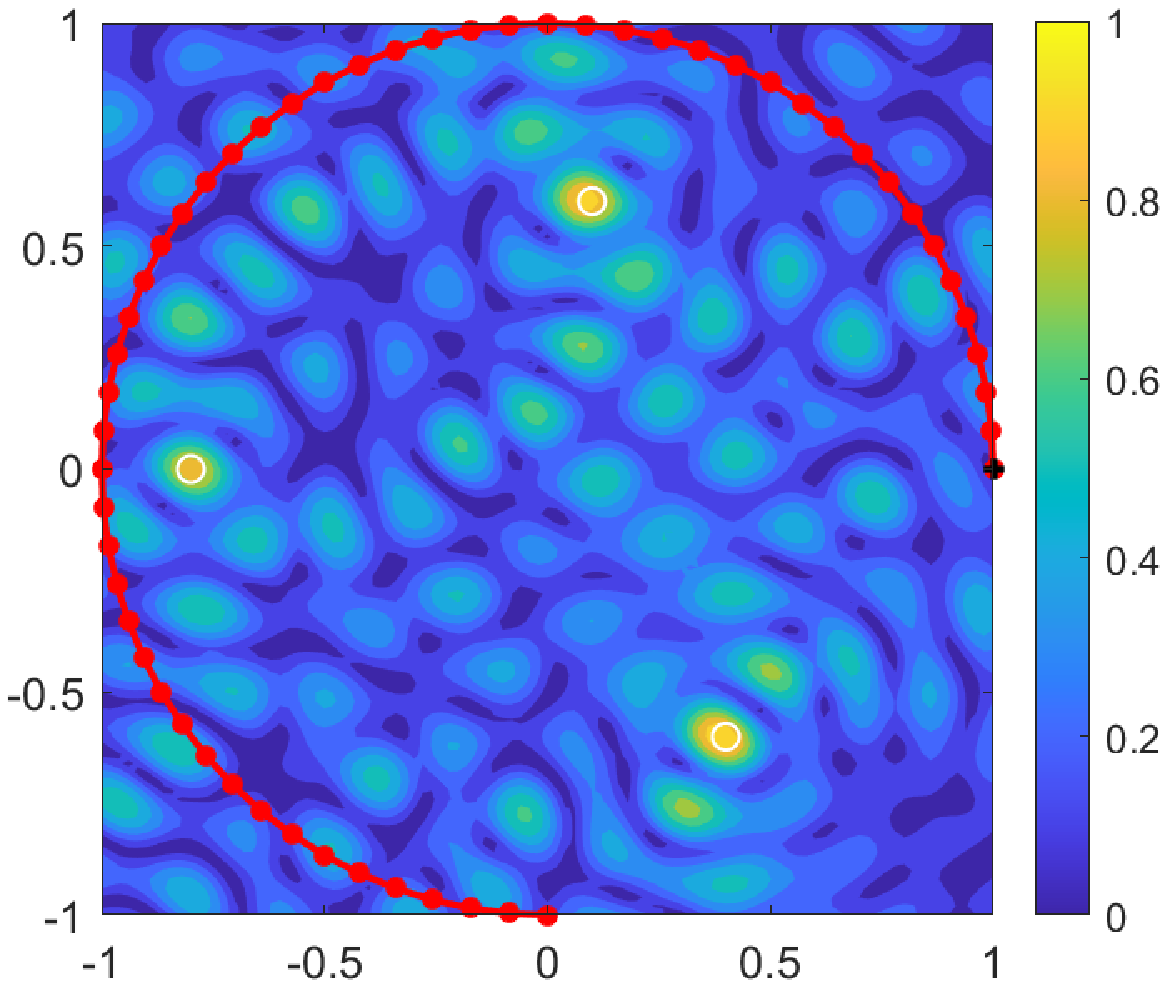}}
	\subfigure[\label{MDSM2-1}$\theta_{1}=0$ and $\theta_{N}=\frac{\pi}{2}$]{\includegraphics[width=0.32\textwidth]{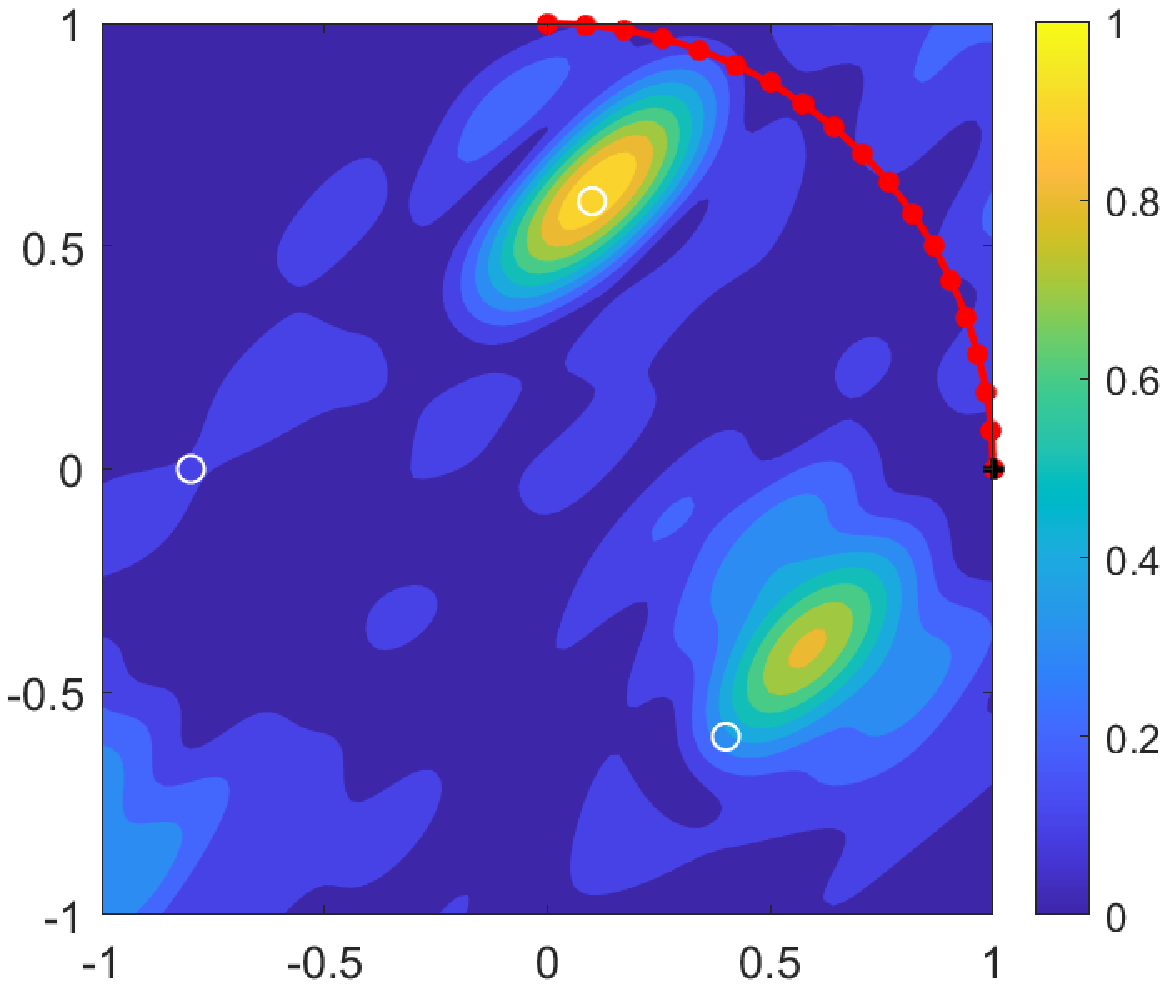}}
	\subfigure[\label{MDSM2-2}$\theta_{1}=0$ and $\theta_{N}=\pi$]{\includegraphics[width=0.32\textwidth]{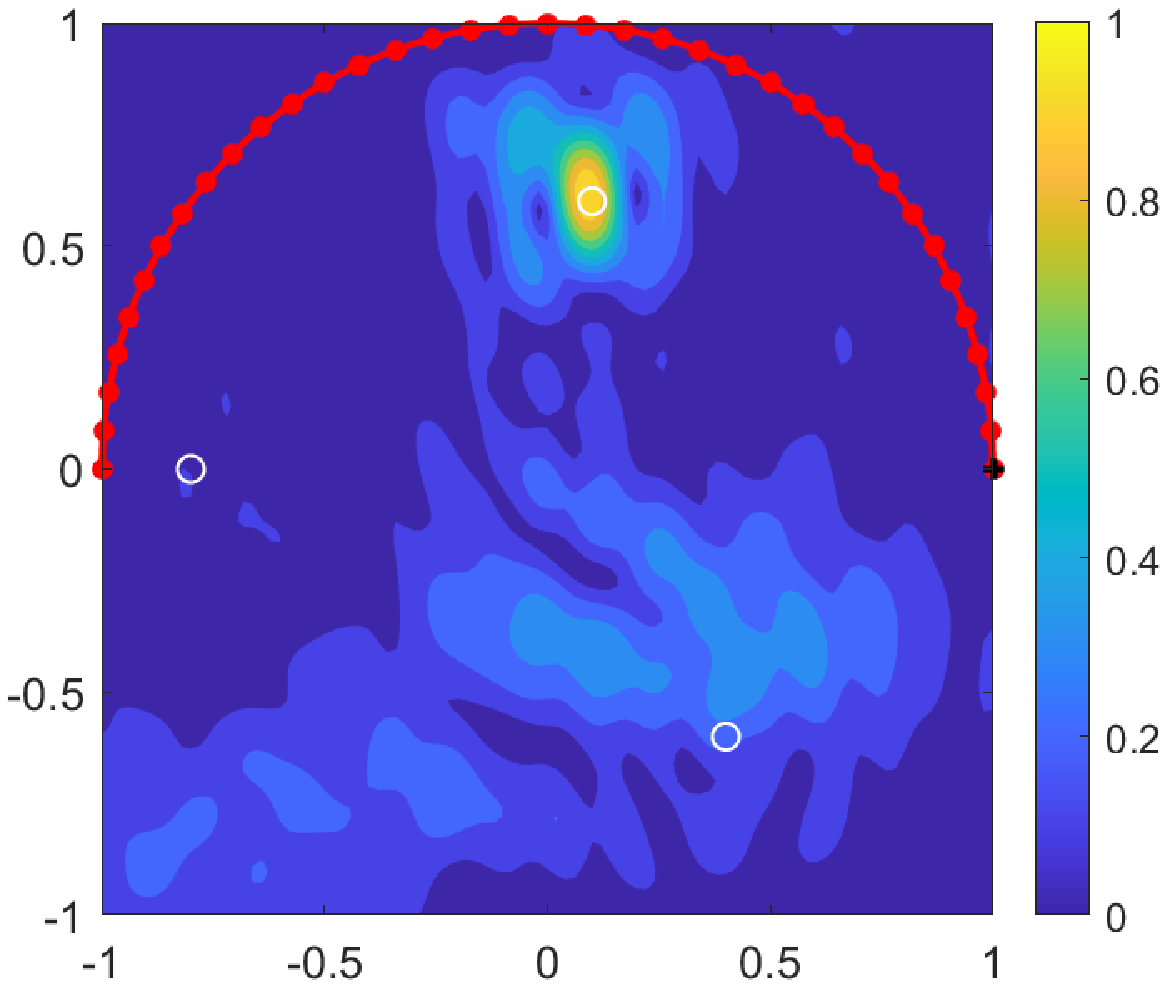}}
	\subfigure[\label{MDSM2-3}$\theta_{1}=0$ and $\theta_{N}=\frac{3}{2}\pi$]{\includegraphics[width=0.32\textwidth]{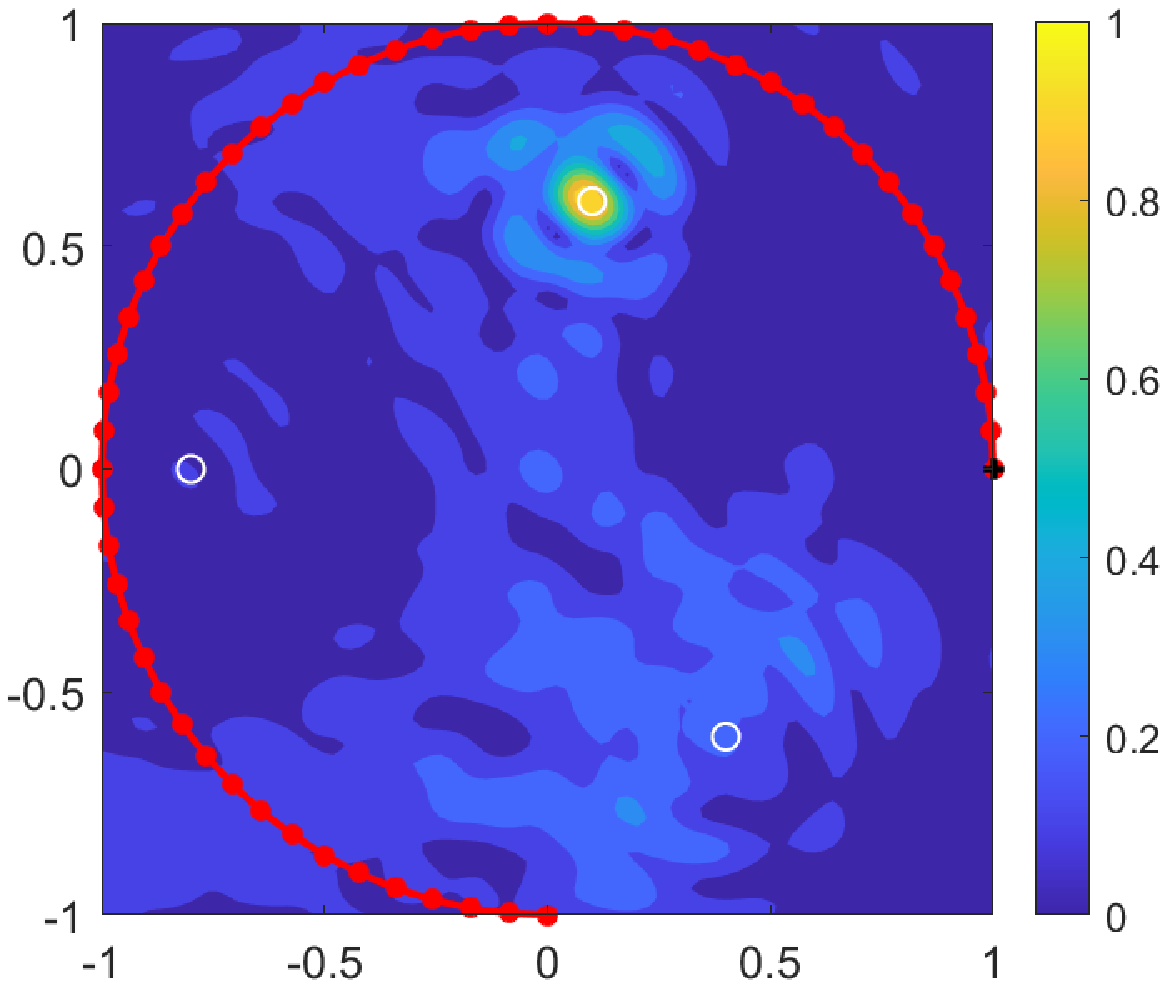}}
	\caption{DSM (top) and MDSM (bottom) for multiple inhomogeneities. The DSM identifies all inhomogeneities with the measurement angle $\geq\pi$, and the MDSM detects some, but not all, inhomogeneities.}
	\label{Ex2-DSMResult}
\end{figure}

\begin{figure}[h!]
	\centering
	\subfigure[\label{DSM3-1}$\theta_{1}=0$ and $\theta_{N}=\frac{\pi}{2}$]{\includegraphics[width=0.32\textwidth]{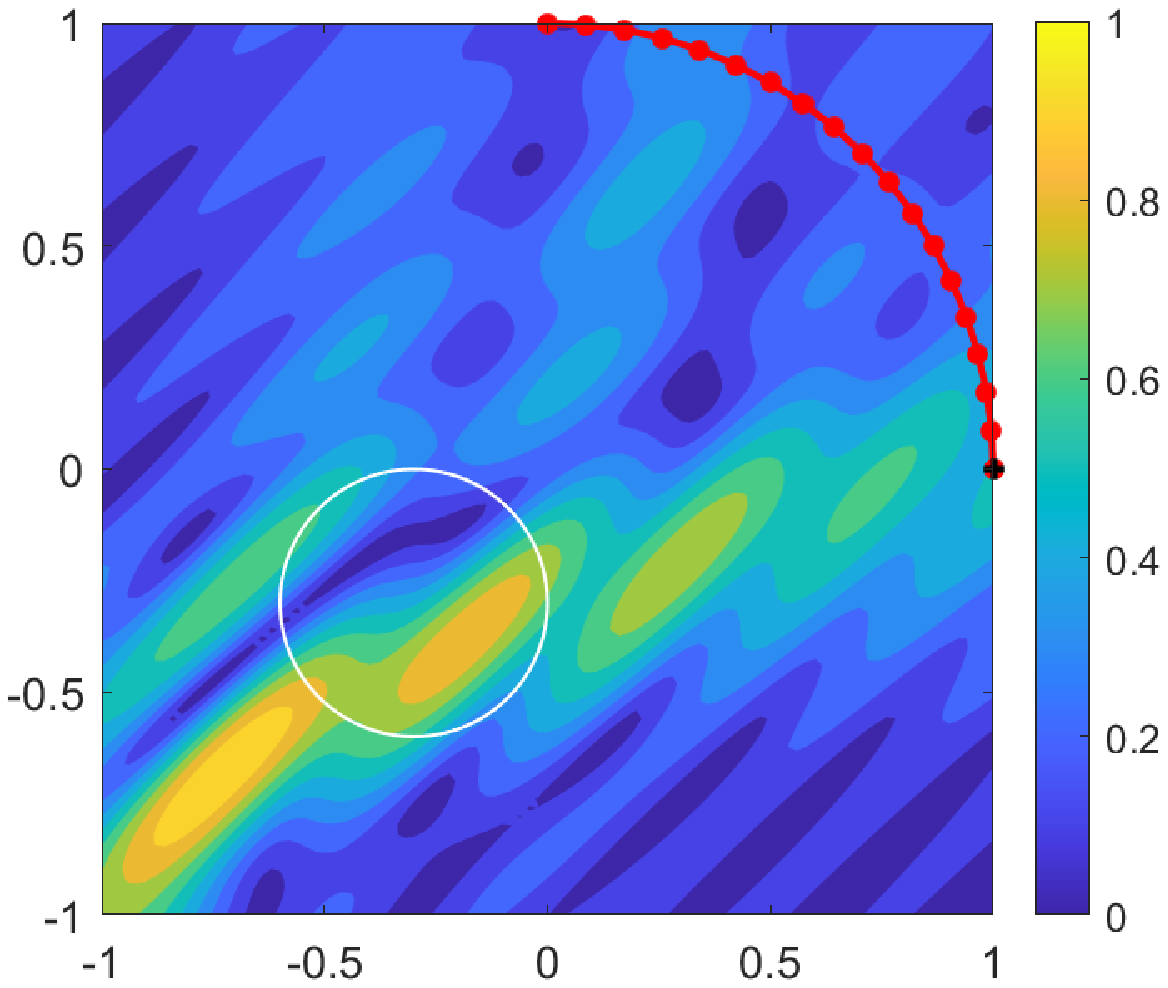}}
	\subfigure[\label{DSM3-2}$\theta_{1}=0$ and $\theta_{N}=\pi$]{\includegraphics[width=0.32\textwidth]{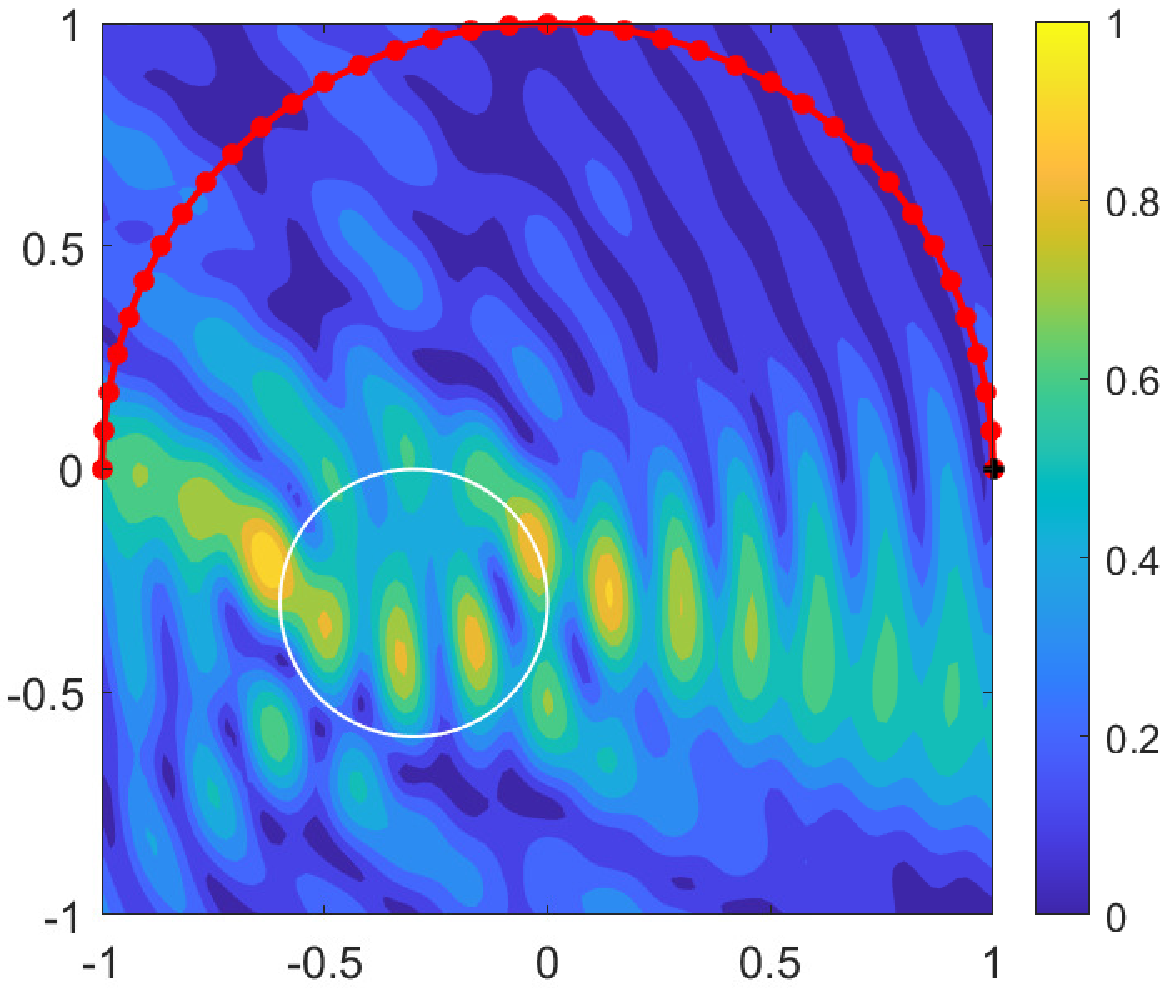}}
	\subfigure[\label{DSM3-3}$\theta_{1}=0$ and $\theta_{N}=\frac{3}{2}\pi$]{\includegraphics[width=0.32\textwidth]{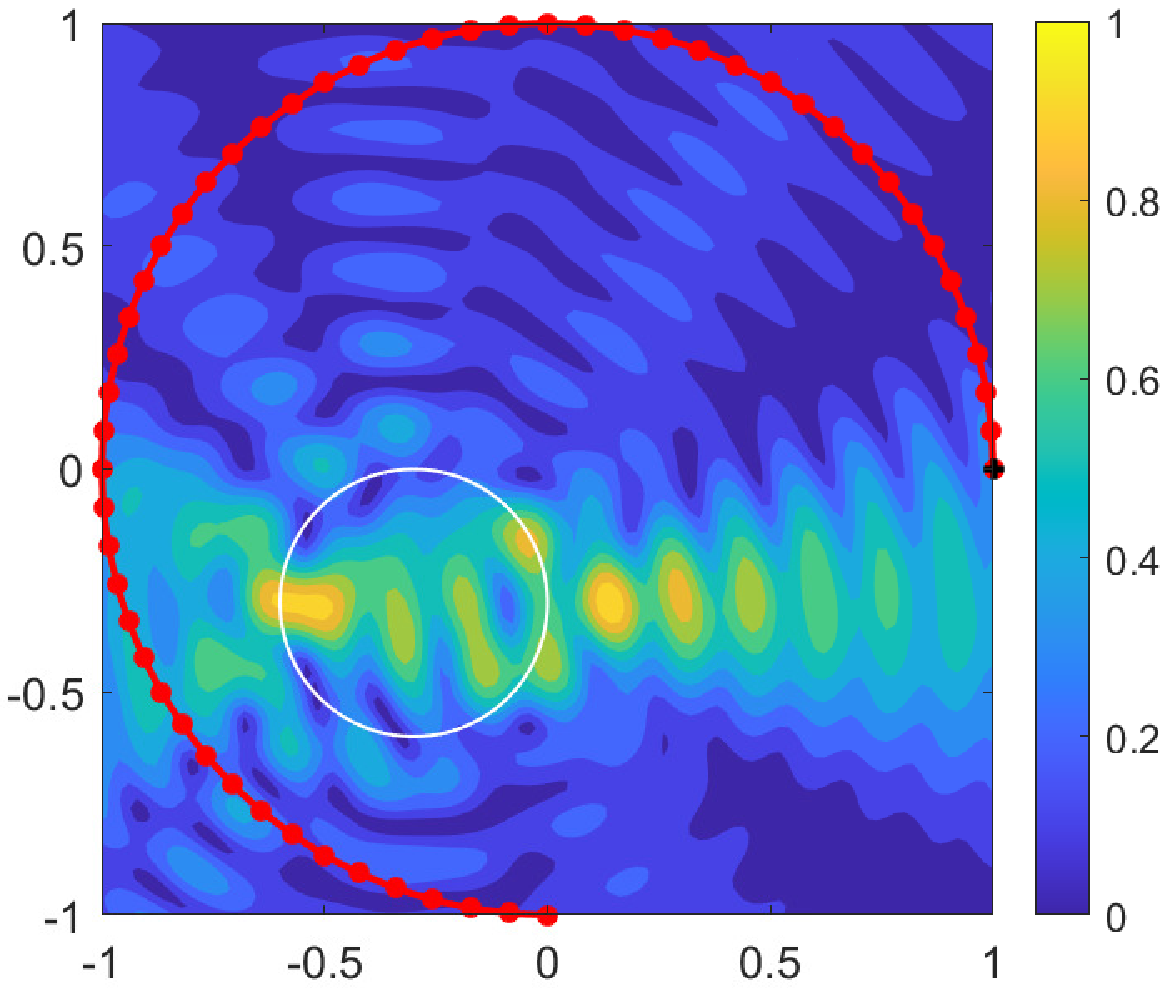}}
	\subfigure[\label{MDSM3-1}$\theta_{1}=0$ and $\theta_{N}=\frac{\pi}{2}$]{\includegraphics[width=0.32\textwidth]{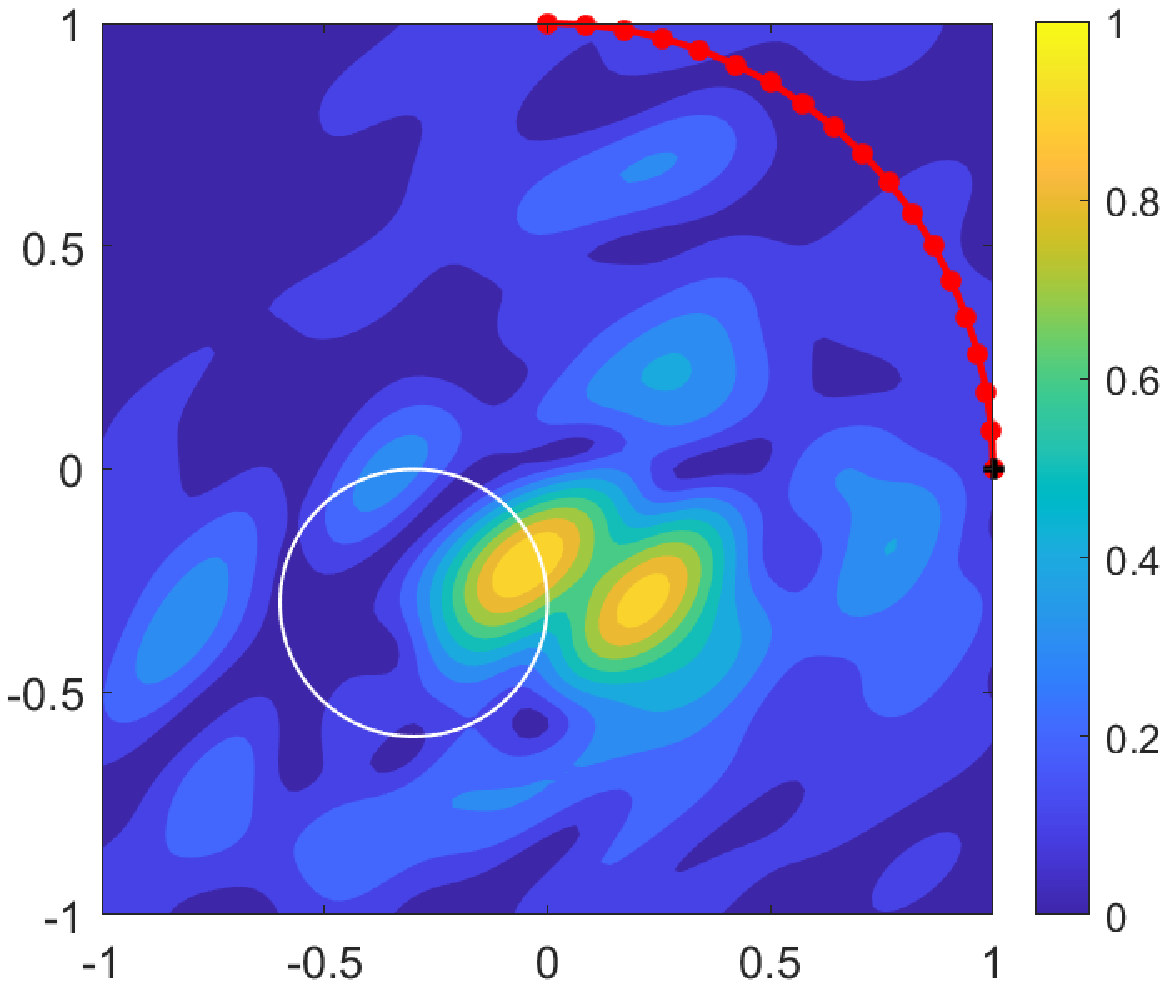}}
	\subfigure[\label{MDSM3-2}$\theta_{1}=0$ and $\theta_{N}=\pi$]{\includegraphics[width=0.32\textwidth]{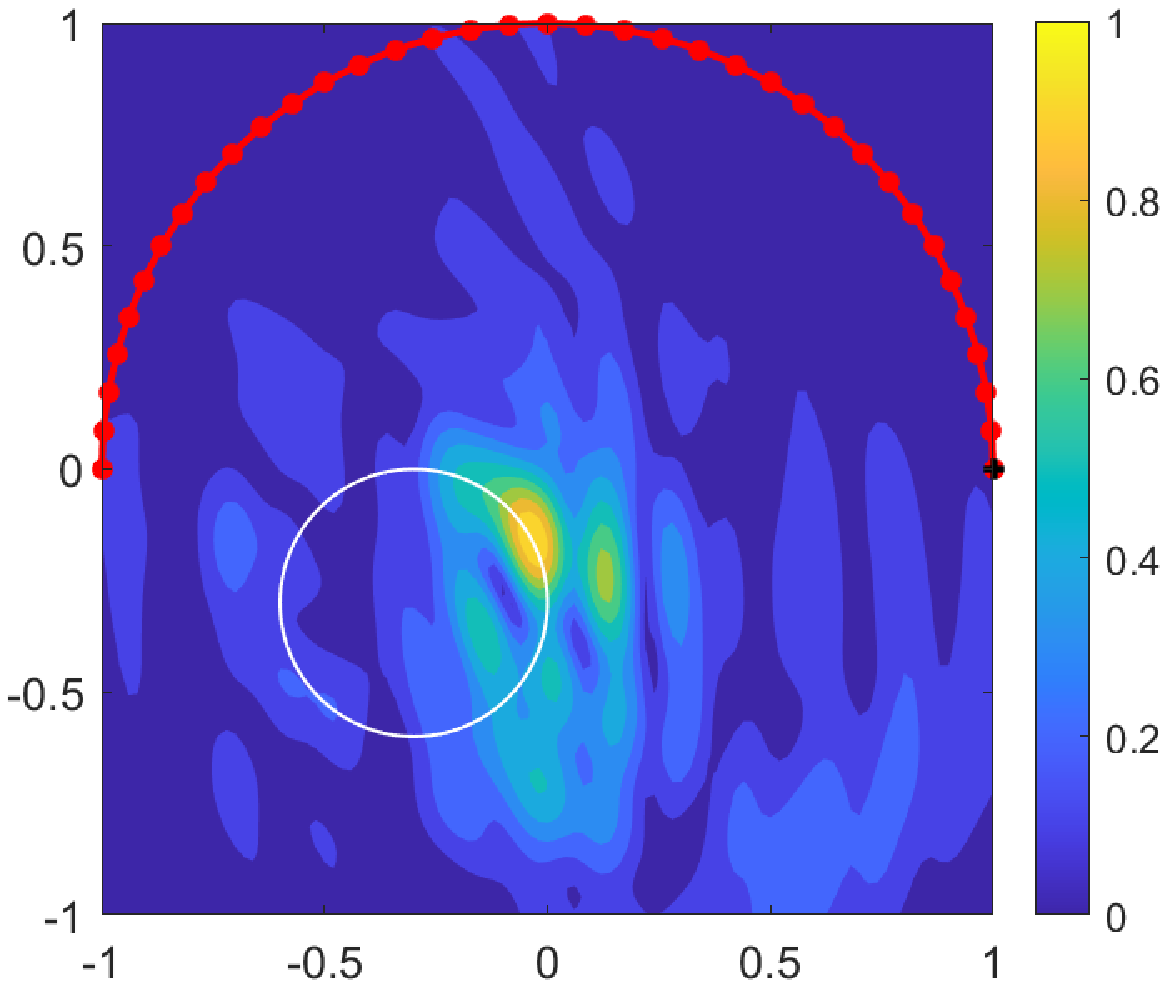}}
	\subfigure[\label{MDSM3-3}$\theta_{1}=0$ and $\theta_{N}=\frac{3}{2}\pi$]{\includegraphics[width=0.32\textwidth]{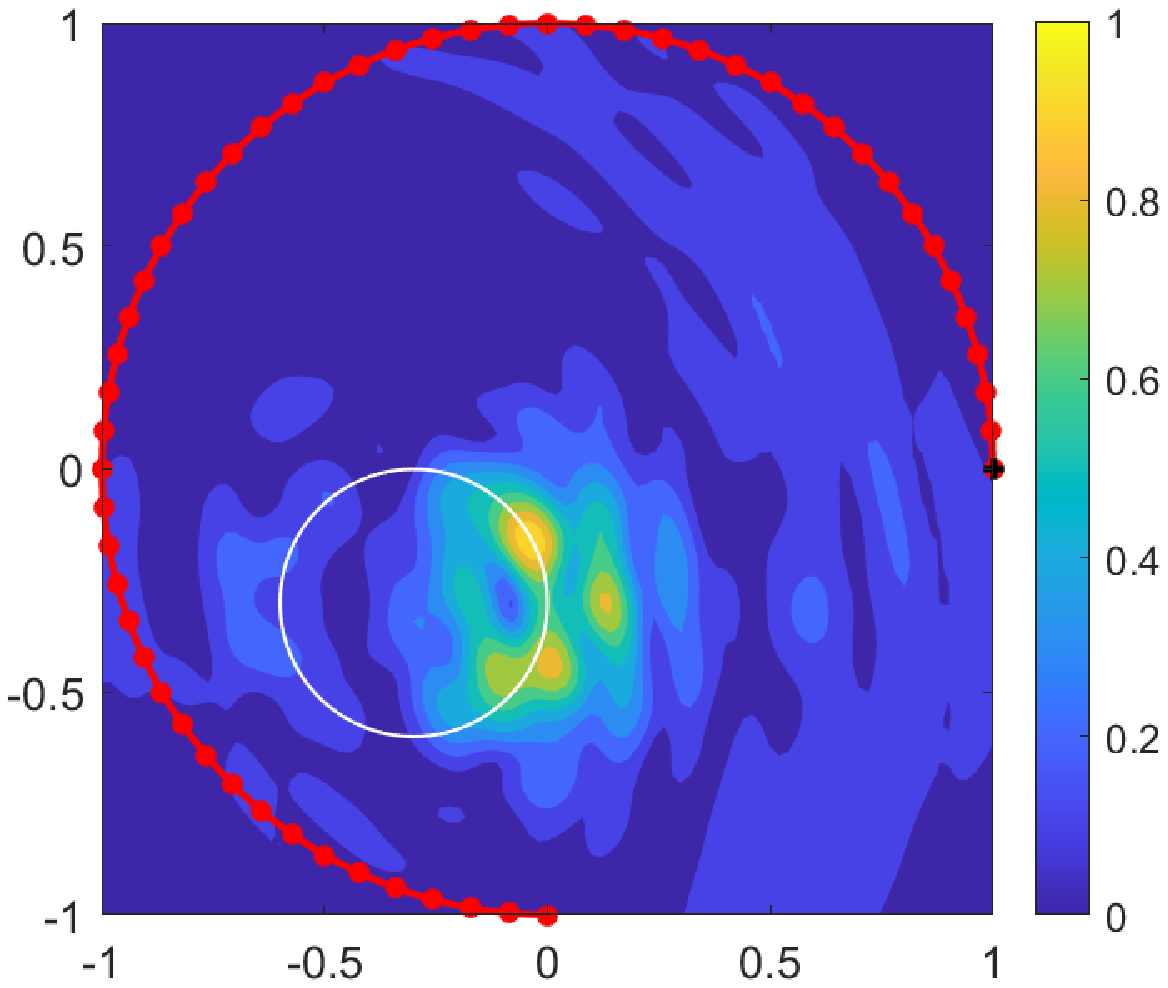}}
	\caption{DSM (top) and MDSM (bottom) for an extended target. The indicator functions of DSM and MDSM does not localize the extended target.}
	\label{Ex3-DSMResult}
\end{figure}


\section{Conclusion}\label{sec:6}
We propose monostatic sampling methods in limited-aperture problems in two dimensions. Thanks to the far-field feature of the scattering problem under the small and well-separated inhomogeneities hypothesis, the asymptotic structures of the proposed sampling methods are identified in terms of the geometric, electromagnetic features of inhomogeneities, the measurement angle of a monostatic system, and the Bessel functions of the first kind and the Struve functions. We compare the obtained asymptotic structures with the results of multistatic DSM.
Numerical simulations validate the theoretical results with noisy synthetic data. 
Interestingly, according to the numerical simulations, for an extended target the proposed imaging scheme with multi-frequency has less accurate imaging performance than the single-frequency one. 
Analysis and improvement of the proposed monostatic sampling methods for extended targets would be of interest in the future study.

	\section*{Acknowledgment}
	This research was supported by Basic Science Research Program through the National Research Foundation of Korea (NRF) funded by the Ministry of Education (NRF-2019R1A6A1A10073887 and NRF-2021R1A2C1011804).


%
	\bibliographystyle{IEEEtran}
	\bibliography{References.bib}

\begin{thebibliography}{10}
\providecommand{\url}[1]{#1}
\csname url@samestyle\endcsname
\providecommand{\newblock}{\relax}
\providecommand{\bibinfo}[2]{#2}
\providecommand{\BIBentrySTDinterwordspacing}{\spaceskip=0pt\relax}
\providecommand{\BIBentryALTinterwordstretchfactor}{4}
\providecommand{\BIBentryALTinterwordspacing}{\spaceskip=\fontdimen2\font plus
\BIBentryALTinterwordstretchfactor\fontdimen3\font minus
  \fontdimen4\font\relax}
\providecommand{\BIBforeignlanguage}[2]{{%
\expandafter\ifx\csname l@#1\endcsname\relax
\typeout{** WARNING: IEEEtran.bst: No hyphenation pattern has been}%
\typeout{** loaded for the language `#1'. Using the pattern for}%
\typeout{** the default language instead.}%
\else
\language=\csname l@#1\endcsname
\fi
#2}}
\providecommand{\BIBdecl}{\relax}
\BIBdecl

\bibitem{2Dsurvey}
D.~R. Luke and R.~Potthast, ``The no response test---a sampling method for
  inverse scattering problems,'' \emph{SIAM J. Math. Anal.}, vol.~63, no.~4,
  pp. 1292--1312, 2003.

\bibitem{survey_sampling}
R.~Potthast, ``A survey on sampling and probe methods for inverse problems,''
  \emph{Inverse Probl.}, vol.~22, no.~2, p.~R1, 2006.

\bibitem{Kang_DSM_mono-static}
S.~{Kang}, M.~{Lambert}, and W.~{Park}, ``Analysis and improvement of direct
  sampling method in the monostatic configuration,'' \emph{IEEE Geosci. Remote.
  Sens. Lett.}, vol.~16, no.~11, pp. 1721--1725, 2019.

\bibitem{Park}
W.-K. Park, ``Multi-frequency subspace migration for imaging of perfectly
  conducting, arc-like cracks in full- and limited-view inverse scattering
  problems,'' \emph{J. Comput. Phys.}, vol. 283, pp. 52--80, 2015.

\bibitem{Joh_MUSIC_CL}
Y.-D. Joh, Y.~M. Kwon, and W.-K. Park, ``{MUSIC}-type imaging of perfectly
  conducting cracks in limited-view inverse scattering problems,'' \emph{Appl.
  Math. Comput.}, vol. 240, pp. 273--280, 2014.

\bibitem{zhang2015ofdm}
T.~Zhang and X.-G. Xia, ``{OFDM} synthetic aperture radar imaging with
  sufficient cyclic prefix,'' \emph{IEEE Trans. Geosci. Remote Sens.}, vol.~53,
  no.~1, pp. 394--404, 2015.

\bibitem{ConstructionAndBuildingMaterials-2016}
X.~Liu, M.~Serhir, and M.~Lambert, ``Detectability of underground electrical
  cables junction with a ground penetrating radar: electromagnetic simulation
  and experimental measurements,'' \emph{Constr. Build. Mater.}, vol. 158, pp.
  1099--1110, 2018.

\bibitem{dsm2d_farfield}
J.~Li and Z.~Zou, ``A direct sampling method for inverse scattering using
  far-field data,'' \emph{Inverse Probl. Imag.}, vol.~7, no.~3, pp. 757--775,
  2013.

\bibitem{Kang_3DDSM}
S.~Kang and M.~Lambert, ``Structure analysis of direct sampling method in 3d
  electromagnetic inverse problem: near- and far-field configuration,''
  \emph{Inverse Prob.}, vol.~37, no.~7, p. 075002, jun 2021.

\bibitem{PARK2018648}
W.-K. Park, ``Direct sampling method for retrieving small perfectly conducting
  cracks,'' \emph{J. Comput. Phys.}, vol. 373, pp. 648 -- 661, 2018.

\bibitem{DSM2D_tomography}
Y.~T. Chow, K.~Ito, and J.~Zou, ``A direct sampling method for electrical
  impedance tomography,'' \emph{Inverse Probl.}, vol.~30, no.~9, p. 095003, aug
  2014.

\bibitem{DSM2D_diffusive_tomography}
Y.~T. Chow, K.~Ito, K.~Liu, and J.~Zou, ``Direct sampling method for diffusive
  optical tomography,'' \emph{SIAM J. Sci. Comput.}, vol.~37, no.~4, pp.
  A1658--A1684, 2015.

\bibitem{bektas2016direct}
H.~O. Bektas and O.~Ozdemir, ``Direct sampling method for monostatic radar
  imaging,'' in \emph{URSI International Symposium on Electromagnetic Theory
  (EMTS)}, 2016, pp. 152--154.

\bibitem{DSM2D_time_heat}
Y.~T. Chow, K.~Ito, and J.~Zou, ``A time-dependent direct sampling method for
  recovering moving potentials in a heat equation,'' \emph{SIAM J. Sci.
  Comput.}, vol.~40, no.~4, pp. A2720--A2748, 2018.

\bibitem{Kang_mfDSM_limited}
S.~{Kang}, M.~{Lambert}, C.~Y. {Ahn}, T.~{Ha}, and W.~{Park}, ``Single- and
  multi-frequency direct sampling methods in a limited-aperture inverse
  scattering problem,'' \emph{IEEE Access}, vol.~8, pp. 121\,637--121\,649,
  2020.

\bibitem{ESM-Bayesian-LimitedAperture}
Z.~Li, Z.~Deng, and J.~Sun, ``Extended-sampling-bayesian method for limited
  aperture inverse scattering problems,'' \emph{SIAM Journal on Imaging
  Sciences}, vol.~13, no.~1, pp. 422--444, 2020.

\bibitem{Ammari1}
H.~Ammari, J., Garnier, W.~Jing, H.~Kang, M.~Lim, K.~S{\o}lna, and H.~Wang,
  \emph{Mathematical and Statistical Methods for Multistatic Imaging}.\hskip
  1em plus 0.5em minus 0.4em\relax Springer, 2010.

\bibitem{Kang_2DDSM}
S.~Kang, M.~Lambert, and W.-K. Park, ``\BIBforeignlanguage{English}{Direct
  sampling method for imaging small dielectric inhomogeneities: analysis and
  improvement},'' \emph{\BIBforeignlanguage{English}{Inverse Prob.}}, vol.~34,
  no.~9, p. 095005 (18pp), 2018.

\bibitem{Colton_Kress1}
D.~Colton and R.~Kress, \emph{Inverse Acoustic and Electromagnetic Scattering
  Theory}, 2nd~ed.\hskip 1em plus 0.5em minus 0.4em\relax Springer, 1998.

\bibitem{integral_bessel}
W.~Rosenheinrich, \emph{Tables of some indefinite integral of bessel
  functions}, 2016.

\bibitem{bessel_book}
M.~Abramowitz, I.~A. Stegun, and R.~H. Romer, \emph{Handbook of mathematical
  functions with formulas, graphs, and mathematical tables}.\hskip 1em plus
  0.5em minus 0.4em\relax American Association of Physics Teachers, 1988.

\end{thebibliography}

\end{document}